\documentclass{amsart}
\usepackage{amsaddr}

\usepackage{ dsfont }
\usepackage{amssymb, amsmath,amsfonts,amsthm,color,bigints, physics}
\usepackage{tikz}
\usepackage[utf8]{inputenc}
\usepackage[T1]{fontenc}
\usepackage{biblatex}

\usepackage{times}
\usepackage{mathtools, subfigure, mdframed}
\usepackage{a4wide}
\usepackage{pifont}
\usepackage{ifthen}
\usepackage[normalem]{ulem}
\usepackage{eso-pic}
\usepackage{enumitem}
\usepackage{units}
\usetikzlibrary{shapes.geometric}
\usetikzlibrary{arrows.meta,arrows}
\usepackage{calligra}
\usepackage{algorithm}
\usepackage{algorithmic}
\newcommand{\frechet}{Fr\'{e}chet }
\usepackage{ upgreek }
\usepackage{mathrsfs}
\usepackage{enumerate}
\usepackage{xspace}
\usepackage{hyperref}  
\usepackage{framed}
\usepackage{bbm}
\usepackage{float}
\usepackage{comment}
\usepackage{textcomp}
\usepackage{url}
\usepackage{srcltx} 
\usepackage{hyperref} 

\usepackage{ wasysym }
\usepackage{graphicx}
\usepackage{tabularx}
\usepackage{array}
\usepackage{setspace}
\usepackage{color}
\usepackage{listings}
\usepackage{url}
\usepackage{wrapfig}
\usepackage{sidecap}
\usepackage{multirow}
\usepackage{tabularx}
\usepackage[intoc]{nomencl}
\usepackage{ifthen}
\usepackage{amsthm}
\usepackage{mathrsfs}
\newcommand{\R}{\mathbb{R}}
\newcommand{\N}{\mathbb{N}}

\newtheorem{theorem}{Theorem}[section]
\newtheorem{lemma}[theorem]{Lemma}

\theoremstyle{definition}
\newtheorem{definition}[theorem]{Definition}

\newtheorem{cor}[theorem]{Corollary}

\newtheorem{assumptions}[theorem]{Assumptions}
\theoremstyle{remark}
\newtheorem{remark}[theorem]{Remark}

\newcommand{\be}{\begin{equation} \label}
\newcommand{\ee}{\end{equation}}
\newcommand{\bea}{\begin{eqnarray}\label}
\newcommand{\eea}{\end{eqnarray}}
\newcommand{\lbal}{\left\{ \begin{array}{l}}
\newcommand{\lball}{\left\{ \begin{array}{ll}}
\newcommand{\ear}{\end{array} \right.}

\expandafter\newcommand\csname peeps\endcsname{p_{1,\varepsilon}}
\expandafter\newcommand\csname pzeps\endcsname{p_{2,\varepsilon}}
\expandafter\newcommand\csname pdeps\endcsname{p_{3,\varepsilon}}

\newcommand{\peps}{\rho_\varepsilon}
\newcommand{\Geps}{\Gamma_{\varepsilon}}

\newcommand{\beps}{b_{\varepsilon}}
\newcommand{\keps}{k_{\varepsilon}}

\newcommand{\eps}{\varepsilon}
\newcommand{\io}{\int_\Omega}

\usepackage{comment}
\newcommand{\vp}{\varphi}
\newcommand{\nn}{\nonumber}
\newcommand{\bom}{\overline\Omega}
\newcommand{\Om}{\Omega}
\newcommand{\tme}{T_{max,\eps}}
\numberwithin{equation}{section}
\newcommand{\dt}{\frac{d}{dt}}

\definecolor{bluegray}{rgb}{0.4, 0.6, 0.8}

\begin{document}
\allowdisplaybreaks
\title{Existence of large-data solutions to a thermo-piezoelectric system and forward operator analysis for associated inverse problems}

\author{Torben J. Fricke}
\address{Universit\"at Paderborn, Institut f\"ur Mathematik\\ 33098 Paderborn, Germany\\ tjfricke@math.uni-paderborn.de}
\author{Raphael Kuess}
\address{Humboldt-Universit\"at zu Berlin, Institut f\"ur Mathematik\\ 10099 Berlin, Germany\\ raphael.kuess@hu-berlin.de}
\author{Felix Meyer}
\address{Universit\"at Paderborn, Institut f\"ur Mathematik\\ 33098 Paderborn, Germany\\ felix.meyer@math.uni-paderborn.de}
\date{\today}

\maketitle

\begin{abstract}
We consider an inverse problem governed by the initial-boundary value problem for the thermo-piezoelectric Kelvin-Voigt damped dynamical system 
\begin{align*}\left\{ \begin{array}{l} \rho(z,t) u_{tt}- \frac{\dd{}}{\dd{z}}\left(\Gamma(\Theta) u_{zt} +p_1 u_z+p_2(z,t)\phi_z^0 +p_2(z,t)\chi_z- \beta \Theta\right)=0,\vspace{1mm}\\
-\frac{\dd{}}{\dd{z}}(p_2(z,t)u_z-p_3(z,t)\phi_z^0-p_3(z,t)\chi_z)=0,\vspace{1mm}\\
b(z,t) \Theta_t-\frac{\dd{}}{\dd{z}}\left(k(z,t)\Theta_z\right) - \Gamma(\Theta) u_{zt}^2+\beta \Theta u_{zt}=0, \end{array} \right.
 \end{align*}
in an open bounded interval $\Omega\subset\mathbb{R}$, for the evolution of the displacement variable $u$, the electric potential $\phi^0$ and the temperature $\Theta\geq 0$, where $\chi$ is a given Dirichlet lift function. Assuming that
the material coefficients $\beta, p_1 \in \R^+$ and the 
the material functions $\rho$, $\Gamma$, $p_2$, $p_3$, $b$  and $k$ are strictly positive and bounded, 
a global-in-time existence result is established for weak solutions. The present manuscript demonstrates that this can be achieved under energy- and entropy-minimal assumptions, in the sense that global weak solutions are shown to exist for any initial data 
$$u_0\in W^{1,2}(\Omega)\text{ with }u_0|_{\partial\Om}\in\mathbb{R},\quad u_{0t}\in L^2(\Omega)\text{ with }u_{0t}|_{\partial\Om}\in\mathbb{R}\quad\text{and}\quad 0\le\Theta_0\in L^2(\Omega).$$
The qualitative analysis of the evolution problem then allows to model and analyze the structural properties of the corresponding 
forward operator that naturally arises in inverse parameter identification settings. 
Therein, two modeling approaches of the observation operator as approximations of the electrical surface charge are presented and results on their well-definedness and boundedness are established. 
With the results on well-definedness and boundedness of the model operator, established in this paper as well, results on well-definedness, boundedness and continuous Fr\'{e}chet differentiability of the forward operator are presented.\\



\end{abstract}
\pagestyle{plain}

\newpage
\section{Introduction}
\label{sec:intro}

\noindent Inverse parameter identification problems are usually comprised of a model and additional observations and aims at identifying a parameter function $ f $, which appears in the underlying model. Therefore, the model contains the searched for parameter $ f $ and the state, respectively the solution to the PDE $ l,  $ which can be modeled as 
 \begin{equation}\label{sys1}
  A(f, l)=0 .
 \end{equation}
 Frequently, $A$ can be referred to as a differential operator.
In order to recover information on the parameter $ f $, 
we have given noisy observations $y^\delta$, where we 
assume that they obey the deterministic and known noise level $ \delta>0 $\[ \norm{y-y^{\delta}} \leq \delta .\] 
We denote the noiseless modeled observation data with $ y,  $ and the observation operator with $C$, i.e.,
\begin{equation}\label{sys2}
C(f,l)=y.
\end{equation}  
A classical approach to formulate and solve inverse parameter identification problems is the reduced approach, where the model is eliminated by introducing a so-called parameter-to-state map 
$S$,
which maps an arbitrary fixed parameter $ f $ to the corresponding state $l$. This approach needs bijectivity of the model operator $A$, i.e., existence and uniqueness of solutions to the underlying PDE. 
Alternatively, one can consider the all-at-once approach, where a larger system, composed of $A$ and the $C$ simultaneously defines the forward operator. This yields that the observations and the underlying model as a system for $ (f, l) $, where two infinite dimensional variables $ f $ and $ l $ are to be determined, is used. Consequently, the forward operator reads as  \[ {F}(f, l)=\left(\begin{array}{c}A(f, l)\\C(f,l)\end{array}\right)=\left(\begin{array}{c}0\\y\end{array}\right)=\textbf{y},  \] where $ \textbf{y} $ contains the right hand side of the model and the given data. This approach only needs the existence of solutions to the underlying PDE model. 
Such inverse problems are typically ill-posed, i.e.  $ F $ is not continuously invertible and our given measurements $ y^{\delta} $ are noisy.

\noindent In this paper the model operator $A$ describes the
one-dimensional piezoelectric model posed on a bounded open interval $\Omega\subset\mathbb{R}$. Throughout this paper, we study the coupled dynamical system
\be{sys0}
\lball
\rho(z,t) u_{tt}-\frac{\dd{}}{\dd{z}} \left(\Gamma(\Theta) u_{zt}
+p(z,t) u_z\right)= -\frac{\dd}{\dd{z}}(\beta \Theta),&\quad (z,t)\in\Omega\times (0, T)\\
b(z,t) \Theta_t-\frac{\dd{}}{\dd{z}}\left(k(z,t)\Theta_z\right) - \Gamma(\Theta) u_{zt}^2+\beta \Theta u_{zt}=0,&\quad (z,t)\in\Omega\times (0, T)\\ 
p_1u_z +\Gamma u_{zt}
+ p_2  \phi_z^0=-p_2  \chi_z,&\quad\text{on }~\partial\Omega\times (0, T)\\
k\Theta_z =0,&\quad\text{on }~\partial\Omega\times (0, T)\\
u(t=0)=u_0,\quad u_t(t=0)=u_1,\quad \Theta(t=0)=\Theta_0,&\quad\text{in }~\Omega,
    \ear
\ee
where $\rho,\ \Gamma, \ p,\ b$ and $k$ are prescribed positive functions, $\beta$ is a positive real parameter and $u_0$, $u_1$ and $\Theta_0$ are given suitably regular, satisfying $u_0|_{\partial\Om}\in\mathbb{R}$, $u_{1}|_{\partial\Om}\in\mathbb{R}$ and $\Theta_0\ge 0$. The system under consideration couples the mechanical displacement $u(z,t)$, the electric potential $\phi(z,t)$, and the temperature $\Theta(z,t)$.

\subsection{Related Work}

\noindent The existence and well-posedness of piezoelectric and thermoelastic PDE systems have been the subject of extensive study in the literature. A number of studies have been conducted on one-dimensional systems that are closely related to the model under consideration in this paper.

\noindent For instance, \cite{Winklerweak1d} analyses a system with constant parameters $\rho=b=k=1$ and constant $p$, establishing the existence of global weak solutions. This corresponds closely to the results obtained in Theorem \ref{EX} of the present work.
Most prior work on thermoelastic waves (e.g. \cite{Winklerweak1d}, \cite{winkler2}, \cite{Fricke}, \cite{Meyer}) assumes constant $\rho$, $k$, and $b$. Our model allows space- and time-dependent coefficients, which more accurately represents realistic materials.
We therefore adapt the strategy of \cite{Winklerweak1d} to our circumstances, whereby the presence of coefficients that depend on both space and time requires several non-trivial modifications due to the increased generality and nonlinear dependence of the coefficients. First, space-time dependent coefficients require careful test-function arguments and complicate the limit passage in the approximate-system (Section \ref{sec5}, Lemma  \ref{lem_sqrtL2}.
\noindent There are more one-dimensional studies, which can be found in \cite{winkler2} and \cite{Fricke}. In these works, both $\Gamma$ and the initial data are assumed to be smooth enough. This allows us to analyze the problem more regularly and to show that there are global classical solutions. \cite{Fricke} also includes a connection to the electric field through a term $p=\Gamma+a$. This gives results that are similar to \cite{winkler2}. However, \cite{Fricke} requires $\Gamma$ to be positive, no boundedness condition is needed.
 Another closely related study is \cite{Meyer}, which assumes that $\rho$, $b$ and $k$ are constant parameters, and establishes existence of global classical solutions under the additional restriction that $\Gamma$ is bounded with $\Gamma''(\xi)\le 0$ for $\xi\ge 0$.

\noindent In \cite{mielke} , \cite{owc}, \cite{rossi}, \cite{rubi}, \cite{yosh} and \cite{zimmer}, higher-dimensional extensions were considered. In these cases, the analysis focused primarily on weakened concepts of solvability, consistent with the utilization of global weak solutions as employed in this study.
It is also noteworthy that global large-data solutions to purely thermoelastic systems, excluding Kelvin-Voigt damping, have been established in foundational studies. In particular, the system was examined on a bounded interval with Dirichlet boundary conditions for the displacement field $u$ and Neumann-type conditions for the temperature field $\Theta$ in \cite{Bies1}. This study established the existence and uniqueness of global classical solutions under suitable regularity assumptions on the initial data. Moreover, an analysis of the long-term behavior of three solutions was conducted in \cite{Bies2}, thereby demonstrating their propensity to stabilize towards equilibrium states.

\noindent Building on these analytical foundations, a separate line of research has focused on inverse and control problems for piezoelectric systems. 
These problems have been investigated extensively in the literature, including \cite{drz047}, \cite{978-3-030-56356-1}, \cite{RKDWAW}, \cite{19}, \cite{PDE_constrained}, \cite{18}, and \cite{14}, among others.
 In \cite{drz047}, an optimal boundary control problem for the electrical flux is studied, and existence and uniqueness are established for solutions of the undamped piezoelectric PDE and its adjoint equation, assuming material parameters in $L^\infty(\Omega)$.
 In \cite{14} boundary control problems are considered as well, providing analytical results for the undamped homogeneous piezoelectric system with an elasticity tensor in $C^2(\Omega)$, a permittivity tensor in $L^\infty(\Omega)$, and a constant piezoelectric coupling parameter.
 In \cite{978-3-030-56356-1}, the piezoelectric PDE is coupled with a parabolic temperature equation and an elliptic magnetic-field equation, analogous to the electrical equation of the classical piezoelectric model. The authors prove existence and uniqueness for systems with coefficients of regularity $C^{0,1}(\Omega)$ or $L^\infty(\Omega)$.
 In \cite{PDE_constrained} and \cite{19} the authors examine shape optimization problems and establish existence and uniqueness for solutions of the undamped inhomogeneous piezoelectric PDE with time- and space-constant parameters, along with the corresponding adjoint equations.

\noindent Finally, \cite{RKDWAW} proves well-posedness to a piezoelectric dynamical system
governing mechanical displacement and electrical potential.
with matrix-valued Sobolev–Bochner material parameters damping parameters and inhomogeneities.
This system of coupled hyperbolic-elliptic partial differential equations is further analyzed regarding higher order regularity results, including an a-priori energy estimate. The forward operator is shown to be well-defined and Fréchet differentiable, leading to the formulation of the inverse problem as a minimization problem. Finally,  weak lower semi-continuity, first-order optimality conditions, and the analysis of the adjoint system are discussed.

\subsection{Contribution}

\noindent  In Section \ref{sec:mod}, we transform the wave model into a parabolic system via the standard substitution $v:=u_t$.
It is then necessary to introduce a regularized system in which all variable coefficients are replaced by smooth approximations. This renders the model sufficiently regular to guarantee global classical solvability, which provides the starting point for the subsequent compactness arguments. Section \ref{sec4} establishes $\eps$-independent estimates via nonlinear integral inequalities that improve the temperature's $L^q$ integrability for $q\in(1,3)$ using Gagliardo-Nirenberg interpolation. This is essential for handling the nonlinearity in couplings.
While the weak solution property for candidates $(u,\Theta)$ arising from the limit process can be proven directly for the first equation, this is more difficult for the second equation. A central analytical difficulty lies in passing to the limit in the term $\sqrt{\Gamma_\eps(\Theta_\eps)}u_{\eps zt}$ in $L^2(\Om\times(0,\infty)).$
Our main argument for this is based on the sub-continuity of the $L^2$-norm, for which we exploit the already proven weak solution property in the first equation by choosing suitable test functions. 
\noindent Utilizing these ingredients, we are able to establish the limit in the approximate system and thereby derive a global weak solution to the original thermo-piezoelectric model, as outlined in Theorem \ref{EX}.

\noindent Utilizing the existence of weak solutions to the thermo–piezoelectric system introduced in Section \ref{sec:mod} yields surjectivity of the model operator $A$. 
Consequently, the second contribution of this study is analysis of the model operator $A$, demonstrating its boundedness and well-definedness. 
Since surface charges are usually modeled via boundary integrals of the derivative of the state, two possible observation operators $C$ and $C^\gamma$, for $\gamma>0$ suitably small, are introduced to approximate the surface charge and overcome well-posedness issues. 
For these observations operators boundedness and well-definedness results are presented.
Due to the two different observation operators, we investigate two different forward operators as well.  
Finally, two inverse problems of identifying the material parameters are modeled via the all-at-once approach and its forward operators are proven to be bounded, well-defined and continuous \frechet differentiable, which can be found in Section \ref{sec:OpAna}.

\section{Modeling}\label{sec:mod}

\noindent Suppose the electric excitation of a transversely isotropic piezoceramic along the polarization direction. In this setting we assume that the material covers the $x-y$ plane fully (infinitely) and has thickness $h$.
Hence, we only consider the thickness ($z-$) direction, meaning that we operate on the domain $\Omega=[0, h] \subset \R$. 
Furthermore, we incorporate the generation of heat by acoustic waves and mechanical losses according to the Kelvin-Voigt damping model.
Consequently, we take the interaction of the one-dimensional mechanical displacement $u(t,z)$, the one-dimensional electrical potential $\phi(t,z)$ and the one-dimensional temperature $\Theta(t,z)$ into account. 
These mechanical, electrical and thermal processes are coupled can be described by the following thermo-piezoelectric dynamical system 
\be{mixedDir}
\lball
    \rho u_{tt}- \frac{\dd{}}{\dd{z}}\left( c^E\left( u_z +\tau u_{zt}\right)
+ e \phi_z\right)=-\frac{\dd}{\dd{z}}(\beta \Theta) \quad~&\text{ in }~\Omega\times (0, T) \\[1mm]
-\frac{\dd{}}{\dd{z}}\left(eu_{z}-\varepsilon^S\phi_{z}\right)=0\quad~&\text{ in }~\Omega\times (0, T) \\[1mm]
c_{th}\rho \Theta_t-\frac{\dd{}}{\dd{z}}\left(k\Theta_z\right) - \tau c^E (u_{zt})^2+\beta \Theta u_{zt}=0\quad~&\text{ in }~\Omega\times (0, T) \\[1mm] 
\phi(z=0)=0,\quad\phi(z=h)=\phi^e\quad~&\text{ on }~ (0, T)\\[1mm] 
c^E\left(\tau u_{zt} +u_{z}\right)+e\phi_{z}  
-\beta\Theta=0\quad~&\text{ on }~\partial\Omega \times (0, T) \\[1mm]
k\Theta_z =0 \quad~&\text{ on }~\partial\Omega\times (0, T) \\[1mm]
u(t=0)=u_0\quad, u_t(t=0)=u_1\quad, \Theta(t=0)=\Theta_0 \quad&~\text{in }~\Omega,
	\ear
\ee

\noindent where $ T>0 $ is an arbitrary end time of the observed time period $(0,T)$. The mixed Dirichlet boundary conditions \eqref{mixedDir} describe the excitation behavior via the grounding on the bottom, i.e., at $0$,  and the electrical excitation on the top, i.e., at $h$, by the known signal $\phi^e \in H^1\left( 0, T\right) $. 
 Furthermore, 
$ \rho $ 
 is the mass density, 
 $ c_{\text{th}} $
is the heat capacity and $ k  $ is the thermal conductivity of the material. 
 The positive and bounded function $ \tau $ is the Kelvin-Voigt damping parameter and can be understood as relaxation and $\beta $ is a stress coefficient.
 The parameters describing
the material behavior are the elasticity parameter
$ c^E$,
the piezoelectric coupling parameter
$ e$,
and the permittivity parameter
$ \varepsilon^S $. Recent experiments have shown that $c^E$ behaves almost constantly for different temperatures, see \cite{KFH2W5a}. Consequently, we reformulate the material parameters as follows
\begin{equation}\label{param_p}
    c^E:=p_1, \quad e:=p_2(z,t),\quad \varepsilon^S:=p_3(z,t). 
\end{equation}
Furthermore, we denote $f(z,t)=(p_1,p_2(z,t),p_3(z,t))^T$.
Similarly to  \cite{RKDWAW}, we homogenize the mixed Dirichlet boundary conditions by using a Dirichlet lift Ansatz. 
 Therefore, we introduce the Dirichlet lift function 
$ \chi\in H^1((0,T); H^m(\Omega, \mathbb{R})), ~m\geq 2 $ with the property that  
\[\mathrm{Tr}(\chi(t))=\begin{cases}
	{\phi}_e(t) &~~\mathrm{at}~h\\
	0 & ~~\mathrm{at}~0
	\end{cases} ~\text{ a.e. in time}\] and can express $\phi$ as   $ \phi(t)=\phi^0(t)+\chi(t) $ a.e. in time, where $ {\phi}^0(t) \in H_{0}^1(\Omega, \mathbb{R})$.This motivates the following definitions
\begin{align}    \Gamma(z,t)&:=\tau(\Theta(z,t))p_1\label{Gam},\\ 
    b(z,t)&:=c_{th}(z,t)\rho(z,t)\label{bzt}.
\end{align}
For the subsequent analysis, the following assumptions are imposed.
\begin{assumptions}\label{vorrfunc}
    Let the following hold:
\be{}
\lball
\begin{array}{l}
\rho \in C^2([0,T];C(\bom, \R)) \text{ is positive,}\\
p_1\in \R^+,\\
p_i\in C^1(\bom\times [0,T]) \text{ for }i= 2,3\text{ are positive,}\\
\Gamma \in C^0([0,\infty))\text{ is positive and bounded,}\\
b\in C^1([0,T];C(\bom, \R))\text{ is positive,}\\
k \in C([0, T];C^1(\bom, \R))\text{ is positive and}\\
\chi\in H^1((0,T); H^m(\Omega, \mathbb{R})), ~m\geq 2\\
\beta \in \R^+.
\end{array}
\ear
\ee
\end{assumptions}
\noindent For the initial conditions, we assume
\be{initial}
u_0,u_1,\Theta_0\in L^1(\Om)\qquad \mbox{ with }\Theta_0\ge0.
\ee
This leads to
\be{GLIu1}
    \lball
    \rho u_{tt}-\frac{\dd{}}{\dd{z}} \left( p_1u_z +\Gamma(\Theta)u_{zt}
+ p_2  \phi_z^0\right)=\frac{\dd{}}{\dd{z}}(p_2  \chi_z) -\frac{\dd}{\dd{z}}(\beta \Theta)&\text{ in }~\Omega\times (0, T)\\[1mm]
-\frac{\dd{}}{\dd{z}}\left( p_2u_z- p_3  \phi_z^0\right) =-\frac{\dd{}}{\dd{z}} (p_3  \chi_z)&\text{ in }~\Omega\times (0, T)\\[1mm]
b \Theta_t-\frac{\dd{}}{\dd{z}}\left(k\Theta_z\right) - \Gamma(\Theta) u_{zt}^2+\beta \Theta u_{zt}=0&\text{ in }~\Omega\times (0, T)\\[1mm]
p_1u_z +\Gamma(\Theta)u_{zt}
+ p_2  \phi_z^0 
-\beta\Theta=-p_2  \chi_z &\text{ on }~\partial\Omega\times (0, T)\\[1mm]
k\Theta_z =0 &\text{ on }~\partial\Omega\times (0, T)\\[1mm]
u(t=0)=u_0,\quad u_t(t=0)=u_1,\quad \Theta(t=0)=\Theta_0 &\text{ in }~\Omega.
	\ear
\ee
We formulate the corresponding weak solution by
\begin{definition}\label{DefWeakFull}
   Let $\Omega\subset\mathbb{R}$ be a bounded open interval, and suppose that for $T=\infty$, Assumptions \ref{vorrfunc} hold and let $u_0,u_1,\Theta_0\in L^1(\Omega)$ with $\Theta_0\ge 0$. Then a global weak solution of the system above is a triple $(u,\Theta,\phi^0)$ of functions
    \bea{weakdef1full}
    \lball
    \phi^0 \in L^1_{loc}([0,\infty);W^{1,1}(\Om))\\
    u\in C^0([0,\infty);L^1(\Omega))\cap L_{loc}^1([0,\infty);W^{1,1}(\Omega))\quad\text{and}\\
    \Theta\in L_{loc}^1([0,\infty);W^{1,1}(\Omega))
    \ear
    \eea
    such that
    \bea{weakdef2full}
    u_t\in L_{loc}^2(\Omega\times[0,\infty))
    \eea
    as well as
    \bea{weakdef3full}
    \{\Gamma(\Theta) u_{zt},\ \Gamma(\Theta)u_{zt}^2,\ \Theta u_{zt}\}\subset L_{loc}^1(\Omega\times[0,\infty)),
    \eea
    and $u(\cdot,0)=u_0$ a.e. in $\Omega$, $\Theta\ge 0$ a.e. in $\Om\times(0,\infty)$ satisfying
\bea{weak1full}
     &\quad&\int_0^\infty\int_\Omega \rho u_t\varphi_t+\int_0^\infty\io \rho_t u_t\vp -\io \rho(\cdot,0) u_{0t}\vp(\cdot,0)\nn\\
     &=&\int_0^\infty \io \big(\Gamma(\Theta)u_{zt}+p_1u_z+p_2\phi^0_z+p_2\chi_z-\beta\Theta)\big)\varphi_z\hspace{-10mm}
\eea
as well as
\bea{weakelipfull}
\int_0^\infty \io (p_2u_z-p_3\phi^0_z)\psi_z = \int_0^\infty \io p_3 \chi_z \psi_z
\eea
and
\bea{weakthermofull}
    &\quad&\int_0^\infty \int_\Omega b_t\Theta \vp+\int_0^\infty \io b\Theta \varphi_t-\io b(\cdot,0) \Theta_0\vp(\cdot,0)\nn\\
    &=&\int_0^\infty\io k\Theta_z\varphi_z-\int_0^\infty \io \big(\beta\Theta u_{zt}+\Gamma(\Theta)u_{zt}^2\big)\vp
\eea
for all $\psi,\vp\in C^\infty_0(\bom\times[0,\infty))$.
\end{definition}
\noindent Now, consider the related system
\be{system1}
\lball
\rho u_{tt}-\frac{\dd{}}{\dd{z}} \left(\Gamma(\Theta) u_{zt}
+p u_z\right)= -\frac{\dd}{\dd{z}}(\beta \Theta)&\quad\text{in }~\Omega\times (0, T)\\
b \Theta_t-\frac{\dd{}}{\dd{z}}\left(k\Theta_z\right) - \Gamma(\Theta) u_{zt}^2+\beta \Theta u_{zt}=0&\quad\text{in }~\Omega\times (0, T)\\ 
\Gamma(\Theta) u_{zt}
+pu_z 
-\beta\Theta=0 &\quad\text{on }~\partial\Omega\times (0, T)\\
k\Theta_z =0 &\quad\text{on }~\partial\Omega\times (0, T)\\
u(t=0)=u_0,\quad u_t(t=0)=u_1,\quad \Theta(t=0)=\Theta_0 &\quad\text{in }~\Omega,
    \ear
\ee
where $p\in C^1(\bom\times[0,\infty))$ and $p>0$.
For the corresponding solution concept, we define in a similar way to the previous one:
 
\begin{definition}\label{DefWeak}
   Let $\Omega\subset\mathbb{R}$ be a bounded open interval, and suppose that for $T=\infty$, Assumptions \ref{vorrfunc} holds. Furthermore, let $p\in C^1(\bom\times(0,T))$ satisfy $p>0$ and let $u_0,u_1,\Theta_0\in L^1(\Omega)$ with $\Theta_0\ge 0$. Then a global weak solution of the system above is a pair $(u,\Theta)$ of functions
    \bea{weakdef1}
    \lball
    u\in C^0([0,\infty);L^1(\Omega))\cap L_{loc}^1([0,\infty);W^{1,1}(\Omega))\quad\text{and}\\
    \Theta\in L_{loc}^1([0,\infty);W^{1,1}(\Omega))
    \ear
    \eea
    such that
    \bea{weakdef2}
    u_t\in L_{loc}^2(\Omega\times[0,\infty))
    \eea
    as well as
    \bea{weakdef3}
    \{\Gamma(\Theta) u_{zt},\ \Gamma(\Theta)u_{zt}^2,\ \Theta u_{zt}\}\subset L_{loc}^1(\overline{\Omega}\times[0,\infty)),
    \eea
    and $u(\cdot,0)=u_0$ a.e. in $\Omega$, $\Theta\ge 0$ a.e. in $\Om\times(0,\infty)$ satisfying

\bea{weak1}
     &\quad&\int_0^\infty\int_\Omega \rho u_t\varphi_t+\int_0^\infty\io \rho_t u_t\vp -\io \rho(\cdot,0) u_{0t}\vp(\cdot,0)\nn\\
     &=&\int_0^\infty \io \big(\Gamma(\Theta)u_{zt}+p u_z-\beta\Theta)\big)\varphi_z\hspace{-10mm}
\eea
and
\bea{weak2}
    &\quad&\int_0^\infty \int_\Omega b_t\Theta \vp+\int_0^\infty \io b\Theta \varphi_t-\io b(\cdot,0) \Theta_0\vp(\cdot,0)\nn\\
    &=&\int_0^\infty\io k\Theta_z\varphi_z-\int_0^\infty \io \big(\beta\Theta u_{zt}+\Gamma(\Theta)u_{zt}^2\big)\vp
\eea
for all $\vp\in C^\infty_0(\bom\times[0,\infty))$.
\end{definition}
\noindent To ensure that the two solution concepts of the systems \eqref{GLIu1} and \eqref{system1} are indeed equivalent in our specific case, we will verify this in the following lemma:
\begin{lemma}\label{aequivalenzLK}
Let Assumptions \ref{vorrfunc} hold. Then, the triple $(u,\phi^0,\Theta)$ solves \eqref{GLIu1} in the sense of Definition \ref{DefWeakFull} if and only if the pair $(u,\Theta)$ is a solution to \eqref{system1} in the sense of Definition \ref{DefWeak} with $$p:=p_1+\frac{p_2^2}{p_3}\quad\mbox{ and }\quad \phi^0\in L^1_{loc}([0,\infty);W^{1,1}(\Om))$$
    satisfying  the relation
\be{aequi1}
p_3(z,t)\phi_z^0(z,t)=p_2(z,t) u_z(z,t)-p_3(z,t)\chi_z(z,t) \qquad\mbox{ a.e. in }\Om\times[0,\infty)
\ee
\end{lemma}
\begin{proof}
    $\Rightarrow$: Let $(u,\phi^0,\Theta)$ be a weak solution of \eqref{GLIu1} in the sense of definition \ref{DefWeakFull}.
We set 
\be{aequi2}
F(z,t):= p_2(z,t) u_z(z,t) -p_3(z,t)\phi^0_z(z,t)-p_3(z,t)\chi(z,t).
\ee
Recalling our assumptions we have $F\in L^1_{loc}(\bom\times[0,\infty))$ and equation \eqref{weakelipfull} reformulates to
\be{aequi3}
\int_0^\infty \io F\psi_z dzdt =0 \qquad \mbox{ for all }\psi\in C^\infty_0(\bom\times[0,\infty)).
\ee
Choosing $\psi(z,t)=\eta(z)\nu(t)$ with $\eta\in C^\infty_0(\bom)$ and $\nu(t)\in C^\infty_0(0,\infty)$, we have
$$\io F(z,t)\eta_z(z) dz=0 \qquad \forall \eta\in C^\infty_0(\bom) \mbox{ and a.e. }t\in[0,\infty)$$
which implies $\partial_z F(z,t)=0$ in the distributional sense $\mathcal{D}'(\Om)$. Since $F(\cdot,t)\in L^1(\Om)$, for $t>0$, we have $F$ being spatial homogeneous
$$F(z,t)=C(t) \qquad \mbox{for a.e. }z\in\Om, \mbox{ and a.e. }t\in [0,\infty)$$
Going back to \eqref{aequi3} and choosing any function satisfying $\psi(z,t)=\eta(z)\nu(t)$ with $\nu>0$ and $\io \eta_z(z)=1$ we obtain
$$0=\int_0^\infty C(t) \nu(t) \io \eta_z(z)dzdt= \int_0^\infty C(t)\nu(t) dt$$
and thereby $C(t)=0$ for a.e. $t\in[0,\infty)$.
Since Assumptions \ref{vorrfunc} ensure $p_2/p_3\in L^\infty_{loc}(\Om\times[0,\infty))$ and $u_z,\chi_z\in L^1_{loc}([0,\infty);L^1(\Om))$, the substitution $$p_2\phi^0_z=\frac{p_2^2}{p_3}u_z -p_2\chi_z$$ is permissible in \eqref{weak1full} and for any $\vp\in C^\infty_0(\bom\times[0,\infty))$ we reformulate
\begin{align*}
&\quad\int_0^\infty\int_\Omega \rho u_t\varphi_t+\int_0^\infty\io \rho_t u_t\vp -\io \rho(\cdot,0) u_{0t}\vp(\cdot,0)\nn\\
     &=\int_0^\infty \io \big(\Gamma(\Theta)u_{zt}+p_1u_z+p_2\phi^0_z+p_2\chi_z-\beta\Theta)\big)\varphi_z\\
     &= \int_0^\infty \io \left(\Gamma(\Theta)u_{zt}+p_1u_z \frac{p_2^2}{p_3} u_z -\beta\Theta\right)\vp_z+\int_0^\infty \io \left(-p_2\chi_z+p_2\chi_z\right)\vp_z\\
     &=\int_0^\infty \io \left(\Gamma(\Theta)u_{zt}+pu_z-\beta\Theta\right)\vp_z\hspace{-10mm}.
\end{align*}
Hence we arrive at \eqref{weak1} and in combination with the fact that the temperature equation formulates exactly the same in both definitions, we conclude that $(u,\Theta)$ is a solution of \eqref{system1} in the sense of Definition \eqref{DefWeak}.\\
$\Leftarrow$: Let $(u,\Theta)$ be a weak solution of \eqref{system1} in the sense of Definition \eqref{DefWeak}. We define 
$$\phi_z^0:= \frac{p_2(z,t)}{p_3(z,t)} u_z(z,t) -\chi_z(z,t) \qquad\mbox{ a.e. in }\Om\times[0,\infty),$$
and set $\phi^0 := \int_a^z \phi_z^0 (s,t)ds$. Since Assumptions \ref{vorrfunc} ensure $p_2/p_3\in L^\infty_{loc}(\Om\times[0,\infty))$ and $u_z,\chi_z\in L^1_{loc}([0,\infty);L^1(\Om))$, we have $\phi_z^0\in L^1_{loc}([0,\infty);L^1(\Om))$ and thus $\phi^0\in L^1_{loc}([0,\infty);W^{1,1}(\Om))$. Our definition also implies
\be{aequi4}\int_0^\infty \io \left(p_2u_z-p_3\phi^0_z\right) \phi_z = \int_0^\infty \io p_3 \chi_z \psi_z \qquad\mbox{ for all }\psi\in C^\infty_0(\bom\times[0,\infty))\ee
and thereby verifies \eqref{weakelipfull}. Moreover, by adding $$\int_0^\infty (-p_2\chi_z+p_2\chi_z)\vp_z=0$$ on the right hand side of \eqref{weak1} and using \eqref{aequi4}, we obtain \eqref{weak1full}. Thus, the triple $(u,\phi^0,\Theta)$ is a weak solution of \eqref{GLIu1}.
\end{proof}

\section{Existence and regularity of solutions}\label{sec:ana}
\noindent This section is devoted to the statement of the global existence and regularity results. The proof is split into two parts, presented in Section \ref{sec4} and \ref{sec5}. In Section \ref{sec4}, we study an approximate system and prove global solvability, while Section \ref{sec5} is concerned with the passage of the limit.

\noindent The primary qualitative analytical finding of this study is encapsulated in the ensuing theorem.

\begin{theorem}\label{EX}
Let $\Omega\subset\mathbb{R}$ be a bounded open interval, and let Assumptions \ref{vorrfunc} holds. 
Then if
\bea{ab01}
c_\Gamma\le \Gamma(\xi)\le C_\Gamma\qquad\text{for all }\xi\ge 0
\eea
hold for some $c_\Gamma>0$ and $C_\Gamma>0$.
Then whenever
$$u_0\in W^{1,2}(\Omega),\quad u_{0t}\in L^2(\Omega)\quad \text{and}\quad \Theta_0\in L^2(\Omega)$$
are such that $u_0|_{\partial\Om}\in\mathbb{R}$, $u_{0t}|_{\partial\Om}\in\mathbb{R}$ and $\Theta_0\ge 0$ in $\bom$, then \eqref{sys0} admits a global weak solution $(u,\Theta)$ in the sense of Definition \ref{DefWeak}, which satisfy
\be{theo3.2.2}
    \lball
    u_t \in L^\infty_{loc}((0,\infty);L^2(\Om))\cap L^2_{loc}([0,\infty);W^{1,2}(\Om)),\\
    u\in C^0(\overline{ \Omega}\times[0,\infty))\cap L^\infty_{loc}((0,\infty); W^{1,2}(\Omega))\qquad and\\
    \Theta \in L^\infty_{loc}((0,\infty);L^1(\Om))\cap\bigcap_{q\in[1,3)} L^q_{loc}(\bom \times [0,\infty))\cap \bigcap_{r\in[1,\frac32)} L^r_{loc} ([0,\infty);W^{1,r}(\Omega)).
    \ear
\ee
\end{theorem}
\subsection*{Remark}
i) In this work, we assume that the coefficients $\rho$, $b$, $k$, $p$ and $\Gamma$ are bounded and positive. In contrast, studies such as \cite{winkler3} and \cite{Fricke} do not impose boundedness, and in particular $\Gamma$ is allowed to be unbounded. Assuming bounded coefficients is reasonable from a physical perspective, as it corresponds to realistic material parameters and ensures that the energy of the system remains controlled. 
Furthermore, this approach is standard in numerous works on related systems (\cite{Meyer},\cite{winkler2},\cite{Winklerweak1d},\cite{WinlkerMultiThermo}).\\
\noindent ii) Whilst a considerable number of preceding studies have taken into account constant coefficients, such as $\rho$ and $b$ (frequently normalized to $1$), our approach of incorporating spatially and temporally varying $\rho$ and $k$ introduces further complexities. In particular, reliance on the Neumann heat semigroup to establish regularity properties is no longer a viable option. This necessitates a more sophisticated approach involving higher-order Sobolev estimates and energy bounds in the analysis.

\section{Approximate System and a priori estimates}\label{sec4}

\subsection{Preliminaries}
Before deriving the first auxiliary results, we transform the wave model into a parabolic setting using the substitution $v:=u_t$ and regularize the original system using an approximate system which converges to the original thermo-piezoelectric system as the approximation parameter $\varepsilon \searrow 0$. For this, let $(v_{0\eps})_{\eps\in(0,1)}\subset C^\infty(\Om)$, $(u_{0\eps})_{\eps\in(0,1)}\subset C^\infty(\Om)$ and $(\Theta_{0\eps})_{\eps\in(0,1)}\subset C^\infty(\overline{\Om})$ be such that $\Theta_{0\eps}\ge 0$ in $\Om$ for all $\eps\in(0,1)$, and that as $\eps\searrow 0$ we have
\be{initkonv}
	\lball
    v_{0\eps}\to u_{0t}\quad & \text{in }L^2(\Om), \\[1mm]
    u_{0\eps}\to u_0\quad & \text{in }W^{1,2}(\Om)\quad\text{and}\\[1mm]
    \Theta_{0\eps}\to\Theta_0\quad & \text{in }L^1(\Om).
	\ear
\ee
We moreover let Assumptions \ref{vorrfunc} holds and $0<p\in C^1(\bom\times[0,\infty))$. In addition let $(\rho_\eps)_{\eps\in(0,1)},$ $ (p_{\eps})_{\eps\in(0,1)},$ $(b_\eps)_{\eps\in(0,1)}$ and $(k_\eps)_{\eps\in(0,1)}$ be positive and in $C^\infty(\overline{\Omega}\times[0,\infty))$, and $(\Geps)_{\eps\in(0,1)}$ be in $C^\infty([0,\infty))$ such that
\bea{konvabs1}
\Geps&\to&\Gamma \quad\text{in }{L^\infty_{loc}([0,\infty))},\label{KGZ_Gamma}\\
\peps &\to& \rho \quad\text{in }{ C_{loc}^2([0,\infty);C^0(\bom,\R)),}\label{KGZ_Rho}\\
p_\eps &\to& p \quad\text{in }{ L^\infty_{loc}(\bom\times[0,\infty))},\label{KGZ_P}\\
p_{\eps t} &\to& p_t \quad \text{in }{L^\infty_{loc}(\bom\times[0,\infty))},\label{KGZ_Pt}\\
p_{\eps z} &\to& p_z \quad \text{in } L^\infty_{loc}(\bom\times[0,\infty)), \label{KGZ_Pz}\\
\keps &\to& k \quad \text{ in }{ L^\infty_{loc}(\bom \times [0,\infty))},\label{KGZ_K}\\
b_\eps &\to& b \quad \text{ in } {C_{loc}^1([0,\infty);C^0(\bom,\R))}\label{KGZ_B}
\eea
as $\eps\searrow0$, as well as
\bea{ab1}
c_\Gamma\leq \Geps(\xi) \leq C_\Gamma\qquad\text{for all }\xi\ge 0\text{ and }\eps\in(0,1),
\eea
and
\bea{ab2}
\rho(z,t)\leq \peps(z,t) \leq \rho(z,t)+1,
\eea
\bea{ab5.1}
p(z,t)\le p_{\eps}(z,t)\le p(z,t)+1,
\eea
\bea{ab3}
b(z,t)\leq \beps(z,t) \leq b(z,t)+1,
\eea
and
\bea{ab4}
k(z,t)\leq \keps(z,t)\leq k(z,t)+1
\eea
for all $(z,t)\in\Omega\times(0,\infty)$ and $\eps\in(0,1)$.
 Additionally, it is also necessary to impose the conditions that
\bea{ab6}
b_{\eps z}=0\qquad\text{and}\qquad k_{\eps z}=0,\quad \text{for }z\in\partial\Om,\ t>0\text{ and }\eps\in(0,1),
\eea
as well as
\bea{ab7}
p_{\eps z}=0,\quad\text{for }z\in\partial\Om,\ t>0\text{ and }\eps\in(0,1).
\eea
For $\eps\in(0,1)$ we consider

\be{approx}
	\lball
    \rho_\varepsilon v_{\varepsilon t} = -\eps v_{\varepsilon zzzz}+(\Gamma_\varepsilon(\Theta_\varepsilon) v_{\varepsilon z})_z+(p_{\varepsilon} u_{\varepsilon z})_z -\beta\Theta_{\eps z}, 
	\qquad & z\in\Omega, \ t>0, \\[1mm]
    u_{\varepsilon t}\hspace{3mm} = \varepsilon u_{\varepsilon zz}+v_\varepsilon, 
	\qquad & z\in\Omega, \ t>0, \\[1mm]
    b_\varepsilon \Theta_{\varepsilon t}\hspace{0.5mm} = (k_\varepsilon\Theta_{\varepsilon z})_z + \Gamma_\varepsilon(\Theta_\varepsilon) v_{\varepsilon z}^2 - \beta \Theta_\varepsilon v_{\varepsilon z} , 
	\qquad & z\in\Omega, \ t>0, \\[1mm]
    v_{\eps}=v_{\eps zz}=0,\quad u_\eps=0,\quad \Theta_{\eps z}=0, & z\in\partial\Omega,\ t>0, \\[1mm]
    v_\eps(z,0)=v_{0\eps}(z),\quad u_\eps(z,0)=u_{0\eps}(z),\quad\Theta_{\eps}(z,0)=\Theta_{0\eps}(z), & z\in\Om.
	\ear
\ee
\begin{remark}\label{RemarkShift}
The original system \eqref{system1} implies $u_{tt}=0$ on $\partial\Om$ for all $t>0$, so $u_t|_{\partial\Om}$ is constant in time and determined by $u_{0t}|_{\partial\Om}$. Thus we impose homogeneous Dirichlet boundary conditions for $v_\eps$ and $u_\eps$ in \eqref{approx}. Without loss of generality, $v_{0\eps}|_{\partial\Om}=0$, as otherwise, we replace $v_\eps$ by $v_\eps-c_\eps$ and $u_\eps$ by $u_\eps-c_\eps t$, with $c_\eps:=v_{0\eps}|_{\partial\Om}$, preserving the structure of \eqref{approx}. The condition $v_{\eps zz}=0$ on $\partial\Om$ is the natural second boundary condition associated with the fourth-order term $-\eps v_{\eps zzzz}$ under the Dirichlet conditions.
\end{remark}
\noindent Indeed, all of these problems are accessible by the conventional theory of local solvability in parabolic systems: 
\begin{lemma}\label{2.1Ref}
Let $\varepsilon\in(0,1)$. Then there exist $T_{max,\varepsilon}\in(0,\infty]$ and functions
\be{ae1}
    \lball
    v_\varepsilon\in C^0(\overline{ \Omega}\times[0,T_{max,\varepsilon})\cap C^{4,1}(\overline{ \Omega}\times(0,T_{max,\varepsilon})),\\
    u_\varepsilon\in C^0(\overline{ \Omega}\times[0,T_{max,\varepsilon})\cap C^{2,1}(\overline{ \Omega}\times(0,T_{max,\varepsilon}))\cap C^0([0,T_{max,\varepsilon}); W_0^{1,2}(\Omega))\qquad and\\
    \Theta_\varepsilon\in C^0(\overline{ \Omega}\times[0,T_{max,\varepsilon})\cap C^{2,1}(\overline{ \Omega}\times(0,T_{max,\varepsilon}))
    \ear
\ee
such that $\Theta_\varepsilon\ge 0$ in $\overline{\Omega}\times[0,T_{max,\varepsilon})$, that \eqref{approx} is solved in the classical sense in $\Omega\times(0,T_{max,\varepsilon})$, and that
\bea{ae2}
    if\ T_{max,\varepsilon}<\infty,\quad then\hspace{7cm}\\
    \limsup_{t\nearrow T_{max,\varepsilon}}\Big\{\|v_\varepsilon(\cdot,t)\|_{W^{2+2\eta,\infty}(\Omega)}+\|u_\varepsilon(\cdot,t)\|_{W^{1+\eta,\infty}(\Omega)}+\|\Theta_\varepsilon(\cdot,t)\|_{W^{1+\eta,\infty}(\Omega)}\Big\}=\infty\ for\ all\ \eta>0.\nn
\eea
\end{lemma}
\begin{proof}
    The local existence of such solutions and \eqref{ae2} can be derived in a straightforward manner from the general results concerning existence and extensibility from \cite[Theorem~12.1]{Amann1993} and \cite[Theorem~12.5]{Amann1993}, while the nonnegativity of $\Theta_\eps$ follows by application of the strong maximum principle (see, for instance, \cite[Lemma~2.1]{Winklerweak1d}). For further information pertaining to associated issues, please refer to \cite[Lemma~2.3]{ClaesLankeitWinkler} and \cite{WinlkerMultiThermo}.
\end{proof}
\noindent It is worth mentioning that in similar problems, when the initial data satisfies $\Theta_{0\eps}> 0$ in $\Om$, one can still conclude $\Theta_\eps> 0$ in $\overline{\Om}\times[0,\tme)$ (see \cite{WinlkerMultiThermo}). In this case however, it should be noted that strict positivity cannot simply be taken for granted.

\noindent Prior to embarking the qualitative analysis of the system and the exploitation of suitable energy inequalities, it is first observed that the functions obtained in the aforementioned lemma actually possess higher regularity.
\begin{lemma}
Let $\varepsilon\in(0,1)$, then
\be{ae1infty}
    \lball
    v_\varepsilon\in C^\infty(\overline{ \Omega}\times[0,T_{max,\varepsilon}),\\
    u_\varepsilon\in C^\infty(\overline{ \Omega}\times[0,T_{max,\varepsilon})\qquad and\\
    \Theta_\varepsilon\in C^\infty(\overline{ \Omega}\times[0,T_{max,\varepsilon}).
    \ear
\ee
\end{lemma}
\begin{proof}
    Since all initial data and coefficients are smooth, the right-hand sides of each equation are immediately Hölder-continuous at $t=0$. Standard parabolic Schauder theory (\cite[Ch. IV, Theorem~5.3]{LSU}) then imply that there exists $\alpha_1\in(0,1)$, such that $v_\eps\in C^{1+\alpha_1,\frac{1+\alpha_1}{2}}(\overline{\Om}\times[0,\tme))$, $u_\eps\in C^{2+\alpha_1,\frac{2+\alpha_1}{2}}(\overline{\Om}\times[0,\tme))$ and $\Theta_\eps\in C^{1+\alpha_1,\frac{1+\alpha_1}{2}}(\overline{\Om}\times[0,\tme))$. As $\rho_\eps$, $\Gamma_\eps$, $\Gamma_\eps'$, $p_\eps$, $p_\eps'$, $b_\eps$, $k_\eps$ and $k_\eps$ are smooth and locally Lipschitz continuous in $[0,\infty)$ according to our assumptions, this particularly implies the existence of $\alpha_2\in(0,\alpha_1]$ such that $$\frac{\Gamma_{\eps z}(\Theta_\eps)\Theta_{\eps z}v_{\eps z}}{\rho_\eps}+\frac{(\rho_\eps u_{\eps z})_z}{\rho_\eps}-\frac{\beta\Theta_{\eps z}}{\rho_\eps}\in C^{\alpha_2,\frac{\alpha_2}{2}}(\overline{\Om}\times[0,\tme)),$$
    $$v_\eps\in C^{1+\alpha_2,\frac{1+\alpha_2}{2}}(\overline{\Om}\times[0,\tme))\qquad \text{and}$$
    $$\frac{k_{\eps z}\Theta_{\eps z}}{b_\eps}+\frac{\Gamma_\eps(\Theta_\eps)v_{\eps z}^2}{b_\eps}-\frac{\beta\Theta_\eps v_{\eps z}}{b_\eps}\in C^{\alpha_2,\frac{\alpha_2}{2}}(\overline{\Om}\times[0,\tme)),$$
    where, due to the fact that $v_{0\eps}$, $u_{0\eps}$ and $\Theta_{0\eps}$ are smooth and all compatibility conditions from \cite[Ch. IV, Theorem~5.3]{LSU} are satisfied, basic parabolic Schauder theory (\cite[Ch. IV, Theorem~5.3]{LSU}) yields that $u_\eps\in C^{3+\alpha_2,\frac{3+\alpha_2}{2}}(\overline{\Om}\times[0,\tme))$, and that $v_\eps$ and $\Theta_\eps$ are in $C^{2+\alpha_2,\frac{2+\alpha_2}{2}}(\overline{\Om}\times[0,\tme))$, where a straightforward iterative bootstrap argument shows \eqref{ae1infty}.
\end{proof}

\noindent It is evident from the examination of three elementary test procedures that an energy structure for \eqref{approx} can be obtained, thus enabling the formulation of preliminary estimates. The results are largely of independent use and are therefore listed separately here:
\begin{lemma}
    \label{vepsuepslem}
    If $\varepsilon\in(0,1)$, then
    \bea{vepsueps}
    &&\hspace{-10mm}
    \frac{1}{2}\frac{d}{dt}\io \rho_\eps v_\eps^2+\frac{1}{2}\frac{d}{dt}\io p_\eps u_{\eps z}^2+\io \Gamma_\eps(\Theta_\eps) v_{\eps z}^2+\eps \io v_{\eps zz}^2+\eps \io p_\eps u_{\eps zz}^2+\eps\io p_{\eps z} u_{\eps z}u_{\eps zz}\nonumber\\
    &&= \beta\io \Theta_\eps v_{\eps z}+\frac12\io \rho_{\eps t} v_\eps^2+\frac{1}{2}\io p_{\eps t} u_{\eps z}^2\qquad\text{for all }t\in(0,\tme).
    \eea
\end{lemma}
\begin{proof}
    Testing the first equation of \eqref{approx} with $v_\eps$ gets us
    \bea{t1}
        \frac{1}{2}\frac{d}{dt}\int_\Omega \peps v_\eps^2&=&-\eps\io v_{\eps zzzz}v_\eps + \io (\Geps(\Theta_\eps) v_{\eps z})_z v_\eps + \io (p_\eps u_{\eps z})_z v_{\eps}-\beta\io\Theta_{\eps z} v_{\eps }+\frac{1}{2}\io \rho_{\eps t}v_\eps^2\nonumber\\
        &=& -\eps\io v_{\eps zz}^2-\int_\Omega\Gamma_\eps(\Theta_\eps) v_{\eps z}^2-\io p_\eps u_{\eps z}v_{\eps z}+\beta\io\Theta_\eps v_{\eps z}+\frac12\io \rho_{\eps t}v_\eps^2
    \eea
    for all $t\in(0,\tme)$ and $\eps\in(0,1)$. Testing the second equation in \eqref{approx} we obtain
    \bea{t1.1}
    \frac{1}{2}\frac{d}{dt}\io p_\eps u_{\eps z}^2&=&\io p_\eps u_{\eps z}\cdot(\eps u_{\eps zz}+v_\eps)_z+\frac{1}{2}\io p_{\eps t} u_{\eps z}^2\nonumber\\
    &=&-\eps\io p_\eps u_{\eps zz}^2-\eps\io p_{\eps z} u_{\eps z}u_{\eps zz} +\io p_\eps u_{\eps z}v_{\eps z}+\frac{1}{2}\io p_{\eps t} u_{\eps z}^2
    \eea
    for all $t\in(0,\tme)$ and $\eps\in(0,1)$. Therefore, \eqref{vepsueps} follows directly from \eqref{t1} and \eqref{t1.1}. 
\end{proof}
\noindent When supplemented by a simple observation on evolution of the functional $\io\Theta_\eps$, the above yields the following regularized statement:
\begin{lemma}\label{energy1Lem}
    For any $T>0$ there exist $C_1(T)>0$ and $C_2(T)>0$ such that for all $\eps\in(0,1)$ the solution of \eqref{approx} satisfies for all $t\in(0,\tme)\cap(0,T)$
    \bea{energy1}
    &&\hspace{-10mm}
    \frac{d}{dt}\bigg\{\frac{1}{2}\io\rho_\eps v_{\eps}^2+\frac{1}{2}\io u_\eps^2+\frac{1}{2}\io p_\eps u_{\eps z}^2+\io b_\eps\Theta_\eps\bigg\}+\eps C_1(T)\io v_{\eps zz}^2+\eps C_1(T)\io u_{\eps zz}^2\nonumber\\
    &\le& C_2(T)\io v_{\eps }^2+C_2(T)\io u_\eps^2+C_2(T)\io u_{\eps z}^2+\|b_{\eps t}\|_{L^\infty(\Omega)}\io \Theta_\eps.
    \eea
    \end{lemma}
    \begin{proof}
    Since $\rho_\eps$ and $p_\eps$ satisfy \eqref{ab2} and \eqref{ab5.1} as well as Assumptions \ref{vorrfunc} holds, there exists $c_1=c_1(T)>1$ such that
    $$\frac{1}{c_1} \le p_\eps(z,t), \quad p_{\eps z}(z,t) \le c_1, \quad p_{\eps t}(z,t) \le c_1,\quad\mbox{and}\quad \rho_{\eps t}(z,t) \le c_1$$
    for all $(z,t)\in\Om\times(0,T)$ and $\eps\in(0,1)$ and with Young's inequality
    $$\eps\Big|\io p_{\eps z} u_{\eps z} u_{\eps zz}\Big|\le \frac{\eps}{2c_1}\io u_{\eps zz}^2 +\frac{\eps c_1^3}{2} \io u_{\eps z}^2 \qquad\mbox{for all }t\in(0,\tme)\cap(0,T)\text{ and }\eps\in(0,1).$$
    Testing the second equation of \eqref{approx} with $u_\eps$ directly vields
    \bea{tt3.3}
    \frac{1}{2}\frac{d}{dt}\io u_\eps^2\leq \io u_\eps v_\eps\leq\io u_\eps^2+\io v_\eps^2\quad\text{for all }t\in(0,\tme)\text{ and }\eps\in(0,1).
    \eea
    As
    \bea{t3}
    \quad\frac{d}{dt}\io b_\eps \Theta_\eps=\io \Gamma_\eps(\Theta_\eps)v_{\eps z}^2-\beta\io\Theta_\eps v_{\eps z}+\io b_{\eps t}\Theta_\eps\quad\text{for all }t\in(0,\tme)\text{ and }\eps\in(0,1)\hspace{-10mm}
    \eea  
    by the third equation in \eqref{approx}, we obtain \eqref{energy1} from \eqref{vepsueps}, \eqref{tt3.3} and \eqref{t3} thanks to two favorable cancellations for $C_1(T):=\frac{1}{2c_1}$ and $C_2(T):=c_1^3+1$ in view of $c_1>1$. 
    \end{proof}
\noindent Evident consequences of \eqref{energy1} provide some fundamental regularity information:
\begin{lemma}\label{MassenAbsch}
    For any $T>0$ there exists $C(T)>0$ such that for all $\eps\in(0,1)$
    \bea{t5}
    \io v_\eps^2(\cdot,t)\le C(T)\qquad\text{for all }t\in(0,\tme)\cap(0,T)
    \eea
    and
    \bea{t5/5}
    \io u_\eps^2(\cdot,t)\le C(T)\qquad\text{for all }t\in(0,\tme)\cap(0,T)
    \eea
    and
    \bea{t6}
    \io u_{\eps z}^2(\cdot,t)\le C(T)\qquad\text{for all }t\in(0,\tme)\cap(0,T)
    \eea
    and
    \bea{t7}
    \io \Theta_\eps(\cdot,t)\le C(T)\qquad\text{for all }t\in(0,\tme)\cap(0,T)
    \eea
    as well as
    \bea{t8}
    \eps\int_0^t\io v_{\eps zz}^2\le C(T)\qquad\text{for all }t\in(0,\tme)\cap(0,T)
    \eea
    and
    \bea{t9}
    \eps\int_0^t\io u_{\eps zz}^2\le C(T)\qquad\text{for all }t\in(0,\tme)\cap(0,T).
    \eea
    \begin{proof}
    Let $C_1(T)$ and $C_2(T)$ be as in Lemma \ref{energy1Lem}. We define for all $\eps\in(0,1)$
    \bea{t9.1}
    \hspace{10mm}y_\eps(t):=\Big\{\frac{1}{2}\io\rho_\eps v_{\eps}^2+\frac{1}{2}\io u_\eps^2+\frac{1}{2}\io p_\eps u_{\eps z}^2+\io b_\eps\Theta_\eps\Big\}\quad\text{for all }t\in(0,\tme)\cap(0,T).
    \eea
    Due to \eqref{ab2}, \eqref{ab5.1}, \eqref{ab3} and Assumptions \ref{vorrfunc}, we find $c_1=c_1(T)>0$ and $c_2=c_2(T)>0$ such that
    \bea{tn1}
    c_1\le \rho_\eps(z,t)\le c_2,\quad c_1\le p_\eps(z,t)\le c_2\quad\text{and}\quad c_1\le b_\eps(z,t)\le c_2
    \eea
    for all $(z,t)\in\Om\times(0,T)$ and $\eps\in(0,1)$, so that with \eqref{energy1} the following applies 
    \bea{t9.3}
    \hspace{10mm}y_\eps'(t)\le \frac{C_2(T)}{c_1}\io \rho_\eps v_\eps^2+C_2(T)\io u_\eps^2+\frac{C_2(T)}{c_1}\io p_\eps u_{\eps z}^2+\frac{\|b_{\eps t}\|_{L^\infty(\Omega)}}{c_1}\io b_\eps\Theta_\eps\le c_4 y_\eps(t)
    \eea
    for all $t\in(0,\tme)\cap(0,T)$ and $\eps\in(0,1)$, with $c_3\equiv c_3(T):=\|b_{\eps t}\|_{L^\infty((0,T)\times\Om)}$ and $c_4\equiv c_4(T):=\max\big\{\frac{2C_2(T)}{c_1},2C_2(T),\frac{\|b_{\eps t}\|_{L^\infty((0,T)\times\Omega)}}{c_1}\big\}$. Thus,
    $$y(t)\le y_0 e^{c_3 T}\qquad\text{for all }t\in(0,\tme)\cap(0,T)\text{ and }\eps\in(0,1)$$
    follows directly from the Gronwall Lemma, and thus also statements \eqref{t5}-\eqref{t7}. Now an integration of \eqref{energy1} results in the inequality
        \bea{t10}
        &&\hspace{-10mm}
        \frac{1}{2}\io \rho_\eps v_\eps^2+\frac{1}{2}\io u_\eps^2+\frac{1}{2}\io p_\eps u_{\eps z}^2+\io b_\eps \Theta_\eps+\eps C_1(T)\int_0^t\io v_{\eps zz}^2+\eps C_1(T)\int_0^t\io u_{\eps zz}^2\nonumber\\
        &\le& \frac{1}{2}\io \rho_\eps(\cdot,0)v_{0\eps}^2+\frac{1}{2}\io u_{0\eps }^2+\frac{1}{2}\io p_{\eps}(\cdot,0)u_{0\eps z}^2+\io b_\eps(\cdot,0)\Theta_{0\eps}\nonumber\\
        &&+C_2(T)\int_0^t\io v_{\eps }^2+C_2(T)\int_0^t\io u_{\eps }^2+C_2(T)\int_0^t\io u_{\eps z}^2+\|b_{\eps t}\|_{L^\infty(\Omega)}\int_0^t\io \Theta_\eps
        \eea
        for all $t\in(0,\tme)\cap(0,T)$ and $\eps\in(0,1)$.
    This already implies \eqref{t8} and \eqref{t9} due to the nonnegativity of $\Theta_\eps$ and due to \eqref{t5}-\eqref{t7}.
    \end{proof}
\end{lemma}
 
\subsection{Global solvability of the approximate problems}
The goal of this section is to ensure that for any fixed $\eps\in(0,1)$, the first alternative of the extensibility criterion cannot occur. So, we look at higher regularity properties that might be dependent on $\eps$. The initial step in this process is to demonstrate that $\Theta_\eps$ is limited in $L^2(\Omega)$.
\begin{lemma}
    If $\tme<\infty$ for some $\eps\in(0,1)$, then there exists $C(\eps)>0$ such that
    \bea{tt0}
    \io\Theta_\eps^2\le C(\eps)\qquad\text{for all }t\in(0,\tme)
    \eea
    and
    \bea{tt01}
    \int_0^{\tme}\io\Theta_{\eps z}^2<\infty.
    \eea
    \begin{proof}
        Due to $\tme<\infty$ we can find $c_1=c_1(\eps)>0$ fulfilling
        \bea{tn2}
        c_1\le k_\eps(z,t)\qquad\text{for all }(z,t)\in\Om\times(0,\tme),
        \eea
        and due to \eqref{t5}, \eqref{t8} and the Gagliardo-Nirenberg inequality as well as Young's inequality, we can find $c_2>0$ and $c_3=c_3(\eps)>0$ such that
        \bea{GN00}
        \int_0^t\|v_{\eps z}\|_{L^3(\Omega)}^3\le c_2\int_0^t\bigg\{\io \|v_{\eps zz}\|_{L^2(\Om)}^\frac{7}{4}\|v_{\eps}\|_{L^2(\Om)}^\frac{5}{4} +\|v_\eps\|_{L^2(\Om)}^3\bigg\}\le c_3
        \eea
        for all $t\in(0,\tme)$. Testing the third equation of \eqref{approx} yields
        \bea{tt0.1}
        \frac{1}{2}\frac{d}{dt}\io b_\eps\Theta_\eps^2&=&-\io k_\eps\Theta_{\eps z}^2+\io\Geps(\Theta_\eps)v_{\eps z}^2\Theta_\eps-\beta\io\Theta_\eps^2 v_{\eps z}+\frac{1}{2}\io b_{\eps t}\Theta_\eps^2\nonumber\\
        &\le&-\io k_\eps\Theta_{\eps z}^2+ C_\Gamma\io v_{\eps z}^2\Theta_\eps-\beta\io \Theta_\eps^2 v_{\eps z}+\frac{1}{2}\io b_{\eps t}\Theta_\eps^2\nonumber\\
        &\le&-\io k_\eps\Theta_{\eps z}^2+(C_\Gamma+\beta)\io v_{\eps z}^3+(C_\Gamma+\beta+1)\io \Theta_\eps^3+\|b_{\eps t}\|_{L^\infty(\Omega)}^2\io\Theta_\eps
        \eea
        for all $t\in(0,\tme)$, whereby again by the Gagliardo-Nirenberg inequality, Young's inequality, \eqref{t7} and \eqref{tn2}
        \bea{tt0.11}
        (C_\Gamma+\beta+1)\io \Theta_\eps^3&\le& c_4\|\Theta_{\eps z}\|_{L^2(\Omega)}^\frac{4}{3}\|\Theta_\eps\|_{L^1(\Omega)}^\frac{5}{3}+c_4\|\Theta_\eps\|_{L^1(\Omega)}^3\nn\\
        &\le& \frac{c_1}{2}\|\Theta_{\eps z}\|_{L^2(\Om)}^2+c_5\|\Theta_\eps\|_{L^1(\Om)}^5+c_4\|\Theta_\eps\|_{L^1(\Omega)}^3\nonumber\\
        &\le &\frac{c_1}{2}\|\Theta_{\eps z}\|_{L^2(\Om)}^2+c_6\nn
        \eea
        for all $t\in(0,\tme)$, with some $c_4>0$, $c_5=c_5(\eps)>0$ and $c_6=c_6(\eps)>0$. In conjunction with \eqref{tt0.1}, and following an integration over time, this results directly in
        \bea{tt0.2}
        \frac{1}{2}\io b_\eps \Theta_\eps^2+\frac{1}{2}\int_0^t\io k_\eps\Theta_{\eps z}^2\le (C_\Gamma+\beta)\int_0^t\io v_{\eps z}^3+\int_0^t \|b_{\eps t}\|_{L^\infty(\Omega)}^2\io \Theta_\eps+\io b_\eps(\cdot,0)\Theta_{0\eps}^2+c_6t\nonumber
        \eea
        for all $t\in(0,\tme)$. Thus \eqref{tt0} and \eqref{tt01} are direct consequences of  \eqref{t7} and \eqref{GN00}.
    \end{proof}
\end{lemma}
\noindent This lets us now use the fourth-order artificial diffusion mechanism in the first equation in \eqref{approx}, as well as the second-order diffusion mechanism in the second equation in \eqref{approx}, to obtain the following.
\begin{lemma}\label{MassenAbsch2eps}
    If $\tme<\infty$ for some $\eps\in(0,1)$, then there exists $C(\eps)>0$ such that
    \bea{tt1}
    \io v_{\eps z}^2(\cdot,t)\le C(\eps)\qquad\text{for all }t\in(0,\tme)
    \eea
    and
    \bea{tt2}
    \io u_{\eps zz}^2(\cdot,t)\le C(\eps)\qquad\text{for all }t\in(0,\tme),
    \eea
    and moreover we have
    \bea{tt3}
    \int_0^{\tme}\io v_{\eps zzz}^2<\infty\quad{and}\quad \int_0^{\tme}\io u_{\eps zzz}^2<\infty.
    \eea
    \begin{proof}
        Due to \eqref{t6}, \eqref{tt0} and the fact that $\tme<\infty$ we can find $c_1=c_1(\eps)>0$ such that
        \bea{tt3.1}
        \io u_{\eps z}^2\le c_1\quad\text{and}\quad \io \Theta_\eps^2\le c_1\qquad\text{for all }t\in(0,\tme),
        \eea
        and because of our assumptions on $\rho_\eps$ we can find $c_2=c_2(\eps)>0$ and $c_3=c_3(\eps)>0$ such that  
        \bea{tn3}
        c_2\le \rho_\eps(z,t)\le c_3,\quad \rho_{\eps z}(z,t)\le c_3\quad\text{and}\quad p(z,t)\le c_3\quad\text{for all }(z,t)\in\Om\times(0,\tme).\nn
        \eea
        Relying on the smoothness properties of $(v_\eps,u_\eps,\Theta_\eps)$ recorded in Lemma \ref{2.1Ref}, we test the first equation of \eqref{approx} with $-\frac{v_{\eps zz}}{\rho_\eps}$ to see that for all $t\in(0,\tme)$
        \bea{tt4}
        \frac{1}{2}\frac{d}{dt}\io  v_{\eps z}^2+\eps\io \frac{1}{\peps}v_{\eps zzz}^2&=&\eps\io\frac{\rho_{\eps z}}{\peps^2}v_{\eps zzz}v_{\eps zz}+ \io \frac{\Geps(\Theta_\eps)}{\peps}v_{\eps z}v_{\eps zzz}-\io\frac{\rho_{\eps z}}{\peps^2}\Geps(\Theta_\eps)v_{\eps z}v_{\eps zz}\nn\\
        &&+\io \frac{p_{\eps}}{\peps}u_{\eps z}v_{\eps zzz}-\io\frac{\rho_{\eps z}}{\peps^2}p_{\eps}u_{\eps z}v_{\eps zz}+\beta\io\frac{1}{\peps} \Theta_{\eps z} v_{\eps zz}.
        \eea
        By means of Young's inequality we moreover get that
        \bea{tt5}
        \eps\io\frac{\rho_{\eps z}}{\peps^2}v_{\eps zzz}v_{\eps zz}\le\frac{\eps}{8 c_3}\io v_{\eps zzz}^2+ \frac{2\eps c_3^3}{c_2^4}\io v_{\eps zz}^2\qquad\text{for all }t\in(0,\tme),
        \eea
        and
        \bea{tt5ttt}
        &&\hspace{-10mm}
        \io\frac{\Geps(\Theta_\eps)}{\peps}v_{\eps z}v_{\eps zzz}-\io\frac{\rho_{\eps z}}{\peps^2}\Geps(\Theta_\eps)v_{\eps z}v_{\eps zz}\nn\\
        &\le& \frac{\eps }{8c_3}\io v_{\eps zzz}^2+\Big\{\frac{2c_3 C_\Gamma^2}{\eps c_2^2}+\frac{C_\Gamma c_3}{c_2^2}\Big\}\io v_{\eps z}^2+\frac{C_\Gamma c_3}{c_2^2}\io v_{\eps zz}^2\qquad\text{for all }t\in(0,\tme),
        \eea
        as well as
        \bea{tt6}
        \io \frac{p_{\eps}}{\peps}u_{\eps z} v_{\eps zzz}\le\frac{\eps }{8 c_3}\io v_{\eps zzz}^2+\frac{2 c_3^3}{\eps c_2^2}\io u_{\eps z}^2
        \eea
        for all $t\in(0,\tme)$, and
        \bea{tt6-7t}
        -\io\frac{\rho_{\eps z}}{\peps^2}p_{\eps}u_{\eps z}v_{\eps zz}\le \frac{c_3^2}{c_2^2}\io u_{\eps z}^2+\frac{c_3^2}{c_2^2}\io v_{\eps zz}^2
        \eea
        for all $t\in(0,\tme)$, as well as
        \bea{tt7}
        \io\frac{\beta}{\peps} \Theta_{\eps z}v_{\eps zz}&=&-\io \frac{\beta}{\peps}\Theta_\eps v_{\eps zzz}+\beta\io\frac{\rho_{\eps z}}{\peps^2}\Theta_\eps v_{\eps zz}\nn\\
        &\le&\frac{\eps}{8 c_3}\io v_{\eps zzz}^2+\Big\{\frac{2 c_3\beta^2}{\eps c_2^2}+\beta \frac{c_3}{c_2^2}\Big\}\io\Theta_\eps^2 +\beta \frac{c_3}{c_2^2}\io v_{\eps zz}^2\qquad\text{for all }t\in(0,\tme).
        \eea
        Combining \eqref{tt4}-\eqref{tt7} yields
        \bea{tt7-8t}
        \frac{1}{2}\frac{d}{dt}\io v_{\eps z}^2+\frac{\eps}{2}\io\frac{1}{\peps}v_{\eps zzz}^2\le c_4\io v_{\eps zz}^2+c_4\io v_{\eps z}^2+c_4\io u_{\eps z}^2+c_4\io\Theta_{\eps}^2
        \eea
        for all $t\in(0,\tme)$, with $$c_4\equiv c_4(\eps):=\max\left\{\frac{2\eps c_3^3}{c_2^4} +\frac{C_\Gamma c_3}{c_2^2}+\frac{c_3^2}{c_2^2}+\beta \frac{c_3}{c_2^2},\frac{2c_3 C_\Gamma^2}{\eps c2^2}+\frac{C_\Gamma c_3}{c_2^2}, \frac{2 c_3^3}{\eps c_2^2}+\frac{c_3^2}{c_2^2},\frac{2c_3\beta^2}{\eps c_2^2}+\beta \frac{c_3}{c_2^2}\right\}.$$ Due to the Gagliardo Nirenberg inequality, we can find $c_5=c_5(\eps)>0$ such that
        \bea{tt4-5}
        c_4\io v_{\eps zz}^2\le \frac{\eps}{4 c_3}\io v_{\eps zzz}^2+c_5\io v_{\eps z}^2\qquad\text{for all }t\in(0,\tme).
        \eea
        As furthermore $u_{\eps t}=\eps u_{\eps zz}+v_\eps$ on $\partial\Omega$ for all $t\in(0,\tme)$ and thus $u_{\eps zz}=0$ on $\partial\Omega$ for all $t\in(0,\tme)$, we obtain
        \bea{tt8}
        \frac{1}{2}\frac{d}{dt}\io u_{\eps zz}^2+\eps \io u_{\eps zzz}^2=-\io v_{\eps z}u_{\eps zzz} \le \frac{\eps}{2}\io u_{\eps zzz}^2+\frac{1}{2\eps}\io v_{\eps z}^2
        \eea
        for all $t\in(0,\tme)$ by the second equation in \eqref{approx}, from \eqref{tt3.1} and \eqref{tt7-8t}-\eqref{tt8} we infer that $y(t):=\io v_{\eps z}^2+\io u_{\eps zz}^2+1$, $t\in[0,\tme)$, satisfies
        \bea{tt9}
        y'(t)+\frac{\eps}{2}\io\frac{1}{\peps} v_{\eps zzz}^2+\eps\io u_{\eps zzz}^2\le c_6 y(t)\qquad\text{for all }t\in(0,\tme)
        \eea
        with $c_6\equiv c_6(\eps):=\max\{2c_4+2c_5+\eps^{-1},4c_1c_4\}$. By an ODE comparison argument, this firstly implies \eqref{tt1} and \eqref{tt2}, and thereupon an integration in \eqref{tt9} shows that \eqref{tt3} holds.
    \end{proof}
\end{lemma}
\noindent Prior to refuting \eqref{ae2} for $\tme<\infty$, it is necessary to collect several additional regularities.
\begin{lemma}
    If $\eps\in(0,1)$ is such that $\tme<\infty$, then there exists $C(\eps)>0$ fulfilling
    \bea{tt10.1}
    \io \Theta_{\eps z}^2(\cdot,t)\le C(\eps)\qquad\text{for all }t\in(0,\tme)
    \eea
    and
    \bea{tt102}
    \int_0^{\tme}\io \Theta_{\eps zz}^2<\infty.
    \eea
    \begin{proof}
        Due to the Gagliardo-Nirenberg inequality, we can find $c_1>0$ such that
        \bea{GNtt}
        \|\varphi\|_{L^4(\Om)}^4\le c_1\|\varphi_z\|_{L^2(\Om)}\|\varphi\|_{L^2(\Om)}^3+c_1\|\varphi\|_{L^2(\Om)}^4\qquad\text{for all }\varphi\in W^{1,2}(\Omega).
        \eea
        Furthermore there exist $c_2=c_2(\eps)>0$ and $c_3=c_3(\eps)>0$ such that
        $$c_2\le k_\eps(z,t)\le c_3\quad\text{and}\quad c_2\le b_{\eps}(z,t)\le c_3\quad\text{for all }(z,t)\in\Om\times(0,\tme).$$
        Testing the third equation of \eqref{approx} with $-\frac{\Theta_{\eps zz}}{b_\eps}$ yields
        \bea{tt10.2}
        \frac{1}{2}\frac{d}{dt}\io \Theta_{\eps z}^2+\io \frac{k_\eps}{b_\eps} \Theta_{\eps zz}^2=-\io \frac{1}{b_\eps}\big\{k_{\eps z}\Theta_{\eps z}+\Geps(\Theta_\eps) v_{\eps z}^2-\beta\Theta_{\eps}v_{\eps z}\big\}\cdot\Theta_{\eps zz}
        \eea
        for all $t\in(0,\tme)$. We can use Young's inequality to get the following results:
        \bea{tt10.3}
        -\io \frac{k_{\eps z}}{b_\eps} \Theta_{\eps z}\Theta_{\eps zz}\le \frac{c_2}{4 c_3}\io\Theta_{\eps zz}^2+\frac{c_3\|k_{\eps z}\|_{L^\infty(\Om)}^2}{c_2^3}\io \Theta_{\eps z}^2 \qquad\text{for all }t\in(0,\tme)
        \eea
        and
        \bea{tt10.4}
        -\io\frac{\Geps(\Theta_\eps)}{b_\eps}v_{\eps z}^2\Theta_{\eps zz}\le\frac{c_2}{4c_3}\io \Theta_{\eps zz}^2+\frac{c_3C_\Gamma^2}{c_2^3}\io v_{\eps z}^4\qquad\text{for all }t\in(0,\tme),
        \eea
        as well as
        \bea{tt10.5}
        \hspace{8mm}\io\frac{\beta}{b_\eps}\Theta_\eps v_{\eps z}\Theta_{\eps zz}\le\frac{c_2}{4c_3}\io \Theta_{\eps zz}^2+\frac{c_3\beta^2}{c_2^3}\io v_{\eps z}^4+\frac{c_3\beta^2}{c_2^3}\io \Theta_\eps^4\qquad\text{for all }t\in(0,\tme).
        \eea
        Combining \eqref{tt10.2}-\eqref{tt10.5}, we obtain 
        \bea{tt10.6}
        \frac{d}{dt}\io  \Theta_{\eps z}^2+\frac{1}{2}\io \frac{k_\eps}{b_\eps}\Theta_{\eps zz}^2\le c_4\io \Theta_{\eps z}^2+c_4\io v_{\eps z}^4+c_4\io \Theta_\eps^4\qquad\text{for all }t\in(0,\tme)\hspace{-10mm}
        \eea
        with $c_4\equiv c_4(\eps):=\max\big\{\frac{2c_3\|k_{\eps z}\|_{L^\infty(\Om)}^2}{c_2^3},\frac{2c_3C_\Gamma^2}{c_2^3}+\frac{2c_3\beta^2}{c_2^3}\big\}$.
        Due to \eqref{GNtt}, \eqref{tt0}, \eqref{tt1} and Young's inequality, we can find $c_5=c_5(\eps)>0$ and $c_6=c_6(\eps)>0$ fulfilling 
        \bea{tt10.7}
        c_4\io v_{\eps z}^4+c_4\io \Theta_\eps^4&\le& c_5\|v_{\eps zz}\|_{L^2(\Om)}+c_5\|\Theta_{\eps z}\|_{L^2(\Omega)}+c_5\nonumber\\
        &\le& c_6\io v_{\eps zz}^2+c_6\io\Theta_{\eps z}^2+c_6\qquad\text{for all }t\in(0,\tme).
        \eea
        In conjunction with \eqref{tt10.6}, and following an integration over time, this results directly in
        \bea{tt10.71}
        \io \Theta_{\eps z}^2+\frac{1}{2}\int_0^t\io \frac{k_\eps}{b_\eps}\Theta_{\eps zz}^2\le(c_4+c_6)\int_0^t\io\Theta_{\eps z}^2+c_6\int_0^t\io v_{\eps zz}^2+c_6t+\io \Theta_{0\eps z}^2\nonumber
        \eea
        for all $t\in(0,\tme)$. Now recalling \eqref{t8} and \eqref{tt01} we directly obtain  \eqref{tt10.1} and \eqref{tt102}.
    \end{proof} 
\end{lemma}
\noindent We are now collecting the first final regularities, and in our case even a little more, to gradually show the global existence.
\begin{lemma}\label{MassenAbsch3eps}
    Suppose that $\tme<\infty$ for some $\eps\in(0,1)$. Then one can find $C(\eps)>0$ such that
    \bea{tt11}
    \|\Theta_\eps(\cdot,t)\|_{W^{2,2}(\Omega)}\le C(\eps)\qquad\text{for all }t\in(0,\tme).
    \eea
    \begin{proof}
        Prior to embarking upon the proof, it is necessary to verify that $\Theta_{\eps zzz}=0$ on $\partial\Om$ for all $t\in(0,\tme)$. By the third equation of \eqref{approx} we directly obtain that
        \bea{thetarand}
        b_{\eps z}\Theta_{\eps t}+b_\eps\Theta_{\eps zt}&=&k_\eps\Theta_{\eps zzz}+2k_{\eps z}\Theta_{\eps zz}+k_{\eps zz}\Theta_{\eps z}+\Geps'(\Theta_\eps)\Theta_{\eps z}v_{\eps z}^2+2\Geps(\Theta_\eps)v_{\eps z}v_{\eps zz}\nonumber\\
        &&-\beta\Theta_{\eps z}v_{\eps z}-\beta\Theta_\eps v_{\eps zz}\qquad\text{in }\partial\Omega\times(0,\tme).
        \eea
        However, due to $v_{\eps zz}=\Theta_{\eps z}=0$ on $\partial\Om$ for all $t\in(0,\tme)$ and to $k_\eps>0$ in $\overline\Om$ for all $t\in(0,\tme)$, because of \eqref{ab4}, and with the assistance of \eqref{ab6} we can directly verify that all terms, with the exception of $k_\eps\Theta_{\eps zzz}$, are equivalent to zero, which means that the additional boundary condition follows.\\
        Based on the Gagliardo-Nirenberg inequality, we can find $c_1>0$ such that
        \bea{GNt1}
        \|\varphi\|_{L^6(\Omega)}^6\le c_1\|\varphi_z\|_{L^2(\Omega)}^2\|\varphi\|_{L^2(\Omega)}^4+c_1\|\varphi\|_{L^2(\Omega)}^6\qquad\text{for all }\varphi\in W^{1,2}(\Omega)
        \eea
        and
        \bea{GNt2}
        \|\varphi_z\|_{L^3(\Omega)}^3\le c_1\|\varphi_{zz}\|_{L^2(\Omega)}^\frac{7}{4}\|\varphi\|_{L^2(\Omega)}^\frac{5}{4}+c_1\|\varphi\|_{L^2(\Omega)}^3\qquad\text{for all }\varphi\in W^{2,2}(\Omega).
        \eea
        We also note, that $\big|\frac{b_{\eps z}}{b_\eps^2}|$ is bounded in $\Om\times(0,\tme)$ and once more that we can find $c_2=c_2(\eps)>0$ and $c_3=c_3(\eps)>0$ fulfilling
        $$c_2\le k_\eps(z,t)\le c_3\quad\text{and}\quad c_2\le b_{\eps}(z,t)\le c_3\quad\text{for all }(z,t)\in\Om\times(0,\tme),$$
        due to the assumption that $\tme$ is finite.
        Testing the third equation of \eqref{approx} with $\Theta_{\eps zzzz}$ and integrating by parts with respect to $\Theta_{\eps zzz}=0$ on $\partial\Om$ for all $t\in(0,\tme)$ gives 
        \bea{tt12}
        \frac{1}{2}\frac{d}{dt}\io \Theta_{\eps zz}^2
        &=& -\io \Big\{\frac{1}{b_\eps}\big((k_\eps \Theta_{\eps z})_{z}+\Geps(\Theta_\eps)v_{\eps z}^2-\beta\Theta_{\eps}v_{\eps z}\big)\Big\}_{z}\cdot\Theta_{\eps zzz}\nonumber\\
        &=& -\io \frac{k_\eps}{b_\eps}\Theta_{\eps zzz}^2-\io  \Big(2\frac{k_{\eps z}}{b_\eps}-\frac{b_{\eps z}k_\eps}{b_\eps^2}\Big)\Theta_{\eps zz}\Theta_{\eps zzz}-\io \Big(\frac{k_{\eps zz}}{b_\eps}-\frac{b_{\eps z}k_{\eps z}}{b_\eps^2}\Big)\Theta_{\eps z}\Theta_{\eps zzz}\nonumber\\
        &&-\io \frac{\Geps'(\Theta_\eps)}{b_\eps}\Theta_{\eps z}v_{\eps z}^2\Theta_{\eps zzz}-2\io \frac{\Geps(\Theta_\eps)}{b_\eps}v_{\eps z}v_{\eps zz}\Theta_{\eps zzz}+\io\frac{b_{\eps z}\Geps(\Theta_\eps)}{b_\eps^2}v_{\eps z}^2\Theta_{\eps zzz}\nonumber\\
        &&+\io\frac{\beta}{b_\eps} \Theta_{\eps z}v_{\eps z}\Theta_{\eps zzz}+\io\frac{\beta}{b_\eps} \Theta_{\eps}v_{\eps zz}\Theta_{\eps zzz}-\io\frac{b_{\eps z}\beta}{b_\eps^2}\Theta_\eps v_{\eps z}\Theta_{\eps zzz}
        \eea
        for all $t\in(0,\tme)$, where Young's inequality gives us $c_4=c_4(\eps)>0$ and $c_5=c_5(\eps)>0$ such that
        \bea{tt13}
        &&\hspace{-10mm}
        -\io  \Big(2\frac{k_{\eps z}}{b_\eps}-\frac{b_{\eps z}k_\eps}{b_\eps^2}\Big)\Theta_{\eps zz}\Theta_{\eps zzz}-\io \Big(\frac{k_{\eps zz}}{b_\eps}-\frac{b_{\eps z}k_{\eps z}}{b_\eps^2}\Big)\Theta_{\eps z}\Theta_{\eps zzz}\nonumber\\
        &\le& \frac{c_2}{8c_3}\io\Theta_{\eps zzz}^2+c_4\io \Theta_{\eps zz}^2+\frac{c_2}{8c_3}\io\Theta_{\eps zzz}^2+c_4\io \Theta_{\eps z}^2\nonumber\\
        &\le& \frac{c_2}{4c_3}\io\Theta_{\eps zzz}^2+c_5\io \Theta_{\eps zz}^2+c_5\io \Theta_{\eps z}^6+c_5\quad\text{for all }t\in(0,\tme)
        \eea
        and
        \bea{tt14}
        &&\hspace{-10mm}
        -\io \frac{\Geps'(\Theta_\eps)}{b_\eps}\Theta_{\eps z}v_{\eps z}^2\Theta_{\eps zzz}-2\io \frac{\Geps(\Theta_\eps)}{b_\eps}v_{\eps z}v_{\eps zz}\Theta_{\eps zzz}+\io\frac{b_{\eps z}\Geps(\Theta_\eps)}{b_\eps^2}v_{\eps z}^2\Theta_{\eps zzz}\nonumber\\
        &\le& \frac{c_2}{4c_3}\io \Theta_{\eps zzz}^2+c_4\io \Theta_{\eps z}^2v_{\eps z}^4+c_4\io v_{\eps z}^2v_{\eps zz}^2+c_4\io v_{\eps z}^4\hspace{-10mm}\nonumber\\
        &\le& \frac{c_2}{4c_3}\io \Theta_{\eps zzz}^2+c_5\io \Theta_{\eps z}^6+c_5\io v_{\eps z}^6+c_5\io v_{\eps zz}^3+c_5\quad\text{for all }t\in(0,\tme)
        \eea
        as well as
        \bea{tt15}
        &&\hspace{-10mm}
        \io\frac{\beta}{b_\eps} \Theta_{\eps z}v_{\eps z}\Theta_{\eps zzz}+\io\frac{\beta}{b_\eps} \Theta_{\eps}v_{\eps zz}\Theta_{\eps zzz}-\io\frac{b_{\eps z}\beta}{b_\eps^2}\Theta_\eps v_{\eps z}\Theta_{\eps zzz}\nonumber\\
        &\le& \frac{c_2}{4c_3}\io\Theta_{\eps zzz}^2+c_4\io \Theta_{\eps z}^2v_{\eps z}^2+c_4\io \Theta_{\eps}^2v_{\eps zz}^2+c_4\io \Theta_{\eps}^2v_{\eps z}^2\hspace{-10mm}\nonumber\\
        &\le&\frac{c_2}{4c_3} \io\Theta_{\eps zzz}^2+c_5\io\Theta_{\eps z}^6+c_5\io v_{\eps z}^6+c_5\io \Theta_\eps^6+c_5\io v_{\eps zz}^3+c_5
        \eea
        for all $t\in(0,\tme)$. From combining \eqref{tt12}-\eqref{tt15}, we immediately get the following:
        \bea{tt16}
        &&\hspace{-10mm}
        \frac{d}{dt}\io \Theta_{\eps zz}^2+\frac{1}{2}\io \frac{\keps}{b_\eps}\Theta_{\eps zzz}^2\nonumber\\
        &\le& c_6\io \Theta_{\eps zz}^2+c_6\io \Theta_{\eps z}^6+c_6\io\Theta_\eps^6+c_6\io v_{\eps zz}^3+c_6\io v_{\eps z}^6+c_6
        \eea
        for all $t\in(0,\tme)$, with $c_6\equiv c_6(\eps):=3c_5$. Due to \eqref{GNt1}, \eqref{GNt2}, \eqref{tt0}, \eqref{tt1} and \eqref{tt10.1}, we can find $c_7=c_7(\eps)>0$ such that
        \bea{tt17}
        c_6\io \Theta_{\eps z}^6\le c_1c_6\io \Theta_{\eps zz}^2\Big\{\io \Theta_{\eps z}^2\Big\}^2+c_1c_6\Big\{\io \Theta_{\eps z}^2\Big\}^3 \le c_7\io\Theta_{\eps zz}^2+c_7
        \eea
        and
        \bea{tt18}
        c_6\io \Theta_{\eps }^6\le c_1c_6\io \Theta_{\eps z}^2\Big\{\io \Theta_{\eps }^2\Big\}^2+c_1c_6\Big\{\io \Theta_{\eps }^2\Big\}^3 \le c_7
        \eea
        as well as
        \bea{tt19}
        c_6\io v_{\eps zz}^3\le c_1c_6\|v_{\eps zzz}\|_{L^2(\Om)}^\frac{7}{4}\|v_{\eps z}\|_{L^2(\Om)}^\frac{5}{4}+c_1c_6\|v_{\eps z}\|_{L^2(\Om)}^3\le c_7\io v_{\eps zzz}^2+c_7
        \eea
        and
        \bea{tt20}
        c_6\io v_{\eps z}^6\le c_1c_6\io v_{\eps zz}^2\Big\{\io v_{\eps z}^2\Big\}^2+c_1c_6\Big\{\io v_{\eps z}^2\Big\}^3 \le c_7\io v_{\eps zz}^2+c_7
        \eea
        for all $t\in(0,\tme)$. It is evident that the combination of \eqref{tt16}-\eqref{tt20}, when integrated over time, results in
        \bea{tt21}
        &&\hspace{-10mm}
        \io \Theta_{\eps zz}^2+\frac{1}{2}\int_0^t\io \frac{\keps}{b_\eps}\Theta_{\eps zzz}^2\nonumber\\
        &\le& (c_6+c_7)\int_0^t\io \Theta_{\eps zz}^2+c_7\int_0^t\io v_{\eps zzz}^2+c_7\int_0^t\io v_{\eps zz}^2+(c_6+4c_7)t+\io \Theta_{0\eps zz}^2\nonumber
        \eea
        for all $t\in(0,\tme)$. Now recalling \eqref{t8}, \eqref{tt3} and \eqref{tt102} we directly obtain \eqref{tt11}.
    \end{proof}
\end{lemma}
\noindent It can thus be concluded that this is sufficient to ensure bounds for $v_\eps$.
\begin{lemma}\label{lem3.3ref}
    If $\eps\in(0,1)$ is such that $\tme<\infty$, then there exists $C(\eps)>0$ fulfilling
    \bea{tt31}
    \|v_\eps(\cdot,t)\|_{W^{3,2}(\Omega)}\le C(\eps)\qquad\text{for all }t\in(0,\tme).
    \eea
    \begin{proof}
    First we need another boundary condition. By the first equation of \eqref{approx} we obtain that
    \bea{t-1}
    \peps v_{\eps t}=-\eps v_{\eps zzzz}+(\Geps(\Theta)v_{\eps z})_z+(p_{\eps}u_{\eps z})_z-\beta\Theta_{\eps z}
    \eea
    in $\partial\Om\times(0,\tme)$. Due to our boundary assumptions in \eqref{approx}, \eqref{ab7} and $u_{\eps zz}=0$ on $\partial\Omega$ for all $t\in(0,\tme)$, and analogous to the previous lemma, all terms except of $\eps v_{\eps zzzz}$ vanish. So $v_{\eps zzzz}=0$ on $\partial\Omega\times(0,\tme)$.\\
    Based on the Gagliardo-Nirenberg inequality, we can find $c_1>0$ such that
    \bea{t-GN1}
    \|\varphi_z\|_{L^2(\Om)}^2\le c_1\|\varphi_{zz}\|_{L^2(\Om)}\|\varphi\|_{L^2(\Om)}+c_1\|\varphi\|_{L^2(\Om)}^2\qquad\text{for all }\varphi\in W^{2,2}(\Om).
    \eea
    Due to $\tme<\infty$, we can also find $c_2=c_2(\eps)>0$ so that
    \bea{t-9}
    \Big|\frac{\rho_{\eps z}(z,t)}{\peps^2(z,t)}\Big|\leq c_2\quad\text{and}\quad \rho_\eps(z,t)\le c_2\qquad\text{for all }(z,t)\in\Om\times(0,\tme).
    \eea
    We let
    \bea{t-9.1}
    \hspace{5mm}h_\eps:=\Geps(\Theta_\eps)v_{\eps zz}+\Geps'(\Theta_\eps)\Theta_{\eps z}v_{\eps z}+p_{\eps}u_{\eps zz}+p_{\eps z}u_{\eps z}-\beta\Theta_{\eps z}\qquad\text{for all }t\in(0,\tme)
    \eea
    and observe, that Lemma \ref{MassenAbsch3eps}, together with the continuity of the embedding $W^{2,2}(\Om)\hookrightarrow W^{1,\infty}(\Om)$ and $\tme<\infty$ ensure the existence of $c_3=c_3(\eps)>0$ satisfying
        \bea{tt35}
        \|h_\eps\|_{L^2(\Om)}\le c_3\big\{\|v_{\eps zz}\|_{L^2(\Om)}+\|v_{\eps z}\|_{L^2}+\|u_{\eps zz}\|_{L^2(\Om)}+\|u_{\eps z}\|_{L^2(\Om)}+1\big\}\qquad\text{for all }t\in(0,\tme).\nonumber
        \eea
        Because of \eqref{t6} and Lemma \ref{MassenAbsch2eps} we can find $c_4=c_4(\eps)>0$ to further estimate the above to
        \bea{tt35.5}
        \|h_\eps\|_{L^2(\Om)}\le c_3\|v_{\eps zz}\|_{L^2(\Om)}+c_4\qquad\text{for all }t\in(0,\tme).
        \eea
    With \eqref{t-GN1} and once again Lemma \ref{MassenAbsch2eps} we can find $c_5=c_5(\eps)>0$ to get
    \bea{t-9.2}
    \|h_\eps\|_{L^2(\Om)}^2\le c_5\|v_{\eps zzz}\|_{L^2(\Om)}^2+c_5\qquad\text{for all }t\in(0,\tme).
    \eea
    Analogue, again with the help of Lemma \ref{MassenAbsch3eps}, together with the continuity of the embedding $W^{2,2}(\Om)\hookrightarrow W^{1,\infty}(\Om)$, we can also estimate $h_{\eps z}$ in $L^2(\Om)$. It obviously follows that a $c_6=c_6(\eps)>0$ exists, so that
    \bea{t-9.3}
    \|h_{\eps z}\|_{L^2(\Om)}\le c_6\big\{\|v_{\eps zz}\|_{L^2(\Om)}+\|v_{\eps zzz}\|_{L^2(\Om)}+\|v_{\eps z}\|_{L^2(\Om)}+\|\Theta_{\eps zz}v_{\eps z}\|_{L^2(\Om)}\nn\\
    +\|u_{\eps zzz}\|_{L^2(\Om)}+\|u_{\eps zz}\|_{L^2(\Om)}+\|u_{\eps z}\|_{L^2(\Om)}+1\big\}
    \eea
    for all $t\in(0,\tme)$. So we can continue to estimate here like before and find $c_7=c_7(\eps)>0$ such that
    \bea{t-9.4}
    \|h_{\eps z}\|_{L^2(\Om)}^2\le c_7\|v_{\eps zzz}\|_{L^2(\Om)}^2+c_7\|u_{\eps zzz}\|_{L^2(\Om)}^2+c_7\qquad\text{for all }t\in(0,\tme).
    \eea
    By direct derivation and integrating by parts we obtain
    \bea{t-10}
    \frac{1}{2}\frac{d}{dt}\io v_{\eps zzz}^2&=&\io v_{\eps tz}\partial_z^5v_{\eps }\nn\\
    &=& -\eps\io \Big(\frac{v_{\eps zzzz}}{\peps}\Big)_z\partial_z^5 v_\eps +\io\Big(\frac{h_\eps}{\peps}\Big)_z\partial_z^5 v_\eps\qquad\text{for all }t\in(0,\tme)
    \eea
    where we can estimate
    \bea{t-11}
    -\eps\io \Big(\frac{v_{\eps zzzz}}{\peps}\Big)_z\partial_z^5 v_\eps&=&-\eps\io\frac{1}{\peps}(\partial_z^5 v_\eps)^2+\eps\io \frac{\rho_{\eps z}}{\peps^2}v_{\eps zzzz}\partial_z^5 v_\eps\nn\\
    &\le& -\frac{\eps}{c_2}\io(\partial_z^5 v_\eps)^2+\frac{\eps}{4c_2}\io(\partial_z^5 v_\eps)^2+\eps c_2^3\io v_{\eps zzzz}^2\nn\\
    &\le&-\frac{3\eps}{4c_2}\io(\partial_z^5 v_\eps)^2+\eps c_1 c_2^3\Big\{\|\partial_z^5 v_\eps\|_{L^2(\Om)}\|v_{\eps zzz}\|_{L^2(\Om)}+\|v_{\eps zzz}\|_{L^2(\Om)}^2\Big\}\nn\\
    &\le&-\frac{\eps}{2c_2}\io(\partial_z^5 v_\eps)^2+c_8\io v_{\eps zzz}^2
    \eea
    for all $t\in(0,\tme)$ and
    \bea{t-12}
    \io\Big(\frac{h_\eps}{\peps}\Big)_z\partial_z^5 v_\eps&=&\io \frac{h_{\eps z}}{\peps}\partial_z^5 v_\eps-\io\frac{h_\eps \rho_{\eps z}}{\peps^2}\partial_z^5v_\eps\nn\\
    &\le& \frac{\eps}{2c_2}\io(\partial_z^5 v_\eps)^2+c_8\io h_{\eps z}^2+c_8\io h_{\eps}\qquad\text{for all }t\in(0,\tme)
    \eea
    with some $c_8=c_8(\eps)>0$. Combining \eqref{t-9.2}, \eqref{t-9.4}, \eqref{t-10}, \eqref{t-11} and \eqref{t-12} results in
    \bea{t-13}
    \frac{1}{2}\frac{d}{dt}\io v_{\eps zzz}^2\le c_9\io v_{\eps zzz}^2+c_9\io u_{\eps zzz}^2+c_9\qquad\text{for all }t\in(0,\tme)
    \eea
    with $c_9:=c_8(c_5+c_7+1)$. With a simple integration over time and with \eqref{tt3} our statement follows.
    \end{proof}
\end{lemma}

\noindent In conclusion, it can be positioned that the aforementioned approximate solutions to \eqref{approx} can in fact be extended so as to exist globally:
\begin{lemma}
For each $\eps\in(0,1)$, the solution of \eqref{approx} is global in time; that is, in Lemma \ref{2.1Ref} we have $\tme=\infty$.
\begin{proof}
    As $W^{2,2}(\Om)\hookrightarrow W^{\sigma,\infty}(\Om)$ whenever $\sigma\in(1,\frac{3}{2})$ and $W^{3,2}(\Om)\hookrightarrow W^{\delta,\infty}(\Om)$ whenever $\delta\in(2,\frac{5} {2})$ (\cite{henry1981geometric}), Lemma \ref{lem3.3ref} in conjunction with \eqref{tt2} and Lemma \ref{MassenAbsch3eps} shows that if $\tme$ was finite for some $\eps\in(0,1)$, then \eqref{ae2} would be violated. The claim thus results from Lemma \ref{2.1Ref}.
\end{proof}
\end{lemma}
\subsection{Further $\eps$-independent estimates. Construction of a limit triple}
\label{sec3.3}
Next addressing a key issue related to the construction of limit objects $v,\ u$ and $\Theta$ through an appropriate extraction of subsequences of solutions to \eqref{approx}. In order to achieve this objective, the initial aim is to establish a statement analogous to \eqref{tt0}, albeit this time independent of $\eps$. This is accomplished by first examining the functional $\io (\Theta_\eps+1)^p$ in their concave range when $p\in(0,1)$. We thereby obtain the following boundedness property of the quantities accordingly dissipated due to diffusion:

\begin{lemma}\label{ThetaGemischt}
   For any $p\in(0,1)$ and any $T>0$ there exists $C(p,T)>0$ such that
\bea{ThetaGemischt1}
\int_0^T \io (\Theta_\eps+1)^{p-2} \Theta_{\eps z}^2 \le C(p,T) \qquad \mbox{ for all } \eps \in (0,1).
\eea
\end{lemma}
\begin{proof}
    Due to the Gagliardo-Nirenberg inequality we can find $c_1=c_1(p)>0$ such that
    \bea{GN1}
    \big\|\varphi\big\|_{L^\frac{2(p+1)}{p}(\Omega)}^\frac{2(p+1)}{p}\le c_1\big\|\varphi_z\big\|_{L^2(\Omega)}^{\frac{2p}{p+1}}\big\|\varphi\big\|_{L^{\frac{2}{p}}(\Omega)}^\frac{4p+2}{p(p+1)}+c_1\big\|\varphi\big\|_{L^\frac{2}{p}(\Omega)}^\frac{2(p+1)}{p}\qquad\text{for all }\varphi\in W^{1,2}(\Omega)
    \eea
    and due to \eqref{ab3} and Assumptions \ref{vorrfunc} we can find $c_2=c_2(T)>0$ such that
    $$c_2\le k_\eps(z,t)\qquad\text{for all }(z,t)\in\Om\times(0,T).$$
    We once again use the third equation in \eqref{approx} along with the homogeneous Neumann boundary conditions for $\Theta_\eps$ to see on the basis of an integration by parts that
    \begin{align}\label{tz1}
        -\frac{1}{p}\frac{d}{dt}\io b_\eps(\Theta_\eps+1)^p&=-\io (\Theta_\eps+1)^{p-1}\cdot\big\{(k_\eps\Theta_{\eps z})_z+\Geps(\Theta_\eps) v_{\eps z}^2-\beta\Theta_\eps v_{\eps z}\big\}\nn\\
        &\quad -\frac{1}{p}\io b_{\eps t}(\Theta_\eps+1)^p\hspace{-5mm}\nonumber\\
        &=-(1-p)\io k_\eps(\Theta_\eps+1)^{p-2}\Theta_{\eps z}^2-\io(\Theta_{\eps}+1)^{p-1}\Geps(\Theta_\eps) v_{\eps z}^2\nonumber\\
        &\ +\beta\io (\Theta_\eps+1)^{p-1}\Theta_\eps v_{\eps z}-\frac{1}{p}\io b_{\eps t}(\Theta_\eps+1)^p\hspace{-5mm}
    \end{align}
    for all $t>0$ and $\eps\in(0,1)$, where using \eqref{ab1} we can estimate
    \bea{tz2}
    \io(\Theta_\eps+1)^{p-1}\Geps(\Theta_\eps)v_{\eps z}^2\ge c_\Gamma\io (\Theta_\eps+1)^{p-1}v_{\eps z}^2\qquad\text{for all }t>0\text{ and }\eps\in(0,1).
    \eea
    Due to \eqref{t7} we can find $c_3=c_3(p,T)>0$ such that with Young's inequality and again \eqref{ab1}, as well as by \eqref{GN1} we obtain
    \begin{align}\label{tz3}
    \beta\io (\Theta_\eps+1)^{p-1}\Theta_\eps v_{\eps z}&\le \beta\io (\Theta_\eps+1)^p |v_{\eps z}| \nn\\
    &\le c_\Gamma\io(\Theta_\eps+1)^{p-1}v_{\eps z}^2 +\frac{\beta^2}{4c_\Gamma}\Big\|(\Theta_\eps+1)^{\frac{p}{2}}\Big\|_{L^\frac{2(p+1)}{p}(\Omega)}^\frac{2(p+1)}{p}\nonumber\\
    &\le c_\Gamma\io(\Theta_\eps+1)^{p-1}v_{\eps z}^2 +\frac{\beta^2 c_1}{4c_\Gamma}\Big\|\frac{p}{2}(\Theta_\eps+1)^{\frac{p-2}{2}}\Theta_{\eps z}\Big\|_{L^2(\Omega)}^\frac{2p}{p+1}\Big\|(\Theta_\eps+1)^\frac{p}{2}\Big\|_{L^\frac{2}{p}(\Omega)}^\frac{4p+2}{p(p+1)}\nonumber\\
    & \ +\frac{\beta^2 c_1}{4c_\Gamma}\Big\|(\Theta_\eps+1)^\frac{p}{2}\Big\|_{L^\frac{2}{p}(\Omega)}^\frac{2(p+1)}{p}\nonumber\\
    &\le c_\Gamma\io(\Theta_\eps+1)^{p-1}v_{\eps z}^2+c_3\Big\|(\Theta_\eps+1)^{\frac{p-2}{2}}\Theta_{\eps z}\Big\|_{L^2(\Omega)}^\frac{2p}{p+1}+c_3
    \end{align}
    for all $t\in(0,T)$ and $\eps\in(0,1)$. As $2p<2(p+1)$, Young's inequality yields
    \bea{tz4}
    c_3\Big\|(\Theta_\eps+1)^{\frac{p-2}{2}}\Theta_{\eps z}\Big\|_{L^2(\Omega)}^\frac{2p}{p+1}\le \frac{c_2(1-p)}{2}\io(\Theta_\eps+1)^{p-2}\Theta_{\eps z}^2+c_4
    \eea
    for all $t\in(0,T)$ and $\eps\in(0,1)$, with a sufficiently large $c_4=c_4(p,T)>0$. When combined with \eqref{tz1}-\eqref{tz3}, this shows that
    \bea{tz5}
    \hspace{5mm}-\frac{1}{p}\frac{d}{dt}\io b_\eps(\Theta_\eps+1)^p+\frac{c_2(1-p)}{2}\io(\Theta_\eps+1)^{p-2}\Theta_{\eps z}^2\le c_3+c_4+\frac{\|b_{\eps t}\|_{L^\infty(\Omega)}}{p}\io(\Theta_\eps+1)^p\hspace{-10mm}
    \eea
    for all $t\in(0,T)$ and $\eps\in(0,1)$ and that thus
    \bea{tz6}
    \hspace{5mm}\frac{c_2(1-p)}{2}\int_0^T\io(\Theta_\eps+1)^{p-2}\Theta_{\eps z}^2\le c_5\io (\Theta_\eps(\cdot,T)+1)^p+c_5T+c_5\int_0^T\io (\Theta_\eps+1)^p
    \eea
    for all $T>0$ and $\eps\in(0,1)$ with $$c_5\equiv c_5(p,T):=\max\left\{\frac{\|b_{\eps}\|_{L^\infty((0,T)\times\Omega)}}p,c_3+c_4,\frac{\|b_{\eps t}\|_{L^\infty((0,T)\times\Omega)}}p\right\}.$$ Again using that $p<1$ in estimating
    \bea{tz7}
    &&\hspace{-10mm}
    c_5\io (\Theta_\eps(\cdot,T)+1)^p+c_5\int_0^T\io (\Theta_\eps+1)^p\nonumber\\
    &\le& c_5\cdot\bigg\{\io(\Theta_\eps(\cdot,T)+1)\bigg\}^p\cdot|\Omega|^{1-p}+c_5\cdot\int_0^T\bigg\{\io(\Theta_\eps+1)\bigg\}^p\cdot|\Omega|^{1-p}\quad\text{for all }\eps\in(0,1)\hspace{-5mm}
    \eea
    by the Hölder inequality, in view of \eqref{t7} we obtain \eqref{ThetaGemischt1} from \eqref{tz6} and \eqref{tz7}.
\end{proof}

\noindent Firstly, this suggests estimates for the temperature variable in some Lebesgue spaces, where the summability powers are conveniently large.

\begin{lemma}\label{ThetaHochq}
 For any $q\in (0,3)$ and any $T>0$ there exists $C(q,T)>0$ such that
 \bea{ThetaHochq1}
\int_0^T \io (\Theta_\eps+1)^q \le C(q,T) \qquad\mbox{ for all }\eps\in(0,1).
 \eea
 \begin{proof}
    Again due to the Gagliardo-Nirenberg inequality we can find $c_1=c_1(q)>0$ such that
    \bea{GN2}
    \big\|\varphi\big\|_{L^\frac{2q}{q-2}(\Omega)}^\frac{2q}{q-2}\le c_1\big\|\varphi_z\big\|_{L^2(\Omega)}^2\|\varphi\|_{L^\frac{2}{q-2}(\Omega)}^\frac{4}{q-2}+c_1\big\|\varphi\big\|_{L^\frac{2}{q-2}(\Omega)}^\frac{2q}{q-2}\qquad\text{for all }\varphi\in W^{1,2}(\Omega).
    \eea
    Due to \eqref{t7} we can find $c_2=c_2(q,T)>0$ such that
    \bea{tz9}
    \Big\|(\Theta_\eps+1)^\frac{q-2}{2}\Big\|_{L^\frac{2}{q-2}(\Omega)}^\frac{2}{q-2}=\io (\Theta_\eps+1)\le c_2\qquad\text{for all }t\in(0,T)\text{ and }\eps\in(0,1).
    \eea
     Without loss of generality assuming that $q\in(2,3)$, we apply Lemma \ref{ThetaGemischt} to $p:=q-2\in(0,1)$ to find $c_3=c_3(q,T)>0$ such that
     \bea{tz10}
    \int_0^T\io (\Theta_\eps+1)^{q-4}\Theta_{\eps z}^2\le c_3\qquad\text{for all }\eps\in(0,1).
\nn     \eea
     With \eqref{GN2} and \eqref{tz9} we obtain
     \bea{tz11}
    \io (\Theta_\eps+1)^q=\Big\|(\Theta_\eps+1)^\frac{q-2}{2}\Big\|_{L^\frac{2q}{q-2}(\Omega)}^\frac{2q}{q-2}\le c_1c_2^2\io(\Theta_\eps+1)^{q-4}\Theta_{\eps z}^2+c_1c_2^q\nonumber
     \eea
     for all $t\in(0,T)$ and $\eps\in(0,1)$ and this infer that
     \bea{tz12}
    \int_0^T\io (\Theta_\eps+1)^q\le c_1c_2^2\int_0^T\io (\Theta_\eps+1)^{q-4}\Theta_{\eps z}^2+c_1c_2^qT\le c_1c_2^2c_3+c_1c_2^qT\nonumber
     \eea
     for all $T>0$ and $\eps\in(0,1)$.
 \end{proof}
\end{lemma}
\noindent When combined with Lemma \ref{ThetaHochq} in the course of another interpolation, the weighted gradient estimate in Lemma \ref{ThetaGemischt} secondly implies bounds for the unweighted quantities $\Theta_{\eps z}$.
\begin{lemma}\label{ThetaZr}
    Let $r\in[1,\frac 32)$ and $T>0$, then one can find $C(r,T)>0$ such that
    \bea{ThetaZr1}
\int_0^T\io |\Theta_{\eps z}|^r \le C(r,T) \qquad \mbox{ for all }\eps \in(0,1).
    \eea
    \begin{proof}
        Since the inequality $r< \frac{3}{2}$ ensures that $\frac{5r-6}{r}<1$, we can pick $p=p(r)\in(0,1)$ such that $p>\frac{5r-6}{r}$, and then obtain that $(5-p)r<\big(5-\frac{5r-6}{r}\big)\cdot r=6$, meaning that $q\equiv q(r):=\frac{(2-p)r}{2-r}$ satisfies
        $$3-q=\frac{3(2-r)-(2-p)r}{2-r}=\frac{6-(5-p)r}{2-r}>0.$$
        As thus $q<3$, besides employing Lemma \ref{ThetaGemischt} we may therefore draw on Lemma \ref{ThetaHochq} to see that given $T>0$ we can find $c_1=c_1(r,T)>0$ and $c_2=c_2(r,T)>0$ such that
        \bea{tz13}
        \int_0^T\io (\Theta_\eps+1)^{p-2}\Theta_{\eps z}^2\le c_1\quad\text{and}\quad \int_0^T\io (\Theta_\eps+1)^q\le c_2\qquad\text{for all }\eps\in(0,1).
        \eea
        Since the fact that $r<2$ enables us to use Young's inequality to see that
        \bea{tz14}
        \int_0^T\io|\Theta_{\eps z}|^r&=&\int_0^T\io\big\{(\Theta_\eps+1)^{p-2}\Theta_{\eps z}^2\Big\}^\frac{r}{2}\cdot(\Theta_\eps+1)^\frac{(2-p)r}{2}\nonumber\\
        &\le& \int_0^T\io (\Theta_\eps+1)^{p-2}\Theta_{\eps z}^2+\int_0^T\io(\Theta_\eps+1)^\frac{(2-p)r}{2-r}\qquad\text{for all }\eps\in(0,1)
\nn        \eea
        in line with our definition of $q$ we conclude \eqref{ThetaZr1} from \eqref{tz13}.
    \end{proof}
\end{lemma}
\noindent The application of Lemma \ref{ThetaHochq} enables the control of the integral on the right-hand side of \eqref{vepsueps}, thus facilitating the derivation of a space-time estimate for the viscosity-related dissipation expressed therein.
\begin{lemma}\label{4.5ref}
    For all $T>0$ there exists $C(T)>0$ such that
    \bea{tg1}
    \int_0^T\io v_{\eps z}^2\le C(T)\qquad \mbox{ for all }\eps \in(0,1).\nn
    \eea
    \begin{proof}
        We once more return to Lemma \ref{vepsuepslem} to find a $c_1=c_1(T)>0$ such that
        \bea{tz20}
    &&\hspace{-10mm}
    \frac{1}{2}\frac{d}{dt}\io \rho_\eps v_\eps^2+\frac{1}{2}\frac{d}{dt}\io p_\eps  u_{\eps z}^2+\io \Gamma_\eps(\Theta_\eps) v_{\eps z}^2\nonumber\\
    &\le& \beta\io \Theta_\eps v_{\eps z}+c_1\io v_\eps^2+c_1\io u_{\eps z}^2\qquad\text{for all }t\in(0,T)\text{ and }\eps\in(0,1).\nn
\nn    \eea
    Since Young's inequality ensure that
    \bea{tz21}
    \beta\io \Theta_\eps v_{\eps z}\le\frac{c_\Gamma}{2}\io v_{\eps z}^2+\frac{\beta^2}{2c_\Gamma}\io \Theta_{\eps}^2\qquad\text{for all }t\in(0,T)\text{ and }\eps\in(0,1),
    \eea
    we thereby obtain that
    \bea{tz22}
    &&\hspace{-10mm}
    \frac{d}{dt}\bigg\{\io \peps v_\eps^2+\io p_\eps u_{\eps z}^2\bigg\}+c_\Gamma\io v_{\eps z}^2\nonumber\\
    &\le& \frac{\beta^2}{c_\Gamma}\io\Theta_\eps^2+2c_1\io v_\eps^2+2c_1\io u_{\eps z}^2\qquad\text{for all }t\in(0,T) \text{ and }\eps\in(0,1)
    \eea
    and hence
    \bea{tz23}
    c_\Gamma\int_0^T\io v_{\eps z}^2&\le& \io \peps(\cdot,0)v_{0\eps}^2+\io p_\eps(\cdot,0)u_{0\eps z}^2\nonumber\\
    &&+\frac{\beta^2}{c_\Gamma}\int_0^T\io \Theta_\eps^2+2c_1\int_0^T\io v_\eps^2+2c_1\int_0^T\io u_{\eps z}^2\qquad\text{for all }\eps\in(0,1).
    \eea
    Thus \eqref{tg1} follows directly from Lemma \ref{MassenAbsch} and Lemma \ref{ThetaHochq}.
    \end{proof}
\end{lemma}
\noindent In preparation for an Aubin-Lions type argument, it is also necessary to note the regularity properties of the time derivatives occurring in \eqref{approx}. The first two of these properties are discussed below.
\begin{lemma}\label{vepstDual} For any $T>0$, one can find $C(T)>0$ such that
\bea{uW1lam}
   \int_0^T \|v_{\eps t} (\cdot,t) \|_{(W_0^{2,2}(\Omega))^\star} \le C(T) \qquad \mbox{ for all }\eps\in(0,1)
\eea
and 
\bea{uepst}
\int_0^T \io u_{\eps t}^2 \le C(T) \qquad \mbox{ for all }\eps\in(0,1).
\eea
\begin{proof}
For fixed $\vp\in C^\infty_0(\Omega)$ fulfilling $\io \vp^2+\io \vp^2_z+\io \vp^2_{zz} \le 1$ from \eqref{approx} and the Cauchy Schwarz inequality as well as \eqref{ab1}-\eqref{ab5.1} we infer that 
\bea{Wslemma1}
&&\hspace{-10mm}
\left|\io \peps v_{\eps t} \vp \right|\nn\\
&=& \bigg| -\io \varepsilon v_{\varepsilon zz} \vp_{zz} -\io\Gamma_\varepsilon(\Theta_\varepsilon) v_{\varepsilon z}\vp_z-\io p_{\varepsilon} u_{\varepsilon z}\vp_z+\beta\io\Theta_{\eps}\vp_z\bigg| \nn\\
&=& \bigg| -\io  \varepsilon v_{\varepsilon zz} \vp_{zz} -\io\Gamma_\varepsilon(\Theta_\varepsilon) v_{\varepsilon z}\vp_z-\io p_{\varepsilon} u_{\varepsilon z}\vp_z +\beta\io\Theta_{\eps}\vp_z\bigg| \nn\\
&\le&  \eps\|v_{\eps zz}\|_{L^2(\Om)}+C_\Gamma\|v_{\eps z}\|_{L^2(\Om)}+\|p_\eps\|_{L^\infty(\Om)}\|u_{\eps z}\|_{L^2(\Om)}+\beta\|\Theta_\eps\|_{L^2(\Om)} \nn
\eea  
for all $t>0$ and $\eps\in(0,1)$, so that with some $c_1=c_1(T)>0$ we have
\begin{align*}
\int_0^T \| \peps v_{\eps t} (\cdot,t)\|_{(W^{2,2}_0(\Omega))^\star}^2 dt \le c_1\eps^2\int_0^T\io v_{\eps zz}^2+c_1\int_0^T\io v_{\eps z}^2+c_1\int_0^T\io u_{\eps z}^2+c_1\int_0^T\io \Theta_\eps^2
\end{align*}
for all $T>0$ and $\eps\in(0,1)$, because of \eqref{ab5.1} and Assumptions \ref{vorrfunc}. We may combine \eqref{t6}, \eqref{t8}, Lemma \ref{ThetaHochq} and Lemma \ref{4.5ref} to obtain $\eqref{uW1lam}$, again due to the positivity of $\peps(z,t)$ for all $(z,t)\in\Omega\times(0,\infty)$.
By collecting $\eqref{t5}$ and $\eqref{t9}$, the estimate \eqref{uepst} results from $\eqref{approx}$ and Young's inequality, since
$$u_{\eps t}^2 \le 2\eps^2 u_{\eps zz}^2 +2 v_\eps^2 \le 2\eps u_{\eps zz}^2+2v_\eps^2 \quad\mbox{ in }\Om\times(0,\infty) $$
for all $\eps\in(0,1)$.
\end{proof}
\end{lemma}
\noindent It is also the case that $\eps$-independent bounds apply to the time derivatives of the temperature variables in some suitable large spaces.
\begin{lemma}\label{W1lam}
    Let $\lambda>3$. Then for any $T>0$ there exists $C(\lambda,T)>0$ such that
   \bea{TW1lam}
   \int_0^T \|\Theta_{\eps t} (\cdot,t) \|_{(W^{1,\lambda}(\Omega))^\star} \le C(T) \qquad \mbox{ for all }\eps\in(0,1).
\eea
\begin{proof}
Using the embedding $W^{1,\lambda}(\Om)\hookrightarrow L^\infty(\Om)$, we can find $c_1=c_1(\lambda)>0$ such that $\|\vp_z\|_{L^\lambda(\Om)}+\|\vp\|_{L^\infty(\Om)}\le c_1$ for all $\vp\in C^1(\overline{\Om})$ fulfilling $\|\vp\|_{W^{1,\lambda}(\Om)}\le 1$. Therefore, by employing the third equation of \eqref{approx} along with the Hölder inequality, \eqref{ab1}, \eqref{ab4}, Young's inequality and the aforementioned principles, it is possible to derive the following equation
    \begin{align}\label{te-1}
    \bigg|\io\beps\Theta_{\eps t}\vp\bigg|&=\bigg|-\io\keps\Theta_{\eps z}\vp_z+\io\Geps(\Theta_\eps)v_{\eps z}^2\vp-\beta\io\Theta_\eps v_{\eps z}\vp\bigg|\nn\\
    &\le \|k_\eps\|_{L^\infty(\Om)}\|\Theta_{\eps z}\|_{L^\frac{\lambda}{\lambda-1}(\Om)}\|\vp_z\|_{L^\lambda(\Om)}+C_\Gamma\|\vp\|_{L^\infty(\Om)}\cdot\io v_{\eps z}^2+\beta\|\vp\|_{L^\infty(\Om)}\io\Theta_\eps|v_{\eps z}|\nn\\
    &\le c_1\|k_\eps\|_{L^\infty(\Om)}\io |\Theta_{\eps z}|^\frac{\lambda}{\lambda-1}+c_1(C_\Gamma+\beta)\io v_{\eps z}^2+c_1\beta\io\Theta_\eps^2+c_1\|k_\eps\|_{L^\infty(\Om)}|\Om|
    \end{align}
    for all $t>0$ and $\eps\in(0,1)$g. It is evident that $\frac{\lambda}{\lambda-1}\in(1,\frac{3}{2})$ since $\lambda>3$. Thus, in view of \eqref{ab3} and Assumptions \ref{vorrfunc}, our claim follows with an integration over time of \eqref{te-1} and with the help of Lemma \ref{ThetaHochq}, Lemma \ref{ThetaZr} and Lemma \ref{4.5ref}.
\end{proof}
\end{lemma}

\section{Passage to the limit and proof of Theorem \ref{EX}}\label{sec5}

\noindent Collecting our accomplishments, we are now able to take a limit along some subsequences to arrive at a limit object $(u,v,\Theta)$, which solves the first equation in \eqref{sys0} in the weak sense as noted in \eqref{DefWeak}.

\begin{remark}
In view of Remark \ref{RemarkShift}, the functions $v$, $u$ and $\Theta$ in \eqref{KGZEIG1} are understood after reversing the shift introduced therein; in particular, they satisfy the original boundary conditions of \eqref{system1} rather than homogeneous Dirichlet conditions prescribed in \eqref{approx}.
\end{remark}
\begin{lemma}\label{KGZ1terTeil}
There exist  $(\eps_j)_{j\in\N}\subset(0,1)$ and functions $(v,u,\Theta)$ such that  
\be{KGZEIG1}
    \lball
    v \in L^\infty((0,\infty);L^2(\Om))\cap L^2_{loc}([0,\infty);W^{1,2}(\Om)),\\
    u\in C^0(\overline{ \Omega}\times[0,\infty))\cap L^\infty((0,\infty); W^{1,2}(\Omega))\qquad and\\
    \Theta \in L^\infty((0,\infty);L^1(\Om))\cap\bigcap_{q\in[1,3)} L^q_{loc}(\bom \times [0,\infty))\cap \bigcap_{r\in[1,\frac32)} L^r_{loc} ([0,\infty);W^{1,r}(\Omega)),
    \ear
\ee
which satisfy $u(\cdot,0)=u_0$, $v(\cdot,0)=u_1$ and $\Theta\ge0$ a.e. in $\Om\times(0,\infty)$, that $\eps_j\searrow 0$ as $j\to \infty$.
\begin{alignat}{2}
v_\eps \hspace{1mm} &\to v \qquad &&\mbox{ a.e. in }\Om\times(0,\infty) \mbox{ and in } L^2_{loc}(\bom\times[0,\infty)),\label{vKgz}\\
v_{\eps z} &\rightharpoonup  v_z\qquad &&\mbox{ in }L^2_{loc} (\bom\times[0,\infty)),\label{vzKgz}\\
u_\eps &\to u \qquad &&\mbox{ in }L^2_{loc} (\bom \times[0,\infty)), \label{uKgz}\\
u_{\eps z} &\rightharpoonup u_z \qquad &&\mbox{ in }L^2_{loc}(\bom\times[0,\infty))\label{uzKgz}\\
 \Theta_\eps \ &\to \Theta \qquad &&\mbox{ a.e. in }\Om\times(0,\infty) \mbox{ and in } L^q_{loc}(\bom\times(0,\infty)) \mbox{ for all }q\in[1,3),\label{ThetaKgz}\\
  \Theta_{\eps z} &\rightharpoonup \Theta_z \qquad &&\mbox{ in } L^r_{loc}(\bom\times(0,T)) \mbox{ for all }r\in[1,3/2 ),\label{ThetazKgz}
\end{alignat}
 Furthermore,
\bea{SchwacheKGZ1}
  \int_0^\infty \int_\Omega \rho v \varphi_t&+&\int_0^\infty \io \rho_t v \vp-\io \rho(\cdot,0)v_{0}\vp(\cdot,0)\nn\\
  &=& \int_0^\infty\io \Big(\Gamma(\Theta)v_z+p u_z\Big)\varphi_z-\int_0^\infty\io \beta\Theta\vp_z,\hspace{-10mm}
\eea
and 
\bea{SchwacheKGZ2}
& &\int_0^\infty \io b_{\eps t}\Theta_\eps\vp + \int_0^\infty\io b_\eps \Theta_\eps \vp_t - \io b_\eps(\cdot,0)\Theta_{0 \eps}\vp(\cdot,0) \nn\\
&\to& \int_0^\infty \io b_t \Theta \vp +\int_0^\infty \io b \Theta \vp_t - \io b(\cdot,0) \Theta_0\vp(\cdot,0) \qquad \mbox{ as }\eps \searrow0,
\eea
as well as 
\bea{SchwacheKGZ3}
\int_0^\infty \io k_\eps \Theta_{\eps z} \vp_z - \int_0^\infty \io \beta \Theta_\eps v_{\eps z }\vp \to \int_0^\infty \io k \Theta_{ z} \vp_z - \int_0^\infty \io \beta \Theta v_{ z}\vp
\eea
as $\eps\searrow0$ for all $\vp\in C^\infty_0 (\bom \times[0,T))$, and 
\bea{KGZSubst}
u_t=v \qquad \mbox{ a.e. in }\Om\times(0,\infty).
\eea
\end{lemma}

\begin{proof}
By collecting Lemma \ref{MassenAbsch} and Lemma \ref{vepstDual} we infer
$$(v_\eps)_{\eps\in(0,1)} \mbox{ is bounded in }L^2((0,T);W_0^{1,2}(\Om))  \mbox{ for all } T>0,$$
and 
$$(v_{\eps t})_{\eps\in(0,1)} \mbox{ is bounded in }L^2((0,T);(W_0^{2,2}(\Om))^\star) \mbox{ for all } T>0,$$
while due to Lemma \ref{MassenAbsch} and Lemma \ref{vepstDual} we obtain
$$(u_\eps)_{\eps\in(0,1)} \mbox{ is bounded in }L^\infty((0,T);W_0^{1,2}(\Om)) \mbox{ for all }T>0,$$
    and that
    $$(u_{\eps t})_{\eps\in(0,1)} \mbox{ is bounded in }L^2(\Om\times(0,T)) \mbox{ for all }T>0.$$
    For $\Theta_\eps$ we conclude in a similar fashion from Lemma \ref{ThetaZr}, Lemma \ref{MassenAbsch} and Lemma \ref{W1lam}
    $$(\Theta_\eps)_{\eps\in(0,1)} \mbox{ is bounded in }L^r((0,T);W^{1,r}(\Om)) \mbox{ for all }T>0 \mbox{ and any }r\in\left(1,\frac32\right),$$
    and that
    $$(\Theta_{\eps t})_{\eps\in(0,1)} \mbox{ is bounded in }L^1((0,T);W^{1,\lambda}((\Om))^\star) \mbox{ for all }T>0 \mbox{ and each }\lambda>3.$$
    With respect to the compactness of the embeddings 
    $$W^{1,2}_0(\Om)\hookrightarrow L^2(\Om), \ W_0^{1,2}(\Om) \hookrightarrow C^0(\bom) \mbox{ and } W^{1,r}(\Om)\hookrightarrow L^1(\Om) \mbox{ for }r>1$$
    we utilize the Aubin-Lions lemma (\cite{temam2024}) three times to infer along a subsequence $(\eps_j)_{j\in\N}\subset(0,1)$ with $\eps_j\searrow0$ as $j\to\infty$  for our limit object 
    $$ v\in L^2_{loc}([0,\infty);W_0^{1,2}(\Om)), \qquad u\in C^0(\bom\times[0,\infty))\cap L^\infty((0,\infty);W^{1,2}_0(\Om)) $$
    and 
    $$\Theta\in \bigcap_{r\in(1,\frac32)} L^r_{loc} ([0,\infty);W^{1,r}(\Om)), $$
that \eqref{vKgz}-\eqref{uzKgz} and \eqref{ThetazKgz} hold and furthermore $\Theta_\eps\to\Theta$ a.e. in $\Om\times(0,\infty)$ as $\eps=\eps_j\searrow0$, whence also $u(\cdot,0)=u_0$, $\Theta(\cdot,0)=\Theta_0$ in $\Om$ and $\Theta\ge0$ a.e. in $\Om\times(0,\infty)$.
In view of Lemma \ref{MassenAbsch} and Lemma \ref{ThetaHochq} we further gather 
$$(v_\eps)_{\eps\in(0,1)} \mbox{ is bounded in }L^\infty((0,\infty);L^2(\Om)), $$
and 
$$(\Theta_\eps)_{\eps\in(0,1)} \mbox{ is bounded in }L^\infty((0,\infty);L^1(\Om))\mbox{ and in }L^1(\Om\times(0,T))\mbox{ for all }T>0 \mbox{ and each }q\in(1,3),$$
so the limit functions actually satisfy \eqref{KGZEIG1}. By applying Vitali's convergence theorem also \eqref{ThetaKgz} holds and the pointwise convergence part guarantees, combined with \eqref{KGZ_Gamma},
\bea{GammaKgz}\Gamma_\eps(\Theta_\eps)\to \Gamma(\Theta) \mbox{ in }L^2_{loc}(\bom\times[0,\infty)) \mbox{ as }\eps=\eps_j\searrow0.\eea
Now, we verify \eqref{SchwacheKGZ1} by fixing $\vp\in C^\infty_0(\bom\times[0,\infty))$, utilizing the first equation in \eqref{approx} and partial integration to obtain 
\begin{align*}
&\quad\int_0^\infty \io \rho_\eps v_\eps \vp_t
+ \int_0^\infty \io \rho_{\eps t} v_\eps \vp -  \io \rho_\eps(\cdot,0) v_{0\eps} \vp(\cdot,0)\\
&= \eps \int_0^\infty \io v_\eps \vp_{zzzz} +\int_0^\infty \io \Gamma(\Theta_\eps) v_{\eps z} \vp_z + \int_0^\infty \io p_{\eps} (z,t) u_{\eps z} \vp_z +\int_0^\infty\io \beta \Theta_\eps \vp_z \\
& \quad -\int_0^\infty \int_{\partial \Omega} \left[\Gamma(\Theta_\eps)v_{\eps z}+p_\eps u_{\eps z}+\beta\Theta_\eps\right]\vp\qquad\mbox{ for all }\eps\in(0,1),
\end{align*}
where in view to the boundary condition the  last term vanishes and due to \eqref{vKgz} and \eqref{KGZ_Rho}
$$\int_0^\infty \io v_\eps \vp_t \to  \int_0^\infty \io v \vp_t \qquad \mbox{ and } \qquad \int_0^\infty \rho_{\eps t} v_\eps \vp \to \int_0^\infty\io \rho_{ t}v \vp\quad \mbox{ as }\eps=\eps_j\searrow0$$
as well as 
$$\eps \int_0^\infty \io v_\eps \vp_{zzzz} \to 0 \qquad\mbox{ as }\eps=\eps_j\searrow0,$$
and due to \eqref{initkonv}, \eqref{uzKgz} as well as \eqref{KGZ_P}
$$ \io v_{0 \eps} \vp(\cdot,0) \to \io u_{0t} \vp(\cdot,0) \qquad \mbox{ and }\qquad \int_0^\infty\io p_{\eps} u_{\eps z}\vp_z \to \int_0^\infty \io p u_z \vp_z \qquad \mbox{ as }\eps=\eps_j\searrow0.  $$
Furthermore, with respect to \eqref{ThetaKgz}, \eqref{vzKgz} and \eqref{GammaKgz} we infer
$$ \int_0^\infty\io \beta \Theta_\eps \vp_z\to \int_0^\infty \io \beta\Theta \vp_z\qquad\mbox{ and }\qquad \int_0^\infty\io \Gamma_\eps(\Theta_\eps) v_{\eps z} \vp_z\to \int_0^\infty \io \Gamma(\Theta) v_z \vp_z $$
as $\eps=\eps_j\searrow0$ and thus we are able to conclude \eqref{SchwacheKGZ1} holds. Recalling the second equation in \eqref{approx}, we obtain for arbitrary $\vp\in C_0^\infty(\bom\times[0,\infty))$
$$-\int_0^\infty\io u_\eps \vp_t=\eps\int_0^\infty \int_{\partial\Omega} u_{\eps z}\vp-\eps \int_0^\infty \io u_{\eps z} \vp_{z} +\int_0^\infty\io  v_\eps \vp$$
and verify on letting $\eps=\eps_j\searrow0$ finally \eqref{KGZSubst}. To accomplish  \eqref{SchwacheKGZ2}, we note that due to $b_\eps$ satisfying \eqref{KGZ_B} and $\Theta_{0\eps}$ fulfilling \eqref{initkonv}
$$\io b_{\eps}(\cdot,0) \Theta_{0\eps} \vp(\cdot,0) \to \io b(\cdot,0) \Theta_{0} \vp(\cdot,0)  \qquad\mbox{ as }\eps=\eps_j\searrow0$$
and since we already derived \eqref{ThetaKgz}, we conclude from \eqref{KGZ_B}
$$\int_0^\infty \io b_{\eps t}\Theta_\eps\vp \to \int_0^\infty \io b_t \Theta \vp \qquad\mbox{ and }\qquad \int_0^\infty\io b_\eps \Theta_\eps \vp_t 
\to \int_0^\infty \io b \Theta \vp_t \qquad \mbox{ as }\eps=\eps_j \searrow0,$$
which yields \eqref{SchwacheKGZ2}. A combination of the already proven \eqref{ThetazKgz} and \eqref{KGZ_K} then yields \eqref{SchwacheKGZ3}.
\end{proof}

\noindent To derive the second equation of our Definition \ref{DefWeak}, we adapt a known technique, which has been used e.g. in \cite{Winklerweak1d}. We start by noting the following lemma without proof, since these basic properties of Steklov averages have been proven elsewhere e.g. \cite[Lemma 5.1]{Winklerweak1d}.

\begin{lemma}\label{SolutionProp}
Let $(u,v,\Theta)$ be as in Lemma \ref{KGZ1terTeil}, defining 
\bea{hatv}
\hat v(z,t):= \lball v(z,t), \qquad x\in\Om, \ t>0,\\
u_{0t}(z), \qquad z\in \Om,\ t<0,
\ear
\eea
as well as
\bea{hatu}
\hat u(z,t):= \lball u(z,t), \qquad \qquad \qquad z\in\Om, \  t>0,\\
u_{0}(z)+t u_{0t}(z),  \ \qquad z\in \Om, \ t<0.
\ear
\eea
and
\bea{Shvp}
(S_h \vp)(z,t) := \frac 1h \int_{t-h}^t \vp(z,s) ds, \qquad z\in\Om, \ t\in\R,\ h>0,\ \vp\in L^1_{loc}(\bom\times\R),
\eea
we have
\bea{DefS}
(S_h\hat v_z)(z,t)= \frac{\hat u_z(z,t)-\hat u_z(z,t-h)}{h} \qquad \mbox{ for a.e. }(z,t)\in\Om\times\R \mbox{ and each }h>0,
\eea
and
\bea{Shv}
S_h\hat v \rightharpoonup v \qquad \mbox{ in }L^2_{loc}(\bom\times[0,\infty)) \qquad\mbox{ as }h\searrow0,
\eea
as well as
\bea{Shvx}
S_h \hat v_z \rightharpoonup v_z \qquad \mbox{ in }L^2_{loc}(\bom \times [0,\infty))\qquad\mbox{ as }h\searrow0.
\eea
\end{lemma}
\noindent 
In preparation for our final boundary process, we derive an inequality from \eqref{SchwacheKGZ1} by restricting the boundary process to a suitable class of test functions. We will further restrict this class later on when our analysis is based on additional properties of those test functions.
\begin{lemma}\label{GesamtAbsch1}
    Let $(u,v,\Theta)$, $\hat v$ and $S_h$ be as in Lemma \ref{KGZ1terTeil} and \ref{SolutionProp}, and let $\xi \in C^\infty_0([0,\infty))$ be nonincreasing and such that $\xi(0)=1$. Then
    \bea{Xi1}
\int_0^\infty \io  \Gamma(\Theta(z,t)) \xi(t) v_z^2(z,t) dz dt &\ge& \frac12 \int_0^\infty \io \big(\rho_t(z,t) \xi(t)+\rho(z,t)\xi'(t)\big) v^2 (z,t) dz dt \nn\\
& &+\frac 12\io \rho(z,0) u_{0t}^2(z) dz +\frac12\int_0^\infty \io \beta\Theta(z,t) v_z(z,t) \xi(t) \nn\\
    & & -\limsup_{h\searrow0} \int_0^\infty\io p(z,t) u_z(z,t)\xi(t)(S_h \hat v_z) dz dt 
\eea
\end{lemma}

\begin{proof}
Firstly, due to \eqref{KGZEIG1} and \eqref{ab01}
\bea{gammaThetaL2}
\qquad \qquad\Gamma(\Theta)v_z \in L^2_{loc}(\bom\times[0,\infty)), \quad \beta\Theta \in L^2_{loc}(\bom\times[0,\infty)) \quad \mbox{ and } \quad u_z\in L^2_{loc}(\bom\times[0,\infty)).
\eea
By a standard approximation argument we thus conclude that \eqref{SchwacheKGZ1} even holds for each $\vp\in L^2((0,\infty);$ $W_0^{1,2}(\Om))$ which is such that $\vp_t\in L^2(\Om\times(0,\infty))$ and that $\vp=0$ a.e. on $\Om\times(T,\infty)$ for some $T>0$. It is therefore possible to utilize
$$\vp(z,t):= \xi(t)\cdot (S_h \hat v)(z,t), \qquad (z,t)\in\Om\times(0,\infty)$$
as a test function in \eqref{SchwacheKGZ1}, for which we compute
$$\vp_t(z,t)=\xi'(t) \cdot (S_h\hat v)(z,t) +\xi(t) \cdot \frac{\hat v(z,t)-\hat v(z,t-h)}{h} \qquad \mbox{a.e. on }\Om\times(0,\infty),$$
and 
$$\vp_z(z,t)=\xi(t) \cdot (S_h \hat v_z) (z,t) \qquad \mbox{a.e. on }\Om\times(0,\infty).$$
Thus, \eqref{SchwacheKGZ1} yields
 \bea{xiAnw1}
 \int_0^\infty\io \rho (z,t)   v(z,t)  \xi'(t) \cdot (S_h \hat v) (z,t)dzdt+\frac 1h\int_0^\infty \io  \rho(z,t) v(z,t) \xi(t) \cdot (\hat v(z,t)-\hat v(z,t-h))dzdt \nn\\
 + \int_0^\infty \io \rho_t(z,t) v(z,t) \xi(t) (S_h \hat v)(z,t)dzdt+\io \rho(z,0)u_{0t}(z)^2dz\nn\\
 =  \int_0^\infty \io \big(\Gamma(\Theta(z,t))v_z(z,t)+p(z,t)u_z(z,t)-\beta\Theta\big)\xi(t)\cdot (S_h \hat v_z)(z,t)dzdt,
 \eea
since $\xi(0)=1$ and $(S_h\hat v(z,0))=\frac 1h\int_{-h}^0 u_{0t}(z) dt = u_{0t}(z)$ for a.e. $z\in\Om$ in view of \eqref{hatv}. With respect to Assumptions \ref{vorrfunc}, \eqref{KGZEIG1} and \eqref{Shv}, we infer from  $ v \in L^2_{loc} (\bom\times[0,\infty))$,  $\rho\in C^2((0,T);C(\Om))$ and $\xi'(t)\in C^\infty_0([0,\infty))$
that 
\bea{KGZ_I1}\qquad \int_0^\infty\io \rho(z,t) \xi'(t) \cdot  (S_h\hat v)(z,t)dz dt \rightarrow \int_0^\infty \io \rho(z,t)\xi'(t) v^2(z,t) dz dt \quad \mbox{ as }h\searrow0.
\eea
In similar fashion, we conclude 
\bea{KGZ_I2}
\int_0^\infty \io \rho_t(z,t) v(z,t) \xi(t)(S_h \hat v)(z,t) dz dt \to \int_0^\infty \io \rho_t(z,t) \xi(t) v^2(z,t) dzdt
\eea
as $h\searrow0$, while from \eqref{gammaThetaL2}, we also obtain that
\bea{KGZ_I4}
\int_0^\infty \io \Gamma(\Theta(z,t)) v_z(z,t) (S_h \hat v_z)(z,t) dz dt \to 
\int_0^\infty \io \Gamma(\Theta(z,t))v_z^2(z,t) dz dt 
\eea
and
\bea{KGZI5}
\int_0^\infty \io \beta\Theta(z,t) (S_h \hat v_z)(z,t) \xi(t) dz dt \to 
\int_0^\infty\io \beta\Theta(z,t)v_z(z,t)\xi(t) dz dt 
\eea
as $h\searrow0$. We consider the left-hand side of \eqref{xiAnw1} once more, to infer by an application of Young's inequality and a linear substitution with respect to $\xi,\rho\ge0$,
\begin{align*}
    &\quad -\frac{1}{h} \int_0^\infty \io \rho(z,t)v(z,t) \cdot(\hat v(z,t)-\hat v(z,t-h)) dzdt \\
    &= \ -\frac{1}{h} \int_0^\infty \io \rho(z,t) \xi(t) v^2(z,t) dz dt + \frac 1h\int_0^\infty \io \rho(z,t) \xi(t) v(z,t) \hat v(z,t-h) dz dt \\
    &\le \ -  \frac{1}{2h} \int_0^\infty \io \rho(z,t) \xi(t) v^2(z,t) dz dt+\frac{1}{2h} \int_0^\infty \io \rho(z,t) \xi(t) \hat v^2(z,t-h) dz dt\\
    & \ = \ -  \frac{1}{2h} \int_0^\infty \io \rho(z,t) \xi(t) v^2(z,t) dz dt+\frac{1}{2h} \int_0^\infty \io \rho(z,t+h) \xi(t+h) v^2(z,t) dz dt \\
    & \quad +\frac 1{2h}\int_0^\infty \io \rho(z,t) \xi(t+h) v^2(z,t) dz dt -\frac 1{2h}\int_0^\infty \io \rho(z,t) \xi(t+h) v^2(z,t) dz dt  \\
   & \quad +\frac 1{2h} \int_{-h}^0 \io \rho(z,t+h) \xi(t+h) u^2_{0t}(z) dz dt \\
    &\ = \  \frac 12 \int_0^\infty \io \frac {\xi(t+h)-\xi(t)}h \rho(z,t) v^2(z,t)  + \frac 12\int_0^\infty \io \frac{\rho(z,t+h)-\rho(z,t)}{h}\xi(t+h) v^2(z,t) dz dt \\
    &\quad +\frac 1{2h}  \int_{-h}^0 \io \rho(z,t+h) \xi(t+h) u^2_{0t}(z) dz dt.
\end{align*}
Since it can easily be seen that $\frac{\xi(t+h)-\xi(t)}h \overset{\star}{\rightharpoonup} \xi'(t)$  and  $\frac{\rho(z,t+h)-\rho(z,t)}{h}\overset{\star}{\rightharpoonup} \rho_t(z,t)$ in $L^\infty((0,\infty))$ for all $z\in\Om$ as well as $\int_{-h}^0\xi(t+h)\rho(z,t+h) \to \xi(0)\rho(z,0)=\rho(z,0)$ for all $z\in\Om$ as $h\searrow0$.
Thus, we are able to conclude 
\begin{align*}
&\limsup_{h\searrow 0} \left\{ \frac1h \int_0^\infty \io \rho(z,t)\xi(t)\cdot(\hat v(z,t)-\hat v(z,t-h)) dz dt \right\}\\
&\ge -\frac 12\int_0^\infty \io \big(\rho_t(z,t) \xi(t)+\rho(z,t)\xi'(t)\big) v^2(z,t) -\frac12\io \rho(z,0)u_{0t}^2(z) dz 
\end{align*}
and further 
\begin{align*}
    \int_0^\infty \io \Gamma(\Theta(z,t)) \xi(t) v_z^2(z,t) dz dt
    &\ge \frac12 \int_0^\infty \io \big(\rho_t(z,t) \xi(t)+\rho(z,t)\xi'(t)\big) v^2 (z,t) dz dt\\
    & \quad +\frac 12\io \rho(z,0) u_{0t}^2(z) dz +\frac12 \int_0^\infty \io \beta\Theta(z,t) v_z(z,t)\xi(t)dzdt\\
    &\quad -\limsup_{h\searrow0} \int_0^\infty\io p(z,t) u_z(z,t) \xi(t)(S_h \hat v_z) dz dt,  
\end{align*}
which completes our proof.
\end{proof}
\noindent To estimate the last limit in our argument on the right hand side of \eqref{Xi1}, we take advantage of the definition of $S_h$ and the assumed regularity of $p\in C^1(\bom\times[0,\infty))$.
\begin{lemma}\label{GesamtAbsch2}
     Let $(u,v,\Theta)$, $\hat v$ and $S_h$ be as in Lemma \ref{KGZ1terTeil} and \ref{SolutionProp}, and let $\xi \in C^\infty_0([0,\infty))$  be such that $\xi'\le 0$ and $\xi(0)=1$, then
\begin{align}\label{KGZ_LAST_INT}
\limsup_{h\searrow0} &\int_0^\infty\io p(z,t) u_z(z,t)\xi(t)(S_h \hat v_z) dz dt \nn\\
&\le -\frac12 \int_0^\infty\io (\xi'(t) p(z,t)+\xi(t)p_t(z,t)) u_z^2(z,t)dz dt-\frac 12\io p(z,0) u_{0z}^2(z)dz 
\end{align}
\end{lemma}

\begin{proof}
In view of \eqref{hatu}  and \eqref{DefS}, we reformulate
\bea{LI1}
& &\int_0^\infty\io p(z,t) u_z(z,t)\xi(t)(S_h \hat v_z) dz dt \nn\\
& &\quad = \frac 1h \int_0^\infty \io \xi(t) p(z,t) \hat u_z(z,t)\cdot\left\{\hat u_z(z,t)-\hat u_z(z,t-h) \right\} dz dt\nn\\
& &\quad =\frac 1h \int_0^\infty \io \xi(t) p(z,t) \hat u_z(z,t-h) \cdot \left\{ \hat u_z(z,t)-\hat u_z(z,t-h)\right\}
dz dt\nn\\
& &\quad \quad + \frac 1h \int_0^\infty\io \xi(t) p(z,t) \left\{\hat u_z(z,t)-\hat u_z(z,t-h)\right\}^2 dz dt\nn\\
& &\quad= \frac1h \int_0^\infty \io \xi(t) p(z,t) \hat u_z(z,t-h) \cdot\left\{\hat u_z(z,t)-\hat u_z(z,t-h)\right\}dz dt\nn\\
& &\quad \quad +h\int_0^\infty \io \xi(t) \cdot (S_h \hat v_z)^2 (z,t) dzdt \qquad \mbox{ for all }h>0.\nn
\eea    
Here we can proceed for the second last as term as in Lemma \eqref{GesamtAbsch1}
to obtain by an application of Young's inequality and a substitution 
\begin{align*}
\frac1h \int_0^\infty &\io \xi(t)  p(z,t) \hat u_z(z,t-h) \cdot\left\{\hat u_z(z,t)-\hat u_z(z,t-h)\right\}dz dt \\
&= \frac1h \int_0^\infty \io \xi(t) p(z,t) \hat u_z(z,t-h)\hat u_z(z,t) dz dt \\
&\quad -\frac1h \int_0^\infty \io \xi(t) p(z,t) \hat u_z^2(z,t-h) \\
&\le \frac{1}{2h}\int_0^\infty \io \xi(t) p(z,t)\hat u_z^2(z,t) dz dt-\frac 1{2h} \int_0^\infty\io \xi(t) p(z,t)\hat u_z^2(z,t-h)dzdt\\
&\le \frac 1{2} \int_0^\infty \io \frac{\xi(t)-\xi(t+h)}h p(z,t) u_z^2(z,t) dz dt\\
& \quad +\frac 12 \int_0^\infty \io \xi(t+h)\frac{p(z,t)-p(z,t+h)}{h} u_z^2(z,t) dz dt\\
& \quad +\frac1{2h} \int_{-h}^0\io \xi(t+h)p(z,t+h)  u_{0z}^2(z) dz dt \\
&\to -\frac12 \int_0^\infty \io \xi'(t) p(z,t) u_z^2(z,t) dz dt -\frac12 \int_0^\infty \io \xi(t)p_t(z,t)  u_z^2(z,t) dz dt \\
&\quad -\frac12\io p(z,0) u_{0z}^2(z) dz dt \qquad \mbox{ as }h\searrow0.
\end{align*}
Since \eqref{Shvx} ensures
$$\limsup_{h\searrow0} \int_0^T \io (S_h \hat v_z)^2 (z,t)dz dt <\infty \qquad \mbox{ for all }T>0,$$
we may conclude
$$h\int_0^\infty \io \xi(t) \cdot (S_h \hat v_z)^2(z,t) dz dt\to 0 \qquad \mbox{ as }h\searrow0.$$
and thus from \eqref{LI1} the claim.
\end{proof}
\noindent Collecting our accomplishments, we note the following Lemma.
\begin{lemma}\label{EndAbsch}
Let $(u,v,\Theta)$, $\hat v$ and $S_h$ be as in Lemma \ref{KGZ1terTeil} and \ref{SolutionProp}, and let $\xi \in C^\infty_0([0,\infty))$  be such that $\xi'\le 0$ and $\xi(0)=1$, then

\begin{align}\label{GesamtAbsch}
&\qquad\frac 12\io \rho(z,0) u_{0t}^2(z) dz +\frac 12\io p(z,0) u_{0z}^2(z)dz  +\int_0^\infty \io \xi(t) \beta \Theta(z,t)  v_z(z,t)dz dt \nn\\
&\le\int_0^\infty \io \xi(t) \Gamma(\Theta(z,t)) v_z^2(z,t) dz dt -\frac12 \int_0^\infty \io \big(\xi'(t) p(z,t) + \xi(t)p_t(z,t)\big)u_z^2(z,t) dz dt \nn\\
&\qquad -\frac12 \int_0^\infty \io \big(\rho_t(z,t) \xi(t)+\rho(z,t)\xi'(t)\big) v^2 (z,t) dz dt
\end{align}
    
\end{lemma}
\begin{proof}
 By combining Lemma \ref{GesamtAbsch1} and Lemma \ref{GesamtAbsch2}, the claim follows.   
\end{proof}
\noindent We are now in a position to derive the final ingredient in verifying \eqref{weak2} from an lower semicontinuity argument for weak convergence.
\begin{lemma}\label{lem_sqrtL2}
    Let $(u,v,\Theta)$ and $(\eps_j)_{j\in\N}$ be as in Lemma \ref{KGZ1terTeil} and let $T>0$. Then 
\bea{sqrtL2}
\sqrt{\Gamma_\eps(\Theta_\eps)}v_{\eps z} \to \sqrt{\Gamma(\Theta)} v_z \qquad \mbox{ in }L^2(\Om\times(0,T)) \qquad \mbox{ as }\eps=\eps_j\searrow0.
\eea
\end{lemma}

\begin{proof}
Let $T>0$, due to $0<p,\rho\in C^1(\bom\times[0,T])$, as noted in Assumptions \ref{vorrfunc} and the Preliminaries, one can find $c_1>0$ such that
$$\frac{p_t(z,t)}{p(z,t)}\le c_1 \quad \mbox{ and }\quad \frac{\rho_t(z,t)}{\rho(z,t)}\le c_1$$
holds for all $x\in\Om$ and $t\in[0,T]$. Furthermore, we fix $\Psi(t)\in C^\infty_0([0,\infty))$ with $\Psi(0)=1$, $[0,T+1]\subset\mbox{supp} \Psi $ and $-1\le\Psi'< 0$ on $[0,T]$ as well as 
$$\kappa:=c_1+1,$$
then $\xi(t)=e^{-\kappa t}\cdot \Psi(t)$ fulfills $\xi\in C^\infty_0([0,\infty))$, $\xi(0)=1$, $\xi'(t)=\Psi'(t) e^{-\kappa t}-\kappa e^{-\kappa t} \Psi(t)\le 0$ and $ \xi \ge c_0$ on $t\in[0,T]$ for some $c_0\ge0$. 
In view of \eqref{KGZ_Gamma}, \eqref{ThetaKgz},\eqref{vzKgz}, the boundedness of $\xi$ and of $\mbox{supp }\xi$ we may infer
\bea{schwacheKGZGammaVz}
\sqrt{\Gamma_\eps(\Theta_\eps)}v_{\eps z} \rightharpoonup \sqrt{\Gamma(\Theta)}v_z
\quad\mbox{ in }L^2(\Om\times(0,T))\qquad\mbox{ as }\eps\searrow0.\eea
By defining
\bea{liminf1}
I_{1,\eps} := \int_0^\infty \io \xi(t) \Gamma_\eps(\Theta_\eps(z,t))v_{\eps z}^2(z,t) dz dt, \qquad \eps\in(0,1),
\eea
we infer due to the positivity of $\Gamma$ and $\xi$ from the lower semicontinuity of weak convergence in $L^2(\Om)$
\bea{liminf1b}
\liminf_{\eps=\eps_j \searrow0} I_{1,\eps}\ge I_1:=\int_0^\infty \io \xi(t) \Gamma(\Theta(z,t))v_{ z}^2(z,t) dz dt,
\eea
and for
\bea{liminf2}
I_{2,\eps} &:=& -\frac 12\int_0^\infty \io \big(\xi'(t) p_\eps(z,t)+\xi(t) p_{\eps t}(z,t)\big) u_{\eps z}^2(z,t) dz dt\\
& & -\frac12 \int_0^\infty \io (\xi'(t)\rho_\eps(z,t)+\xi(t)\rho_{\eps t}(z,t)) v_\eps^2(z,t) dz dt 
, \qquad \eps\in(0,1),\nn
\eea
our choice of $\kappa$ ensures 
$$ \xi'(t)p_\eps+\xi(t)(p_{\eps t}+1) \le 0\quad \mbox{ as well as }\quad \xi'(t)\rho_\eps(z,t)+\xi(t)\rho_{\eps t}(z,t)\le0$$
for all $(z,t)\in\Om\times[0,\infty)$ and thus moreover $I_{2,\eps}\ge 0$. Utilizing the lower semicontinuity of weak convergence in $L^2(\Om)$ once more, we gather
\bea{liminf2b}
\liminf_{\eps=\eps_j \searrow0} I_{2,\eps}\ge I_2&:=& -\frac 12 \int_0^\infty \io \big(\xi(t)p_t(z,t)+\xi'(t)p(z,t)\big) u_z^2(z,t)dzdt\nn\\
& & -\frac12 \int_0^\infty \io \big(\xi'(t)\rho(z,t)+\xi(t)\rho_t(z,t)\big) v^2 (z,t) dz dt
\eea
Recalling \eqref{vepsueps} from Lemma \ref{vepsuepslem},  we are able to conclude
\begin{align*}
&\quad\frac12\int_0^\infty \xi(t) \cdot \left\{ 
\dt \io \rho_\eps(z,t) v_\eps^2(z,t) dz
\right\} dt + \frac12\int_0^\infty \xi(t) \cdot \left\{ \dt \io p_{\eps}(z,t) u_{\eps z}^2(z,t) dz \right\} dt\nn\\ 
&\quad + \int_0^\infty\xi(t)\io \Geps v_{\eps z}^2 +\eps \int_0^\infty  \xi(t) \io v_{\eps zz}^2 +\eps \int_0^\infty \xi(t) \io p_\eps u_{\eps zz}^2 +\eps \int_0^\infty\xi(t)\io p_{\eps z} u_{\eps z} u_{\eps zz}\\ &=   \int_0^\infty\xi(t)\io \beta \Theta_\eps v_{\eps z} dz dt + \frac{1}{2}\int_0^\infty \xi(t)\io \rho_{\eps t}v_\eps^2+\frac 12\int_0^\infty  \xi(t) \io  p_{\eps t} u_{\eps z}^2, \nonumber
\end{align*}
wherefore an integration by parts in the first two integrals yields
\begin{align*}
    \frac12 \int_0^\infty \xi(t) &\left\{\dt \io \rho_\eps(z,t) v_\eps^2(z,t) dz\right\} dt \\
    &\quad = -\frac12\int_0^\infty \io \xi'(t) \rho_\eps(z,t) v_\eps^2 (z,t) dz dt - \frac12\io \rho_\eps(z,0) v_{0\eps}^2(z)  dz \qquad \mbox{ for all }\eps\in(0,1),
\end{align*}
and 
\begin{align*}
  \frac12 \int_0^\infty \xi(t) &\left\{\dt \io p_\eps u_{\eps z}^2(z,t) dz\right\} dt \\
    &\quad = -\frac12\int_0^\infty \io \xi'(t)p_\eps u_{\eps z}^2 (z,t) dz dt -\frac12 \io  p_\eps(z,0) u_{0\eps z}^2(z)  dz \qquad \mbox{ for all }\eps\in(0,1).
\end{align*}
Thus, we obtain for every $\eps\in(0,1)$ by reordering
\begin{align}\label{I1+I2}
    I_{1,\eps} +I_{2,\eps} &=  
\int_0^\infty \io \xi(t) \beta \Theta_\eps(z,t)dzdt + \io \rho_\eps(z,0) v_{0\eps}^2 (z) dz + \io p_\eps(u,0) u_{0 \eps z}^2 (z) dz\nn \\
&\quad -\eps \int_0^\infty \io \xi(t) v_{\eps zz}^2(z,t)dzdt -\eps \int_0^\infty \io \xi(t) p_\eps(z,t) u_{\eps zz}^2(z,t) dzdt\nn\\
& \quad -\eps \int_0^\infty\io \xi(t) p_{\eps z}(z,t) u_{\eps z}(z,t) u_{\eps zz}(z,t) dzdt.
\end{align}
Since $p_{\eps}$ satisfies \eqref{ab5.1} and \eqref{KGZ_Pz} there exists $c_2(T)>0$ such that
$\|p_{\eps z}\|_{L^\infty(\Om)}\le c_1 $
and by Young's inequality we conclude for every $\eps\in(0,1)$
\begin{align*}
-\eps \int_0^\infty \io \xi(t) p_{\eps z}(z,t) u_{\eps z}(z,t) &u_{\eps zz}(z,t) dzdt \\
&\le \eps c_1 \int_0^\infty \io u_{\eps zz}^2(z,t) dzdt+ \frac{\eps c_2}{c_1}\int_0^\infty\io \xi(t) u_{\eps z}^2 (z,t) dzdt
\end{align*}
and furthermore
\begin{align*}
    \io \rho_\eps(z,0) v_{0\eps}^2(z)dz \to \io \rho (z,0) u_{0t}^2(z) dz 
\quad \mbox{ and }\quad
\io p_{\eps}(z,t) u_{0 \eps z}^2 (z) dz \to \io p(z,t) u_{0 z}^2 (z) dz
\end{align*}
as $\eps=\eps_j\searrow0$ due to \eqref{initkonv}.
Since $\mbox{supp }  \xi$ is bounded we infer from \eqref{t6} in Lemma \ref{MassenAbsch}
$$\eps \frac{c_2}{c_1}\int_0^\infty \io \xi(t) u_{\eps z}^2(z,t)dz dt \to 0 $$
as $\eps=\eps_j\searrow0$ and since \eqref{ab5.1} ensures
$$-\eps \int_0^\infty \io \xi(t) (p_{\eps}(z,t)-c_1) u_{\eps zz}^2(z,t)dz dt-\eps \int_0^\infty \io \xi(t)  v_{\eps zz}^2(z,t) dz dt\le 0 $$
as for all $\eps\in(0,1)$. In view of \eqref{uzKgz} and \eqref{ThetaKgz} we also obtain
$$\int_0^\infty \io \xi(t)\beta\Theta_\eps(z,t)v_{\eps z}(z,t) dzdt \to \int_0^\infty \io  \xi(t)\beta \Theta(z,t) v(z,t) dzdt \quad \mbox{ as }\eps\searrow0.$$
From $\eqref{I1+I2}$ it follows that 

\begin{align*}
\limsup_{\eps=\eps_j\searrow0} (I_{1,\eps}+I_{2,\eps}) \le I_3&:= \int_0^\infty \beta \Theta(z,t) v_z(z,t) dzdt
 + \io \rho(z,0) u_{0t}^2(z)dz +\io p(z,0)u_{0z}^2(z) dz
\end{align*}
and by applying Lemma \ref{EndAbsch}
\begin{align}\label{limsupI_1+I_2}
    \limsup_{\eps=\eps_j\searrow0} (I_{1,\eps} +I_{2,\eps}) &\le \int_0^\infty \io \xi(t) \Gamma(\Theta(z,t))v_z^2 dz dt -\frac12 \int_0^\infty \io \big(\xi'(t) p(z,t)  + \xi(t) p_t(z,t)\big)u_z^2(z,t) dz dt \nn\\
& \quad -\frac12 \int_0^\infty \io \big(\rho_t(z,t) \xi(t)+\rho(z,t)\xi'(t)\big) v^2 (z,t) dz dt\nn\\
&= I_1+I_2
\end{align}
Recalling $\eqref{liminf1b}$ and $\eqref{liminf2b}$ this is only possible if 
\bea{I1Konvergenz}I_{1,\eps}\to I_1 \qquad \mbox{ as }\eps=\eps_j\searrow0,\eea
since otherwise \eqref{liminf1} and \eqref{liminf2} would imply that there would exist a constant $c_3>0$ and a subsequence $(\eps_{j_l})_{l\in\N}\subset(\eps_j)_{j\in\N}$,
such that 
$$I_{1,\eps}\ge I_1+c_3 \quad \mbox{ and }\quad I_{2,\eps} \ge I_2-\frac{c_3}{2} \qquad \mbox{ for all }\eps\in(\eps_{j_l})_{l\in\N},$$
which would result in 
$$I_{1,\eps}+I_{2,\eps}\le I_1+I_2+\frac{c_3}{2} \qquad\mbox{ for all } \eps \in(\eps_{j_l})_{l\in\N}$$
and thereby contradict $\eqref{limsupI_1+I_2}$.
In view of the definitions of $(I_{1,\eps})_{\eps\in(0,1)}$ and $I_1$ in \eqref{liminf1} and \eqref{liminf1b}, the convergence in \eqref{I1Konvergenz} along with \eqref{schwacheKGZGammaVz} it readily follows that
$$\sqrt{\xi(t)\Gamma_\eps(\Theta_\eps)}v_{\eps z} \to \sqrt{\xi(t)\Gamma(\Theta)}v_z
\qquad\mbox{ in }L^2(\Om\times(0,\infty)) \qquad\mbox{ as }\eps\searrow0.$$
and thus
$$ \int_0^T\io \left|\sqrt{\Gamma_\eps(\Theta_\eps)}v_{\eps z}-\sqrt{\Gamma(\Theta) }v_z\right|^2\le \int_0^\infty \io \frac{\sqrt{\xi(t)}}{c_0}\left|\sqrt{\Gamma_\eps(\Theta_\eps)}v_{\eps z}-\sqrt{\Gamma(\Theta) }v_z\right|^2\to 0 $$
as $\eps=\eps_j\searrow0$, which completes the proof.

\end{proof}
\noindent Now we are in a position to take the final step in proving Theorem \ref{EX}.\\
\textit{Proof of Theorem \ref{EX}.}\quad Since Lemma \ref{KGZ1terTeil} ensures the existence of functions $(u,v,\Theta)$ such that \ref{KGZEIG1} is satisfied, that $u_t=v$ a.e. on $\Om\times(0,\infty)$ and that \eqref{weak1} holds for all $\vp\in C^\infty_0(\bom\times[0,\infty))$, the verification of \eqref{weak2} for our found pair of functions remains. With regard to \eqref{SchwacheKGZ2} and \eqref{SchwacheKGZ3} from Lemma \ref{KGZ1terTeil}, only 
$$ \int_0^\infty \io \Gamma_{\eps}(\Theta_\eps) v_{\eps z}^2 \vp \to \int_0^\infty \io \Gamma(\Theta) u_{zt}^2\vp \qquad\mbox{ as }\eps=\eps_j\searrow0$$
for all $\vp\in C^\infty_0(\bom\times[0,\infty))$ remained to be shown, which we, however, accomplished in Lemma \ref{lem_sqrtL2}.  The properties $u(\cdot,0)=u_0$ and $v(\cdot,0)=u_1$ as well as the inequality $\Theta\ge 0$ have been asserted by Lemma \ref{KGZ1terTeil} already.
\qed

\section{Operators of the corresponding parameter identification problem}\label{sec:OpAna}
\noindent PDE-based ill-posed parameter identification problems frequently consist of a model and additional observations or measurements, where the model depends on the parameter and also the state. Hence, in our setting the inverse problem aims at identifying the infinite dimensional parameter \[f(z,t)=(p_1,p_2(z,t),p_3(z,t))^T\] 
from given observations, where $p_1,p_2(z,t),p_3(z,t)$ correspond to \eqref{param_p}. 
 Throughout this section, let Assumptions \ref{vorrfunc} and property \eqref{ab01} hold.  Since the model operator will be defined on the whole state space and $\tau$ is prescribed on $[0,\infty)$, its bounded continuous extension to $\mathbb R$ (again denoted by $\tau$) will be used and assumed to satisfy
\[
0<c_\tau\le \tau(\xi)\le C_\tau
\qquad\text{for all }\xi\in\mathbb R,
\]
which does not affect the system \eqref{GLIu1} on the physically relevant set of states and is equivalent to the assumption on $\Gamma$ in Theorem \ref{EX}.

\noindent The following embeddings are used, where for the sake of readability, the constants $c_i$, $i\in \mathbb N$, arising in this context are kept fixed throughout this section. 
First, since the time interval \((0,T)\) has finite measure, there exists a continuous embedding
\[
L^1\bigl((0,T);(W^{1,\infty}(\Omega))^*\bigr)
\hookrightarrow
\bigl(L^\infty((0,T);W^{1,\infty}(\Omega))\bigr)^*,
\]
and thus, there exists \(c_1>0\) such that
\begin{equation}\label{c1Embed}
\|\cdot\|_{(L^{\infty}((0,T); W^{1,\infty}(\Omega)))^*}
\leq c_1 \|\cdot\|_{L^1((0,T); (W^{1,\infty}(\Omega))^*)}.    
\end{equation}
As $\Omega$ is bounded, using the continuous embedding
$L^1(\Omega) \hookrightarrow (W^{1,\infty}(\Omega))^*$
yields that there exists $ c_2>0$ such that
\begin{equation}\label{c2Embed}
\norm{\cdot}_{(W^{1,\infty}(\Omega))^*} \le c_2 \norm{\cdot}_{L^1(\Omega)}.
\end{equation}
Furthermore, 
the continuous embedding
$L^2\bigl((0,T);(H^{1}(\Omega))^*\bigr)
\hookrightarrow
\bigl(H^1((0,T);H^{1}(\Omega))\bigr)^*$
implies that there exists \(c_3>0\) such that
\begin{equation}\label{c3Embed}
\|\cdot\|_{(H^1((0,T);H^{1}(\Omega)))^*}
\leq c_3 \|\cdot\|_{L^2((0,T);(H^{1}(\Omega))^*)}.    
\end{equation}
Motivated by Theorem~\ref{EX}, define
\[
\mathbb W:=
H^1\bigl((0,T);H^1(\Omega,\mathbb R)\bigr)
\times
H^1\bigl((0,T);H_0^1(\Omega,\mathbb R)\bigr)
\times
\Bigl(
L^1\bigl((0,T);W^{1,1}(\Omega,\mathbb R)\bigr)
\cap
L^2(\Omega\times (0,T))
\Bigr).
\]
Consequently, a solution $(u, \phi^0, \Theta)$ to the system \eqref{GLIu1} satisfies $(u, \phi^0, \Theta)\in \mathbb W$. 
\begin{cor}[Regularity]\label{Reg_CTV}
    Let the Assumptions of Theorem \ref{EX} hold, where $(u, \phi^0, \Theta)\in \mathbb W$ is a solution to the system \eqref{GLIu1}. Then,
    \[u_{tt} \in (L^2\bigl((0,T);H^1(\Omega,\mathbb R)\bigr))^*~\text{ and }~ \Theta_t\in L^1((0,T); (W^{1,\infty}(\Omega))^*).  \]
\end{cor}
\begin{proof}
First, it will be shown that the that the second term in the third equation of the system \eqref{GLIu1}, \(-\partial_z(k\Theta_z)\), is an element of \(L^{1}((0,T); (W^{1,\infty}(\Omega))^*)\). 
Assumption \ref{vorrfunc}, $k\in C^1([0,T];C^1(\overline{\Omega}, \R))$.
In particular $c_k:=\norm{k}_{L^\infty([0,T];L^\infty(\overline{\Omega}, \R))}<\infty$.
By denoting the unit ball in some Banach space by $ B^1(\cdot)$ one has
\begin{align}
 \norm{-\frac{\dd{}}{\dd{z}}(k\Theta_z)}_{L^1((0,T); (W^{1,\infty}(\Omega))^*)} 
&=  \int_0^T \norm{-\frac{\dd{}}{\dd{z}}(k\Theta_z)}_{(W^{1,\infty}(\Omega))^*} \dd{t} \notag\\
&\leq \int_0^T  \sup_{\nu \in \partial B^1( W^{1,\infty}(\Omega))} \int_\Omega |k \Theta_z \nu_z| \dd z \dd t \notag\\
&\leq c_k \int_0^T \norm{\Theta_z}_{L^{1}(\Omega)} \dd t \leq c_k \int_0^T \norm{\Theta}_{W^{1,1}(\Omega)} \dd t \notag\\
&=  c_k \norm{\Theta}_{L^1((0,T); W^{1,1}(\Omega))} < \infty,\label{lapthet}
\end{align}
where  $ \| \nu_z\|_{L^{\infty}(\Om)} \le \| \nu\|_{W^{1,\infty}(\Om)}=1$ for $\nu \in \partial B^1( W^{1,\infty}(\Omega, \R))$, Hölder's inequality in space and the Neumann boundary condition for \(\Theta\) were used.
Second, since \(u \in H^1((0,T); H_0^1(\Omega))\), one obtains \(u_{zt} \in L^2((0,T); L^2(\Omega))\).
Hence, for the third term in the third equation of the system \eqref{GLIu1}, $\tau(\Theta) p_1 u_{zt}^2$,  using Assumption \ref{vorrfunc}, Hölder's inequality and that $u_{zt}$ is real-valued results in
\begin{align}
\norm{\tau(\Theta) p_1 u_{zt}^2}_{L^{1}((0,T); (W^{1,\infty}(\Omega))^*)}
&\leq p_1 C_\tau c_2 \norm{u_{zt}^2}_{L^1((0,T); L^1(\Omega))}= p_1 C_\tau c_2 \norm{u_{zt}}_{L^2((0,T); L^2(\Omega))}^2\notag\\ 
&\leq p_1 C_\tau c_2   \norm{u_t}_{L^2((0,T); H^1(\Omega))}^2 < \infty,\label{uxtsquare}
\end{align}
and similarly,
\begin{align}
\norm{\beta \Theta u_{zt}}_{L^1((0,T); (W^{1,\infty}(\Omega))^*))}&\leq \beta  c_2 \norm{\Theta u_{zt}}_{L^1((0,T); L^1(\Omega))}
\leq  \beta  c_2 \norm{\Theta}_{L^2((0,T); L^2(\Omega))}\norm{u_{zt}}_{L^2((0,T); L^2(\Omega))}\notag\\ 
&\leq   \beta  c_2 \norm{\Theta}_{L^2((0,T); L^2(\Omega))} \norm{u_t}_{L^2((0,T); H^1(\Omega))} < \infty. \label{betathetauzt}
\end{align}
Third, using the solution $(u, \phi^0, \Theta)\in \mathbb W$ to the system \eqref{GLIu1} and testing the system \eqref{GLIu1} with the test function $(0,0, \nu)$, where $\nu \in L^1((0,T);(W^{1,\infty}(\Omega)))$ is arbitrary, yields
\begin{align*}
    \abs{\int_0^T\int_\Omega c_{\mathrm{th}}\rho \Theta_t \nu } ~\dd{z} \dd{t}\leq \int_0^T\int_\Omega \abs{k\Theta_z\nu_z - \tau(\Theta)p_1(u_{zt})^2\nu + \beta\Theta u_{zt}\nu} \dd{z} \dd{t}.
\end{align*}
By taking the supremum over $\partial B^1( W^{1,\infty}(\Omega))$, using the Assumption \ref{vorrfunc} and combining the inequalities \eqref{lapthet}, \eqref{uxtsquare},\eqref{betathetauzt} there exists $c_4>0$ such that 
\begin{align*}
    &\norm{ \Theta_t}_{L^1((0,T); (W^{1,\infty}(\Omega))^*))} \\&\qquad\leq c_4\left(\norm{\Theta}_{L^1((0,T); W^{1,1}(\Omega))}+\norm{u_t}_{L^2((0,T); H^1(\Omega))}^2+ \norm{\Theta}_{L^2((0,T); L^2(\Omega))} \norm{u_t}_{L^2((0,T); H^1(\Omega))}\right).
\end{align*}
Lastly, using the solution $(u, \phi^0, \Theta)\in \mathbb W$ to the system \eqref{GLIu1} and testing the system \eqref{GLIu1} with the test function $(\mu, 0, 0)$, where $\mu \in H^1\bigl((0,T);H^1(\Omega,\mathbb R)\bigr)$ is arbitrary, yields
\begin{align}
\int_0^T\int_{\Omega}&
\rho u_{tt}\mu +\left(p_1u_z +\tau(\Theta) p_1u_{zt}
+ p_2  \phi_z^0+p_2\chi_z
-\beta\Theta\right)\mu_z
\dd{\Omega}\dd{t}=0.
\end{align}
Consequently, 
\begin{align}
\abs{\int_0^T\int_{\Omega}
\rho u_{tt}\mu \dd{\Omega}\dd{t}}
&\leq \int_0^T\int_{\Omega}\abs{\left( p_1u_z +\tau(\Theta) p_1u_{zt}
+ p_2  \phi_z^0+p_2\chi_z
-\beta\Theta\right)\mu_z} \dd{\Omega}\dd{t}\\
&\leq p_1\left(\norm{u_z}_{L^2(\Omega\times (0,T))}+C_\tau\norm{u_{zt}}_{L^2(\Omega\times (0,T))}
\right)\norm{\mu_z}_{L^2(\Omega\times (0,T))}\\
&~~+\norm{p_2}_{L^\infty(\Omega)}\left(\norm{\phi_z^0}_{L^2(\Omega\times (0,T))}+C_\tau\norm{\chi_z}_{L^2(\Omega\times (0,T))}
\right)\norm{\mu_z}_{L^2(\Omega\times (0,T))}\\
&~~+\beta\norm{\Theta}_{L^2(\Omega\times (0,T))}\norm{\mu_z}_{L^2(\Omega\times (0,T))}<\infty.
\end{align}
Due to Assumption \ref{vorrfunc} there exists a constant $c_a>0$ such that 
\begin{align}
\norm{u_{tt}}_{\bigl(L^2\bigl((0,T);H^1(\Omega,\mathbb R)\bigr)\bigr)^*}
&\leq c_a \Big( \norm{u}_{ H^1\bigl((0,T);H^1(\Omega,\mathbb R)\bigr)} + \norm{\phi^0}_{H^1\bigl((0,T);H_0^1(\Omega,\mathbb R)\bigr)} 
\\&~~~~~+ \norm{\Theta}_{L^2(\Omega\times (0,T))} +\norm{\chi}_{L^2\bigl((0,T);H^1(\Omega,\mathbb R)\bigr)}\Big) <\infty.
\end{align}
\end{proof}
\noindent Consequently, the state space is defined as 
\begin{align*}
W := \Biggl\{ (u, \phi^0, \Theta) \in \mathbb{W} :~&  u_{tt} \in \bigl(L^2((0,T); H^1(\Omega, \R))\bigr)^*, ~\Theta_t \in  L^1((0,T); (W^{1,\infty}(\Omega, \R))^*) \Biggr\}.
\end{align*}
Furthermore, define
\[
\widehat W
:=
H^1((0,T);H^1(\Omega,\mathbb R))
\times
H^1((0,T);H_0^1(\Omega,\mathbb R))
\times
L^{\infty}((0,T); W^{1,\infty}(\Omega, \R)) \subset W
\]

\noindent and its dual space $W^*\subset \widetilde W:=\widehat W^*$. Note, that the all-at-once formulation below is based on the system \eqref{GLIu1}, where the initial data remain fixed. 
\begin{definition}[Model operator]\label{Aop}
Let $f=(p_1,p_2,p_3)\in X$ and $s=(u,\phi^0,\Theta)\in W$. The model operator $\mathcal A:X\times W\to \widetilde W$ is defined 
for every $(\mu,w,\nu)\in \widehat W$
by
\[
\langle \mathcal A(f,s), (\mu,w,\nu)\rangle:=
\langle \mathcal A_1(f,s),\mu\rangle+\langle \mathcal A_2(f,s),w\rangle+\langle \mathcal A_3(f,s),\nu\rangle,
\]
where
\begin{align}
\langle \mathcal A_1(f,s),\mu\rangle
&:=
\int_0^T\int_\Omega \rho\,u_{tt}\,\mu+
\Bigl(
p_1u_z+\tau(\Theta)p_1u_{zt}+p_2\phi_z^0+p_2\chi_z-\beta\Theta
\Bigr)\mu_z\,\dd{z}\dd{t},
\label{A1def_u_tt}
\\
\langle \mathcal A_2(f,s),w\rangle
&:=
\int_0^T\int_\Omega
\Bigl(
p_2u_z-p_3\phi_z^0-p_3\chi_z
\Bigr)w_z\,\dd{z}\dd{t},
\label{A2def_u_tt}
\\
\langle \mathcal A_3(f,s),\nu\rangle
&:=
\int_0^T\int_\Omega c_{\mathrm{th}}\rho\,\Theta_t+ k\Theta_z\nu_z
- \tau(\Theta)p_1(u_{zt})^2\nu+ \beta\Theta u_{zt}\nu\,\dd{z}\dd{t}.
\label{A3def_theta_t}
\end{align}
\end{definition}

\begin{theorem}\label{wd}
Let the assumptions of Theorem \ref{EX} hold.
Then, the model operator $\mathcal{A}$
of Definition~\ref{Aop} is well-defined and bounded on bounded subsets of $X\times W$. Moreover, for every $p\in X$ there exists $z\in W$ such that
\begin{equation}
    \mathcal A(f,s)=0
\qquad\text{in }\widetilde W.\label{modeeq}
\end{equation}
\end{theorem}

\begin{proof}
Let $p\in X$ be fixed. By Theorem~\ref{EX}, there exists a weak solution $z:=(u,\phi^0,\Theta)\in W$ of the lifted system \eqref{GLIu1}. Thus,  by Definition~\ref{Aop} it holds that equation \eqref{modeeq} is satisfied.
Let $(\mu,w,\nu)\in \widehat W$. The regularities
\[
u\in H^1((0,T);H^1(\Omega)),\qquad
\phi^0\in H^1((0,T);H_0^1(\Omega)),\qquad
\Theta\in L^2(\Omega\times (0,T)),
\]
together with Assumption~\ref{vorrfunc}, Hölder's inequality,
Corollary~\ref{Reg_CTV} and the definition of \(W\), yield
\begin{align*}
|\langle \mathcal A_1(f,s),\mu\rangle|
&\leq \int_0^T\int_{\Omega}
\abs{\rho u_{tt}\mu} \dd{\Omega}\dd{t}+
p_1\left(\norm{u_z}_{L^2(\Omega\times (0,T))}+C_\tau\norm{u_{zt}}_{L^2(\Omega\times (0,T))}
\right)\norm{\mu_z}_{L^2(\Omega\times (0,T))}\\
&~~+\norm{p_2}_{L^\infty(\Omega)}\left(\norm{\phi_z^0}_{L^2(\Omega\times (0,T))}+C_\tau\norm{\chi_z}_{L^2(\Omega\times (0,T))}
\right)\norm{\mu_z}_{L^2(\Omega\times (0,T))}\\
&~~+\beta\norm{\Theta}_{L^2(\Omega\times (0,T))}\norm{\mu_z}_{L^2(\Omega\times (0,T))}.
\end{align*}
Using the definition of $W$, $c_3$ as in inequality \eqref{c3Embed} and $c_a$ from Corollary~\ref{Reg_CTV} one obtains
\begin{align}
\norm{\mathcal A_1(f,s)}_{(H^1((0,T);H^{1}(\Omega)))^*}
&\leq c_3\norm{u_{tt}}_{\bigl(L^2\bigl((0,T);H^1(\Omega,\mathbb R)\bigr)\bigr)^*}+c_3 c_a \Big( \norm{u}_{ H^1\bigl((0,T);H^1(\Omega,\mathbb R)\bigr)} 
\\&~~~+ \norm{\phi^0}_{H^1\bigl((0,T);H_0^1(\Omega,\mathbb R)\bigr)} + \norm{\Theta}_{L^2(\Omega\times (0,T))} +\norm{\chi}_{L^2\bigl((0,T);H^1(\Omega,\mathbb R)\bigr)}\Big) <\infty.
\end{align}
For the second component, again by Hölder's inequality, it holds that
\begin{align*}
|\langle \mathcal A_2(f,s),w\rangle|
&\le
\|p_2\|_{L^\infty(\Omega\times (0,T))}
\|u\|_{L^2((0,T);H^1(\Omega))}
\|w\|_{L^2((0,T);H_0^1(\Omega))}
\\
&\quad
+\|p_3\|_{L^\infty(\Omega\times (0,T))}
\Bigl(
\|\phi^0\|_{L^2((0,T);H_0^1(\Omega))}
+\|\chi\|_{L^2((0,T);H^1(\Omega))}
\Bigr)
\|w\|_{L^2((0,T);H_0^1(\Omega))}.
\end{align*}
Thus
\[
\mathcal A_2(f,s)\in (L^2((0,T);H_0^1(\Omega)))^*.
\]
Lastly, by using $c_1$ as in inequality \eqref{c1Embed} and Corollary~\ref{Reg_CTV} as well as the definition of $W$ it holds that
\begin{align*}
\norm{\mathcal A_3(f,s)}_{(L^\infty((0,T);W^{1,\infty}(\Omega)))^*}     \leq& c_1\norm{ \Theta_t}_{L^1((0,T); (W^{1,\infty}(\Omega))^*))}+c_1c_4\left(\norm{\Theta}_{L^1((0,T); W^{1,1}(\Omega))}
\right.\\&\left.~+\norm{u_t}_{L^2((0,T); H^1(\Omega))}^2+ \norm{\Theta}_{L^2((0,T); L^2(\Omega))} \norm{u_t}_{L^2((0,T); H^1(\Omega))}\right).
\end{align*}
Collecting the preceding bounds yields that
$\mathcal A(f,s)\in \widetilde W$
for every $(f,s)\in X\times W$. Moreover, for every bounded subset $B\subset X\times W$ there exists $K_B>0$ such that
\begin{align}
\|\mathcal A(f,s)\|_{\widetilde W}
\le\;& K_B\Bigl(
\|u\|_{H^1((0,T);H^1(\Omega))}
+\|\phi^0\|_{H^1((0,T);H_0^1(\Omega))}
+\|\Theta\|_{L^1((0,T);W^{1,1}(\Omega))}
\notag\\
&\qquad
+\|\Theta\|_{L^2(\Omega\times (0,T))}
+\|u\|_{H^1((0,T);H^1(\Omega))}^2
+ \norm{\Theta}_{L^2((0,T); L^2(\Omega))} \norm{u}_{H^1((0,T); H^1(\Omega))}+1
\Bigr)
\label{A_bound_u_tt}
\end{align}
for all $(f,s)\in B$, where the constant $K_B$ also absorbs the fixed norm of $\chi$. 
\end{proof}

\noindent  \noindent  In order to recover information on the parameter $f$ and state additional observations are needed. 
In piezoelectric components, electrical surface charges, occurring at the electrodes, are experimentally accessible and frequently used to determine material parameters.
Hence, we assume an excitation signal $\phi^e\in H^1(0,T)$ is applied at the boundary $z=h$,
while the specimen is grounded at the boundary $z=0$. 
Reasonably, we assume that $\norm{\phi^e}_{L^2\left( 0, T\right)}>0$, since otherwise thy system would not respond.
Hence, we observe the surface charge at the upper electrode.
Since the surface charge is the boundary evaluation of the electric displacement field, which is of $L^2$-regularity, we follow a similar approach as in \cite{Veronika}. 
Consequently, we define the observation operator 
via the following approximation approach, similar to \cite{Veronika}.
\begin{definition}[Observation operator]\label{obsOp}
Let $Y:=L^1(0,T)$. Then, with $\phi=\phi^0+\chi$, the observation operator $\mathcal C:X\times W\to Y$ is defined by
\begin{equation}\label{obsdef}
\mathcal C(f,s)
:=
\frac{1}{\|\phi^e\|_{L^2(0,T)}}
\int_\Omega (p_2u_z-p_3\phi_z)\phi_z\,\dd z.
\end{equation}
\end{definition}

\begin{lemma}\label{Cwd}
Let the assumptions of Theorem \ref{EX} hold.
Then, the observation operator $\mathcal C$ of Definition~\ref{obsOp} is well-defined and bounded on bounded subsets of $X\times W$.
\end{lemma}

\begin{proof}
Let $B\subset X\times W$ be bounded. Then there exists $c_{B}>0$ such that
\[
\frac{1}{\|\phi^e\|_{L^2(0,T)}}
\max\{\|p_2\|_{L^\infty(\Omega\times (0,T))},\|p_3\|_{L^\infty(\Omega\times (0,T))}\}
\le c_{B}
\]
for all $(f,s)\in B$. Hence, Hölder's inequality and $c_2$ as in inequality \eqref{c2Embed} yields
\begin{align*}
\|\mathcal C(f,s)\|_Y
&=
\int_0^T |[\mathcal C(f,s)](t)|\,dt
\\
&\le
\frac{1}{\|\phi^e\|_{L^2(0,T)}}
\left(
\|p_2u_z\phi_z\|_{L^1(\Omega\times (0,T))}
+
\|p_3\phi_z^2\|_{L^1(\Omega\times (0,T))}
\right)
\\
&\le
c_{B}
\|u_z\|_{L^2(\Omega\times (0,T))}
\|\phi_z\|_{L^2(\Omega\times (0,T))}
+
c_{B}c_2\|\phi_z\|_{L^2(\Omega\times (0,T))}^2
\\
&\le
c_{B}
\|u\|_{L^2((0,T);H^1(\Omega))}
\|\phi\|_{L^2((0,T);H^1(\Omega))}
+c_{B}c_2
\|\phi\|_{L^2((0,T);H^1(\Omega))}^2.
\end{align*}
Consequently, $\mathcal C$ is well-defined and bounded on bounded subsets of $X\times W$.
\end{proof}

\begin{remark}
If, in addition,
$u,\phi\in L^2((0,T);H^2(\Omega,\mathbb R)),$
then $u_z(h,\cdot)$ and $\phi_z(h,\cdot)$ are well-defined. Thus, the observation operator can be defined as the charge, i.e.,
\begin{equation}\label{PraxisObsOp}
\mathcal C(f,s)=C(f,s)=p_2(h,t)u_z(h,t)-p_3(h,t)\phi_z(h,t).
\end{equation}
\end{remark}

\noindent Since the reduced approach, i.e., the classical setting, where the model is  eliminated, by introducing the parameter-to-state map, 
needs the uniqueness of the weak solutions to the system \eqref{GLIu1}, and Theorem~\ref{EX} establishes existence of solutions to the system \eqref{GLIu1}, but uniqueness of solutions remains delicate, we will focus on the all-at-once approach. 
 Consequently, we are considering the model together with the observations as one system for $ (f, l) $, i.e.,	
	\begin{align*}
	\mathcal A(f, l)&=0 ~\text{ in } \widetilde{W},\\
	\mathcal  C(f,l)&=y ~ \text{ in } Y,
	\end{align*}
    simultaneously, where two infinite dimensional variables $ f $ and $ l $ have to be determined. 
 This motivates the following definition.

\begin{definition}[Forward operator]\label{FOp}
With the model operator $\mathcal{A}$
of Definition~\ref{Aop} and the observation operator $\mathcal C$ of Definition~\ref{obsOp}, the forward operator $F:X\times W\to \widetilde W\times Y$ is defined by
\[
F(f,s):=
\begin{pmatrix}
\mathcal A(f,s)\\[1mm]
\mathcal C(f,s)
\end{pmatrix}.
\]
\end{definition}
\noindent For given noisy measurements $y^\delta$, the inverse problem is modeled by the operator equation
\begin{equation}\label{noisopeq}
F(f,s)=
\begin{pmatrix}
0\\[1mm]
y^\delta
\end{pmatrix}.
\end{equation}
Let the solution space of the temperature be defined as
\begin{align*}
\mathbb{T}:= \Biggl\{ \Theta \in \left( L^1((0,T); W^{1,1}(\Omega, \R)) \cap L^2(\Omega \times (0,T)) \right): ~\Theta_t \in  L^1((0,T); (W^{1,\infty}(\Omega, \R))^*) \Biggr\}.
\end{align*}
For the differentiability analysis, we employ the assumption
$\tau\in C^1\bigl(\overline{\mathbb T};
L^\infty(\Omega\times (0,T))\bigr),$
together with the boundedness of \(\tau\) and \(\tau'\) on bounded subsets of
\(\overline{\mathbb T}\). In particular, 
$\tau(\Theta)\in L^\infty(\Omega\times(0,T))$
and
$\tau'(\Theta)\kappa\in L^\infty(\Omega\times(0,T))$
for every \(\Theta\in\overline{\mathbb T}\) and every  direction
\(\kappa\in\mathbb T\).
For an arbitrary direction \(\xi=(\eta,\psi,\kappa)\in W\), the G\^ateaux
derivative of \(\mathcal A\) with respect to the state is given by
\begin{align}
\langle \mathcal{A}_s'(f,s)\xi,(\mu, w, \nu)\rangle
&=
\int_0^T\int_\Omega
\bigl(
p_1\eta_z+\tau'(\Theta)\kappa\, p_1u_{zt}
+\tau(\Theta) p_1\eta_{zt}+p_2\psi_z-\beta\kappa
\bigr)\mu_z\,\dd{z}\dd{t}
\nonumber\\
&\hspace*{-10mm}+
\int_0^T\int_\Omega
\rho\,\eta_{tt}\,\mu+\bigl(
p_2\eta_z-p_3\psi_z
\bigr)w_z\,\dd{z}\dd{t}
-2\int_0^T\int_\Omega
\tau(\Theta) p_1u_{zt}\eta_{zt}\nu\,\dd{z}\dd{t}
\nonumber\\
&\hspace*{-10mm}+
\int_0^T\int_\Omega
\bigl(
c_{\mathrm{th}}\rho\,\kappa_t+k\kappa_z\nu_z
+\beta(\kappa u_{zt}+\Theta\eta_{zt})\nu
-\tau'(\Theta)\kappa\, p_1(u_{zt})^2\nu
\bigr)\,\dd{z}\dd{t}.
\label{As3_theta_t}
\end{align}
For an arbitrary direction \(q=(q_1,q_2,q_3)\in X\), the G\^ateaux derivative
of \(\mathcal A\) with respect to \(f\) is
\begin{align}
\langle \mathcal{A}_f'(f,s)q,(\mu, w, \nu)\rangle
&=
\int_0^T\int_\Omega
\bigl(
q_1u_z+\tau(\Theta) q_1u_{zt}+q_2\phi_z^0+q_2\chi_z
\bigr)\mu_z\,\dd{z}\dd{t}
\nonumber\\
&\hspace*{-5mm}+
\int_0^T\int_\Omega
\bigl(
q_2u_z-q_3\phi_z^0-q_3\chi_z
\bigr)w_z\,\dd{z}\dd{t}
-\int_0^T\int_\Omega
\tau(\Theta) q_1(u_{zt})^2\nu\,\dd{z}\dd{t}.
\label{Ap3_u_tt}
\end{align}
Similarly, one obtains the G\^ateaux derivatives for the observation operator
\(\mathcal C\),
\begin{align}
\mathcal{C}_s'(f,s)\xi
&:=
\frac{1}{\|\phi^e\|_{L^2(0,T)}}
\int_\Omega
\Bigl(
(p_2\eta_z-p_3\psi_z)\phi_z
+
(p_2u_z-p_3\phi_z)\psi_z
\Bigr)\,\dd z,
\label{Csdef_corr_u_tt}
\\
\mathcal{C}_f'(f,s)q
&:=
\frac{1}{\|\phi^e\|_{L^2(0,T)}}
\int_\Omega
(q_2u_z-q_3\phi_z)\phi_z\,\dd z.
\label{Cpdef_corr_u_tt}
\end{align}

\begin{theorem}[\frechet differentiability of \(F\)]
\label{FrechetThm}
Let the assumptions of Theorem \ref{EX} hold and assume that 
$\tau\in C^1\bigl(\overline{\mathbb T};
L^\infty(\Omega\times(0,T))\bigr).$
Suppose that both \(\tau\) and \(\tau'\) are bounded on bounded subsets of
\(\overline{\mathbb T}\), i.e., for every bounded subset
\(M\subset\overline{\mathbb T}\), there exist constants
\(C_{\tau,M}>0\) and \(C_{\tau',M}>0\) such that
\[
\|\tau(\Theta)\|_{L^\infty(\Omega\times(0,T))}
\leq C_{\tau,M}
\quad\text{and}\quad
\|\tau'(\Theta)\|_{\mathcal L(\mathbb T,
L^\infty(\Omega\times(0,T)))}
\leq C_{\tau',M}
\qquad\text{for all }\Theta\in M,
\]
where \(\mathcal L(\mathbb T,L^\infty(\Omega\times(0,T)))\) denotes the space
of bounded linear operators from \(\mathbb T\) into
\(L^\infty(\Omega\times(0,T))\). Then, the forward operator \(F\) of
Definition~\ref{FOp} is well-defined and continuously \frechet differentiable, and \(F\) as well as \(F'\) are bounded on bounded subsets of \(X\times W\).
\end{theorem}

\begin{proof}
By Theorem~\ref{wd} and Lemma~\ref{Cwd}, the forward operator \(F\) of
Definition~\ref{FOp} is well-defined and bounded on bounded subsets of \(X\times W\).
Let
$\xi=(\eta,\psi,\kappa)\in W$
and let \((\mu,w,\nu)\in\widehat W\) be arbitrary with
$\|(\mu,w,\nu)\|_{\widehat W}\leq 1.$
Then,
\begin{align*}
&\langle
\mathcal A(f,s+\xi)-\mathcal A(f,s)-\mathcal{A}_s'(f,s)\xi,
(\mu,w,\nu)
\rangle
\\
&=
\int_0^T\int_\Omega
p_1\Bigl(
\tau(\Theta+\kappa)(u_{zt}+\eta_{zt})-\tau(\Theta)u_{zt}
-\tau'(\Theta)\kappa\,u_{zt}
-\tau(\Theta)\eta_{zt}
\Bigr)\mu_z\,\dd{z}\dd{t}
\\
&\quad+
\int_0^T\int_\Omega
\beta\Bigl(
(\Theta+\kappa)(u_{zt}+\eta_{zt})
-\Theta u_{zt}
-\kappa u_{zt}
-\Theta\eta_{zt}
\Bigr)\nu\,\dd{z}\dd{t}
\\
&\quad-
\int_0^T\int_\Omega
p_1\Bigl(
\tau(\Theta+\kappa)(u_{zt}+\eta_{zt})^2
-\tau(\Theta)(u_{zt})^2
-\tau'(\Theta)\kappa\,(u_{zt})^2
-2\tau(\Theta)u_{zt}\eta_{zt}
\Bigr)\nu\,\dd{z}\dd{t}.
\end{align*}
Define
\[
r_\tau(\Theta,\kappa):=
\tau(\Theta+\kappa)-\tau(\Theta)-\tau'(\Theta)\kappa.
\]
\newpage
\noindent For the first component of the remainder, one can compute
\begin{align*}
&\tau(\Theta+\kappa)(u_{zt}+\eta_{zt})
-\tau(\Theta)u_{zt}
-\tau'(\Theta)\kappa\,u_{zt}
-\tau(\Theta)\eta_{zt}
\\
&=
\bigl(\tau(\Theta)+\tau'(\Theta)\kappa+r_\tau(\Theta,\kappa)\bigr)
(u_{zt}+\eta_{zt})
-\tau(\Theta)u_{zt}
-\tau'(\Theta)\kappa\,u_{zt}
-\tau(\Theta)\eta_{zt}
\\
&=
\tau'(\Theta)\kappa\,\eta_{zt}
+r_\tau(\Theta,\kappa)(u_{zt}+\eta_{zt}).
\end{align*}
For the second term, we have
\[
(\Theta+\kappa)(u_{zt}+\eta_{zt})
-\Theta u_{zt}
-\kappa u_{zt}
-\Theta\eta_{zt}
=
\kappa\eta_{zt}.
\]
Finally, for the third term, we obtain
\begin{align*}
&-\tau(\Theta+\kappa)p_1(u_{zt}+\eta_{zt})^2
+\tau(\Theta)p_1(u_{zt})^2
+\tau'(\Theta)\kappa\,p_1(u_{zt})^2
+2\tau(\Theta)p_1u_{zt}\eta_{zt}
\\
&=
-\bigl(\tau(\Theta)+\tau'(\Theta)\kappa+r_\tau(\Theta,\kappa)\bigr)
p_1\bigl((u_{zt})^2+2u_{zt}\eta_{zt}+(\eta_{zt})^2\bigr)
+\tau(\Theta)p_1(u_{zt})^2
\\
&\quad
+\tau'(\Theta)\kappa\,p_1(u_{zt})^2
+2\tau(\Theta)p_1u_{zt}\eta_{zt}
\\
&=
-\tau(\Theta)p_1(\eta_{zt})^2
-2\tau'(\Theta)\kappa\,p_1u_{zt}\eta_{zt}
-\tau'(\Theta)\kappa\,p_1(\eta_{zt})^2
-r_\tau(\Theta,\kappa)p_1(u_{zt}+\eta_{zt})^2.
\end{align*}
Hence, the remainder
\[
R(\xi):=\mathcal A(f,s+\xi)-\mathcal A(f,s)-\mathcal{A}_s'(f,s)\xi,
\]
is given by
\begin{align*}
\langle R(\xi),(\mu,w,\nu)\rangle
&=
\int_0^T\int_\Omega
p_1\Bigl(
\tau'(\Theta)\kappa\,\eta_{zt}
+r_\tau(\Theta,\kappa)(u_{zt}+\eta_{zt})
\Bigr)\mu_z\,\dd{z}\dd{t}
\\
&\quad+
\int_0^T\int_\Omega
\beta\,\kappa\eta_{zt}\,\nu\,\dd{z}\dd{t}
-\int_0^T\int_\Omega
\tau(\Theta)p_1(\eta_{zt})^2\nu\,\dd{z}\dd{t}
\\
&\quad-
2\int_0^T\int_\Omega
\tau'(\Theta)\kappa\,p_1u_{zt}\eta_{zt}\,\nu\,\dd{z}\dd{t}
-\int_0^T\int_\Omega
\tau'(\Theta)\kappa\,p_1(\eta_{zt})^2\,\nu\,\dd{z}\dd{t}
\\
&\quad-
\int_0^T\int_\Omega
r_\tau(\Theta,\kappa)p_1(u_{zt}+\eta_{zt})^2\nu\,\dd{z}\dd{t}.
\end{align*}
Since \((\mu,w,\nu)\in \widehat W\) with
\(\|(\mu,w,\nu)\|_{\widehat W}\leq1\) was arbitrary, H\"older's inequality
and analogous arguments as in the estimates \eqref{uxtsquare} and
\eqref{betathetauzt} imply, for \(\|\xi\|_W\neq0\), that there exists a
constant \(C_R>0\) such that
\begin{align*}
\frac{\|R(\xi)\|_{\widetilde W}}{\|\xi\|_W}
&\leq
C_R\Biggl(
\frac{\|\tau'(\Theta)\kappa\|_{L^\infty(\Omega\times(0,T))}}
{\|\xi\|_W}
\|\eta\|_{H^1(\Omega\times(0,T))}
+
\frac{\|\kappa\|_{L^2(\Omega\times(0,T))}}{\|\xi\|_W}
\|\eta\|_{H^1(\Omega\times(0,T))}
\\
&\quad+
\|\tau(\Theta)\|_{L^\infty(\Omega\times(0,T))}
\frac{\|\eta\|_{H^1(\Omega\times(0,T))}^2}{\|\xi\|_W}
\\
&\quad+
\frac{\|\tau'(\Theta)\kappa\|_{L^\infty(\Omega\times(0,T))}}
{\|\xi\|_W}
\Bigl(
\|u\|_{H^1(\Omega\times(0,T))}
\|\eta\|_{H^1(\Omega\times(0,T))}
+
\|\eta\|_{H^1(\Omega\times(0,T))}^2
\Bigr)
\Biggr)
\\
&\quad+
C_R
\frac{\|r_\tau(\Theta,\kappa)\|_{L^\infty(\Omega\times(0,T))}}
{\|\xi\|_W}
\Bigl(
\|u\|_{H^1(\Omega\times(0,T))}
+
\|\eta\|_{H^1(\Omega\times(0,T))}
\\
&\qquad\qquad\qquad\qquad\qquad\qquad\qquad\qquad
+
\|u\|_{H^1(\Omega\times(0,T))}^2
+
\|\eta\|_{H^1(\Omega\times(0,T))}^2
\Bigr).
\end{align*}
Since \(\Theta\in\overline{\mathbb T}\) is fixed and \(\tau'\) is bounded on
bounded subsets of \(\overline{\mathbb T}\), the operator \(\tau'(\Theta)\)
is bounded. Thus, there exists a constant \(C_\Theta>0\) such that
\[
\|\tau'(\Theta)\kappa\|_{L^\infty(\Omega\times(0,T))}
\leq
C_\Theta\|\kappa\|_{\mathbb T}
\leq
C_\Theta\|\xi\|_W.
\]
Moreover, since \(\tau\) is \frechet differentiable, one has
\[
\frac{
\|r_\tau(\Theta,\kappa)\|_{L^\infty(\Omega\times(0,T))}
}{
\|\kappa\|_{\mathbb T}
}
\longrightarrow 0
\qquad\text{as }\|\kappa\|_{\mathbb T}\to0.
\]
If \(\kappa=0\), then \(r_\tau(\Theta,\kappa)=0\). Otherwise, since
\(\|\kappa\|_{\mathbb T}\leq\|\xi\|_W\), we obtain
\[
\frac{
\|r_\tau(\Theta,\kappa)\|_{L^\infty(\Omega\times(0,T))}
}{
\|\xi\|_W
}
\leq
\frac{
\|r_\tau(\Theta,\kappa)\|_{L^\infty(\Omega\times(0,T))}
}{
\|\kappa\|_{\mathbb T}
}
\longrightarrow 0
\qquad\text{as }\|\xi\|_W\to0.
\]
Since $\|\eta\|_{H^1(\Omega\times(0,T))}\leq \|\xi\|_W$ and 
$\|\kappa\|_{\mathbb T}\leq \|\xi\|_W$,
every term on the right-hand side tends to zero as \(\|\xi\|_W\to0\).
Therefore,
\[
\lim_{\|\xi\|_W\to 0}
\frac{
\|\mathcal A(f,s+\xi)-\mathcal A(f,s)-\mathcal{A}_s'(f,s)\xi\|_{\widetilde W}
}{
\|\xi\|_W
}
=0.
\]
Now, fix \(s\in W\). Since \(\mathcal A(\cdot,s)\) is linear in the
parameter variable, one has
\[
\mathcal A(f+q,s)-\mathcal A(f,s)-\mathcal{A}_f'(f,s)q=0
\]
for all admissible \(q\in X\) with \(f+q\in X\). Consequently, \(\mathcal A\)
is \frechet differentiable with respect to the parameter variable.

\noindent Since
$\tau\in C^1\bigl(\overline{\mathbb T};
L^\infty(\Omega\times(0,T))\bigr),$
and since the remaining terms in \(\mathcal A'\) are at most quadratic in the
state, it follows that for
$(f_n,s_n)\underset{n \to \infty}{\longrightarrow}(f,s)$ in 
$X\times W$,
where \(s_n=(u_n,\phi_n,\Theta_n)\) and \(s=(u,\phi,\Theta)\), one obtains
\[
\|\mathcal A'(f_n,s_n)-\mathcal A'(f,s)\|_
{\mathcal L(X\times W,\widetilde W)}
\underset{n \to \infty}{\longrightarrow} 0.
\]
Thus, \(\mathcal A\) is continuously \frechet differentiable.

\noindent For the \frechet differentiability of the observation operator \(\mathcal C\),
one obtains with an arbitrary \(\xi=(\eta,\psi,\kappa)\in W\) and
\(\phi+\psi=(\phi^0+\chi)+\psi\) that
\[
\mathcal C(f,s+\xi)-\mathcal C(f,s)-\mathcal{C}_s'(f,s)\xi
=
\frac{1}{\|\phi^e\|_{L^2(0,T)}}
\int_\Omega
\bigl(
p_2\eta_z-p_3\psi_z
\bigr)\psi_z\,\dd z.
\]
Hence, using H\"older's inequality, one obtains, analogously to above, that
\(\mathcal C\) is \frechet differentiable with respect to the state variable.
Similarly, since \(\mathcal C\) depends linearly on the parameters,
\(\mathcal C\) is \frechet differentiable with respect to the parameters as
well. The continuity of \(\mathcal C'\) follows from the same arguments as for
\(\mathcal A'\). Consequently,
\[
F\in C^1(X\times W,\widetilde W\times Y).
\]
Finally, by Theorem~\ref{wd} and Lemma~\ref{Cwd}, together with the previous
estimates, it follows that for every bounded subset \(B\subset X\times W\),
both \(F\) and \(F'\) are bounded on \(B\).
\end{proof}

\section{Conclusion}

\noindent Firstly, we proved existence of global weak solutions to a thermo-piezoelectric system governed by a Kelvin-Voigt damped/coupled hyperbolic-parabolic-elliptic PDE with Bochner functions as density, thermal conductivity, heat capacity, electrical permittivity and piezoelectric coupling parameters. 

\noindent Secondly, we modeled and analyzed the corresponding operators of the inverse parameter identification problem. 
Therein, we discussed well-definedness, boundedness and  continuous \frechet differentiability of the forward operator.

\section*{Acknowledgements}
\noindent  The authors would like to thank the German Research Foundation (DFG) for financial support within the research group 5208 NEPTUN (444955436).

\noindent The work on this paper was partly funded by the Deutsche Forschungsgemeinschaft (DFG, German Research Foundation) under Germany´s Excellence Strategy – The Berlin Mathematics Research Center MATH+ (EXC-2046/1, EXC-2046/2, project ID: 390685689).

\section*{Data Availability}
\noindent Not applicable. This paper is purely theoretical and contains no empirical data.

\section*{Competing Interests}
\noindent The authors declare no competing interests.

\printbibliography

@InProceedings{978-3-030-56356-1,
author="Avalishvili, G.
and Avalishvili, M.",
title="On variational methods of investigation of mathematical problems for thermoelastic piezoelectric solids",
booktitle="Applications of Mathematics and Informatics in Natural Sciences and Engineering",
year="2020",
publisher="Springer International Publishing",
address="Cham",
pages="1--18"
}

@phdthesis{Veronika,
  author      = {V. Schulze},
  title       = {Modeling and optimization of electrode configurations for piezoelectric material},
  type        = {doctoral thesis},
  school      = {Humboldt-Universität zu Berlin},
  year        = {2023},
}

@article{PDE_constrained,
author = {Leugering, G. and Stingl, M.},
title = {PDE-constrained optimization for advanced materials},
journal = {GAMM-Mitteilungen},
volume = {33},
number = {2},
pages = {209-229},
year = {2010}
}

@article{14,
author = {Nicaise, S.},
year = {2003},
pages = {},
title = {Stability and controllability of the elecromagneto-elastic system},
volume = {60},
journal = {Portugaliae Mathematica. Nova Série}
}

@book{temam2024,
  title={Navier-Stokes equations: Theory and numerical analysis},
  author={Temam, R.},
  volume={343},
  year={2024},
  publisher={American Mathematical Society}
}

@unpublished{KFH2W5a,
    author = {Kuess, R. and Friesen, O. and Claes, L. and Walther, A.},
    title = {Sensitivity-informed identification of temperature-dependent piezoelectric material parameters},
    note = {Optimization Online},
    url={https://optimization-online.org/?p=33478},
    year ={2026}
}

@unpublished{RKDWAW,
author = {Kuess, R. and Walter, D. and  Walther, A.},
title = {Modelling and analysis of an inverse parameter identification problem in piezoelectricity},
note = {Technical report}
}

@article{18,
title = {Contrôlabilité d'un corps piézoélectrique},
journal = {Comptes Rendus de l'Académie des Sciences - Series I - Mathematics},
volume = {333},
number = {3},
pages = {267-270},
year = {2001},
author = {Bernadette, M.}
}

@article{19,
author = {Leugering, G. and Novotny, A. A. and Menzala, G. Perla and Sokołowski, J.},
title = {Shape sensitivity analysis of a quasi-electrostatic piezoelectric system in multilayered media},
journal = {Mathematical Methods in the Applied Sciences},
volume = {33},
number = {17},
pages = {2118-2131},
year = {2010}
}

@article{drz047,
    author = {Antil, H. and Brown, T. S. and Sayas, F.-J.},
    title = "{A problem in control of elastodynamics with piezoelectric effects}",
    journal = {IMA Journal of Numerical Analysis},
    volume = {40},
    number = {4},
    pages = {2839-2870},
    year = {2019},
    month = {12}
}

@incollection{Amann1993,
  author    = {Amann, H.},
  title     = {Nonhomogeneous linear and quasilinear elliptic and parabolic boundary value problems},
  booktitle = {Function Spaces, Differential Operators and Nonlinear Analysis},
  series    = {Teubner-Texte zur Mathematik},
  volume    = {133},
  pages     = {9--126},
  publisher = {Teubner},
  address   = {Stuttgart},
  year      = {1993},
}

@ARTICLE{ClaesLankeitWinkler,
  author  = {Claes, L. and Lankeit, J. and Winkler, M.},
  title   = {A model for heat generation by acoustic waves in piezoelectric materials: Global large-data solutions},
  journal = {Math. Models Meth. Appl. Sci.},
  volume = {35},
pages = {2465-2512},
year = {2025}
}

@unpublished{WinlkerMultiThermo,
    author  = {Winkler, M.},
    title = {Global solvability and stabilization in multi-dimensional small-strain nonlinear thermoviscoelasticity},
 year          = {2026},
  note = {arXiv:2602.05964},
}

@ARTICLE{LSU,
    author = {Ladyzenskaja, O.A. and Solonnikov, V.A. and Ural'ceva, N.N.},
    title = {Linear and quasi-linear equations of parabolic type},
    journal = {Amer. Math. Soc. Transl.}, 
    volume  = {23},
    year    = {1968}
}

@ARTICLE{Fricke,
  author  = {Fricke, T.},
  title   = {Local and global solvability in a viscous wave equation involving general temperature-dependence},
  journal = {Acta Appl. Math.}, 
  volume    =  {200},
  year = {2025}
}

@unpublished{Meyer,
  author  = {Meyer, F.},
  title   = {Global smooth solutions in a one-dimensional thermoviscoelastic model with temperature-dependent paramaters},
year ={2026},
  note    = {arXiv:2602.05621}
}

@ARTICLE{Winklerweak1d,
    author  = {Winkler, M.},
    title   = {Large-data solutions in one-dimensional thermoviscoelasticity involving temperature-dependent viscosities},
    journal = {Z. Angew. Math. Phys.},
  volume  = {25},
  year    = {2025},
  number = {108}
}

@ARTICLE{winkler2,
    author  = {Winkler, M.},
    title = {Large-data regular solutions in a one-dimensional thermoviscoelastic
evolution problem involving temperature-dependent viscosities},
    journal = {J. Evol. Equ.},
  volume  = {76},
  year    = {2025},
  number = {192}
}

@article{Bies1,
    author = {P.M. Bies and T. Cieślak},
    title = {Global-in-time regular unique solutions with positive temperature to one-dimensional thermoelasticity},
    journal = {SIAM J. Anal.},
    volume    = {55},
    year =  {2023}
}

@article{Bies2,
    author = {P.M. Bies and T. Cieślak},
    title = {Time-asymptoticsof of a heated string},
    journal = {Math.Ann.},
    volume    = {391},
    year =  {2025}
}

@article{mielke,
    author = {A. Mielke and T. Roubiček},
    title = {Thermoviscoelasticity in Kelvin-Voigt rheology at large strains},
    journal = {Arch. Ration. Mech. Anal.},
    volume = {238},
    year = {2020}
}

@article{owc,
    author = {S. Owczarek and K. Wielgos},
    title = {On a thermo-visco-elastic model with nonlinear damping forces
 and $L^1$ temperature data},
    journal = {Math. Methods Appl. Sci.},
    volume = {46},
    year = {2023}
}

@article{rossi,
    author = {R. Rossi and T. Roubiček},
    title = {Adhesive contact delaminating at mixed mode, its thermodynamics
 and analysis},
    journal =  {Interfaces Free Bound},
    volume = {15},
    year = {2013}
}

@article{rubi,
    author = {T. Roubiček},
    title = {Nonlinearly coupled thermo-visco-elasticity},
    journal = {NoDEA Nonlinear Differential Equa
tions Appl.},
    volume = {20},
    year = {2013}
}

@article{yosh,
    author = {S. Yoshikawa, I. Pawlow and W.M. Zajaczkowski},
    title = {Quasi-linear thermoelasticity system
 arising in shape memory materials},
    journal = { SIAM J. Math. Anal.},
    volume = {38},
    year = {2007}
}

@article{zimmer,
    author = {J. Zimmer},
    title = {Global existence of a nonlinear system in thermoviscoelasticity with nonconvex energy},
    journal = {J. Math. Anal. Appl.},
    volume = {292},
    year = {2004}
}

@book{henry1981geometric,
  author    = {Henry, D.},
  title     = {Geometric theory of semilinear parabolic equations},
  series    = {Lecture Notes in Mathematics},
  volume    = {840},
  publisher = {Springer},
  address   = {Berlin-Heidelberg-New York},
  year      = {1981}
}

@ARTICLE{winkler3,
    author = {L. Claes and M. Winkler},
    title = {Describing smooth small-data solutions to a quasilinear hyperbolic
parabolic system by $W^{1,p}$ energy analysis.},
    journal = {Nonlinear Anal. Real World Appl.},
    volume = {91},
pages = {104580},
year = {2026}
}
\end{document}